\documentclass{article}

    \PassOptionsToPackage{numbers, compress}{natbib}
    \usepackage[final]{neurips_2024}

\usepackage[utf8]{inputenc} % allow utf-8 input
\usepackage[T1]{fontenc}    % use 8-bit T1 fonts
\usepackage[hidelinks]{hyperref}
\usepackage{url}            % simple URL typesetting
\usepackage{booktabs}       % professional-quality tables
\usepackage{amsfonts}       % blackboard math symbols
\usepackage{nicefrac}       % compact symbols for 1/2, etc.
\usepackage{microtype}      % microtypography
\usepackage{xcolor}         % colors

\newcommand{\reals}{{\mbox{\bf R}}}

\newcommand{\symm}{{\mbox{\bf S}}}  % symmetric matrices
\newcommand{\BEAS}{\begin{eqnarray*}}
\newcommand{\EEAS}{\end{eqnarray*}}
\newcommand{\BEA}{\begin{eqnarray}}
\newcommand{\EEA}{\end{eqnarray}}
\newcommand{\BEQ}{\begin{equation}}
\newcommand{\EEQ}{\end{equation}}
\newcommand{\BIT}{\begin{itemize}}
\newcommand{\EIT}{\end{itemize}}

\newcommand{\range}{{\mathcal R}}

\newcommand{\Tr}{\mathop{\bf Tr}}
\newcommand{\diag}{\mathop{\bf diag}}

\newcommand{\prox}{\mathbf{prox}}

\newcommand{\argmin}{\mathop{\rm argmin}}

\newcommand{\dom}{\mathop{\bf dom}} % domain

\newcommand{\sign}{\mathop{\bf sign}}

\newcommand{\eg}{{\it e.g.}}
\newcommand{\ie}{{\it i.e.}}

\newcounter{algorithmctr}[section]
\renewcommand{\thealgorithmctr}{\thesection.\arabic{algorithmctr}}
   {\refstepcounter{algorithmctr}
   \begin{list}{}{%
       \setlength{\rightmargin}{0\linewidth}%
       \setlength{\leftmargin}{.05\linewidth}}%
       \rmfamily\small
       \item[]{\setlength{\parskip}{0ex}\hrulefill\par%
        \nopagebreak{\bfseries\textsf{Algorithm \thealgorithmctr~}}}}%
   {{\setlength{\parskip}{-1ex}\nopagebreak\par\hrulefill} 
   \end{list}}

   {\refstepcounter{algorithmctr}
   \begin{list}{}{%
       \setlength{\rightmargin}{0\linewidth}%
       \setlength{\leftmargin}{.05\linewidth}}%
       \rmfamily\small
       \item[]{\setlength{\parskip}{0ex}\hrulefill\par%
        \nopagebreak{\bfseries\textsf{Algorithm}}}}%
   {{\setlength{\parskip}{-1ex}\nopagebreak\par\hrulefill} 
   \end{list}}

\usepackage{circuitikz}

\usepackage{natbib}
\usepackage{amsthm,amssymb}
\newtheorem{theorem}{Theorem}[section]
\newtheorem{corollary}{Corollary}[theorem]
\newtheorem{lemma}[theorem]{Lemma}

\newcommand{\wiremulv}[4]{ % 1 : center_x, 2 : center_y, 3 : width_height, 4 : word
    \draw[line width=1] (#1 - #3, #2 - #3) to (#1 + #3, #2 + #3);
    \node at (#1+ #3 + 0.2, #2 ) {#4};
}
\newcommand{\circuitbox}[5]{  % 1 : center_x, 2 : center_y, 3 : width/2, 4 : height/2, 5 : word
\draw 
    (#1 - #3, #2 - #4) [fill=white] rectangle (#1 + #3, #2 + #4);
\node 
    at (#1, #2) {#5};
}

\newcommand{\wiremulseparate}[5]{ % 1 : center_x, 2 : center_y, 3 : width_height, 4 : distance of word, 5 : word
    \draw[line width=1] (#1 - #3, #2 - #3) to (#1 + #3, #2 + #3);
    \node at (#1, #2 + #4) {#5};
}

\newcommand{\wiremul}[4]{ % 1 : center_x, 2 : center_y, 3 : width_height, 4 : word
    \wiremulseparate{#1}{#2}{#3}{#3 + 0.15}{#4}
}

\tikzset{
    declare function={% in case of CVS which switches the arguments of atan2
        atan3(\a,\b)=ifthenelse(atan2(0,1)==90, atan2(\a,\b), atan2(\b,\a));
    },
    kinky cross radius/.initial=+.125cm,
    @kinky cross/.initial=+,
    kinky crosses/.is choice,
    kinky crosses/left/.style={@kinky cross=-},
    kinky crosses/right/.style={@kinky cross=+},
    kinky cross/.style args={(#1)--(#2)}{
        to path={
            let \p{@kc@}=($(\tikztotarget)-(\tikztostart)$),
            \n{@kc@}={atan3(\p{@kc@})+180} in
            -- ($(intersection of \tikztostart--{\tikztotarget} and #1--#2)!%
            \pgfkeysvalueof{/tikz/kinky cross radius}!(\tikztostart)$)
            arc [ radius     =\pgfkeysvalueof{/tikz/kinky cross radius},
            start angle=\n{@kc@},
            delta angle=\pgfkeysvalueof{/tikz/@kinky cross}180 ]
            -- (\tikztotarget)
        }
    }
}

\usepackage{subcaption}
\usepackage{arydshln}
\usepackage{mathtools}
\usepackage{blkarray}
\usepackage{float}

\newcommand{\mC}{\mathcal{C}}
\newcommand{\mL}{\mathcal{L}}
\newcommand{\mR}{\mathcal{R}}
\newcommand{\Rm}{r}
\newcommand{\pp}{p}

\newcommand{\rcm}{\hspace{2mm}}
\newcommand{\inner}[2]{ \left\langle #1 ,  #2 \right\rangle }
\newcommand{\norm}[1]{\left\| #1 \right\|} 
\newcommand{\bmat}[1]{\begin{bmatrix} #1 \end{bmatrix}}
\newcommand{\mat}[1]{\begin{matrix} #1 \end{matrix}}

\newcommand{\rB}{\tilde{B}}
\newcommand{\set}[1]{\left\{#1\right\}}
\newcommand{\pr}[1]{\left(#1\right)}
\newcommand{\ve}{\mathbf{e}}
\newcommand{\vn}{\mathbf{n}}
\newcommand{\vr}{\mathbf{r}}
\newcommand{\row}{\text{Row}}

\newcommand{\Bdom}{\mathcal{J}}
\newcommand{\Bran}{\mathcal{K}}

\usepackage{listings}

\definecolor{codegreen}{rgb}{0,0.6,0}
\definecolor{codegray}{rgb}{0.5,0.5,0.5}
\definecolor{codepurple}{rgb}{0.58,0,0.82}
\definecolor{backcolour}{rgb}{0.95,0.95,0.92}

\lstdefinestyle{mystyle}{
    backgroundcolor=\color{backcolour},   
    commentstyle=\color{codegreen},
    keywordstyle=\color{magenta},
    numberstyle=\tiny\color{codegray},
    stringstyle=\color{codepurple},
    basicstyle=\ttfamily\small,
    breakatwhitespace=false,         
    breaklines=true,                 
    captionpos=b,                    
    keepspaces=true,                 
    numbers=left,                    
    numbersep=5pt,                  
    showspaces=false,                
    showstringspaces=false,
    showtabs=false,                  
    tabsize=2
}

\definecolor{seagreen}{rgb}{0.18, 0.55, 0.34}
\definecolor{mediumviolet-red}{rgb}{0.78, 0.08, 0.52}
\definecolor{khaki}{rgb}{0.94, 0.9, 0.55}

\lstdefinelanguage{mypython}
{
	keywords=[1]{from, import, assert, not, print},
	keywordstyle=[1]{\color{mediumviolet-red}},
	keywords=[2]{surecr, torch, cp, lo, pl},
	keywordstyle=[2]{\color{seagreen}},
	numbers=none,
	upquote=true,
	showstringspaces=false,
	basicstyle=\ttfamily,
	columns=fullflexible,
	keepspaces=true,
	emph={True,False,as,def,return,float,class,match,switch,len},
	emphstyle={\color{seagreen}},
	frame=trBL,
	belowskip=1em,
	aboveskip=1em,
	captionpos=b
}

 \usepackage[bb=boondox]{mathalfa}
\usepackage{xparse}
\DeclareFontFamily{U}{ntxmia}{}
\DeclareFontShape{U}{ntxmia}{m}{it}{<-> ntxmia }{}
\DeclareFontShape{U}{ntxmia}{b}{it}{<-> ntxbmia }{}
\DeclareSymbolFont{lettersA}{U}{ntxmia}{m}{it}
\SetSymbolFont{lettersA}{bold}{U}{ntxmia}{b}{it}
\ExplSyntaxOn
\NewDocumentCommand{\varmathbb}{m}
 {
  \tl_map_inline:nn { #1 }
   {
    \use:c { varbb##1 }
   }
 }
\tl_map_inline:nn { ABCDEFGHIJKLMNOPQRSTUVWXYZ }
 {
  \exp_args:Nc \DeclareMathSymbol{varbb#1}{\mathord}{lettersA}{\int_eval:n { `#1+67 }}
 }
\exp_args:Nc \DeclareMathSymbol{varbbk}{\mathord}{lettersA}{169}
\ExplSyntaxOff
\makeatletter
\DeclareFontFamily{U}{tipa}{}
\DeclareFontShape{U}{tipa}{m}{n}{<->tipa10}{}
\newcommand{\arc@char}{{\usefont{U}{tipa}{m}{n}\symbol{62}}}%

\newcommand{\arc}[1]{\mathpalette\arc@arc{#1}}

\newcommand{\arc@arc}[2]{%
  \sbox0{$\m@th#1#2$}%
  \vbox{
    \hbox{\resizebox{\wd0}{\height}{\arc@char}}
    \nointerlineskip
    \box0
  }%
}
\makeatother

\newcommand{\I}{\varmathbb{I}}

\newcommand{\StaticInterconnect}{
    \begin{circuitikz}[scale=.7, every node/.style={scale=0.8}][american voltages]
        \pgfmathsetmacro{\w}{7}
        \pgfmathsetmacro{\h}{4}
        \pgfmathsetmacro{\boxwidth}{3}
        \pgfmathsetmacro{\midspace}{\w - \boxwidth}
        \pgfmathsetmacro{\wirespace}{\h/6}
        \foreach \i in {1,2,3,4,5} {
            \draw ( \boxwidth/2, \i * \wirespace ) to [short, i=$ $] (\boxwidth/2 + \midspace/3, \i * \wirespace );   
        }        
        \circuitbox{0}{\h/2}{\boxwidth/2}{\h/2}{}
        \pgfmathsetmacro{\netleft}{\boxwidth/10}
        
        \draw
            (\boxwidth/2  , 5 * \wirespace ) 
            to (-\boxwidth/10, 5 * \wirespace ) 
            to (-\boxwidth/10, 3 * \wirespace ) 
            to (\boxwidth/2 , 3 * \wirespace );
        \node at ( -\boxwidth/10 - \netleft - 0.1 , 4 * \wirespace  ) {$N_1$};
        
        \draw
            (\boxwidth/2  , 4 * \wirespace ) 
            to (\netleft, 4 * \wirespace ) 
            to (\netleft, 2 * \wirespace ) 
            to (\boxwidth/2 , 2 * \wirespace );
        \node at (0 - 0.1, 2.25 * \wirespace  ) {$N_2$};
        \draw
            (\boxwidth/2  , \wirespace ) 
            to (\netleft, \wirespace ) ;
        \node at (0 - 0.1, \wirespace  ) {$N_3$};
    
        \node at (0, \h + \h/10 ) {  \textbf{Static interconnect} };
    \end{circuitikz}
}

\newcommand{\DynamicInterconnectNew}{
    \ctikzset{bipoles/length=.8cm}
\begin{circuitikz}[scale=.6, every node/.style={scale=0.7}][american voltages]
    \pgfmathsetmacro{\w}{8}
    \pgfmathsetmacro{\h}{6}
    \pgfmathsetmacro{\boxwidth}{\w/2}
    \pgfmathsetmacro{\midspace}{\w - \boxwidth}
    \pgfmathsetmacro{\wirespace}{\h/6}
    \pgfmathsetmacro{\wirestartwidth}{\boxwidth/5}
    \pgfmathsetmacro{\componentwidth}{3}
    \node at (-0.5 * \componentwidth + 0.3, 8 * \h/7 + 0.5 ) { \large \textbf{Dynamic interconnect} };
    \circuitbox{-0.5 * \componentwidth + 0.3}{4.5 * \h/7}{2.4}{3.5 * \h/7}{ }
    \draw
        ( - \componentwidth, 5 * \wirespace )
        to [short, *-] ( - \componentwidth, 6 * \wirespace )
        to [short, i<=$ $, R, l^=$R_1$] ( - \componentwidth/2, 6 * \wirespace )
        to [short] ( - \componentwidth/2, 6 * \wirespace );        
    \node at (-\componentwidth -0.2, 5 * \wirespace )  {6}; 
    \draw
        ( \wirestartwidth , 6 * \wirespace ) 
        to ( - \componentwidth/2 , 6 * \wirespace ) 
        to [short, *-] ( - \componentwidth/2 , 5 * \wirespace )  
        to [short, i<=$ $, L, l_=$L_1$] ( - \componentwidth, 5 * \wirespace ) 
        to ( - \componentwidth, 4 * \wirespace );
    \node at (- \componentwidth/2, 6 * \wirespace + 0.25 )  {1};
    \draw
        ( \wirestartwidth , 4 * \wirespace )
        to [short, -*] ( - \componentwidth/2, 4 * \wirespace )
        to [short, i<=$ $, L, l_=$L_2$]  (  - \componentwidth, 4 * \wirespace ) ;  
    \node at (-\componentwidth/2, 4 * \wirespace - 0.25 )  {3};
    \draw
        ( \wirestartwidth , 5 * \wirespace )
        to ( 0 , 5 * \wirespace )
        to [short, *-, i=$ $] ( 0 , 4 * \wirespace )
        to ( 0 , 3 * \wirespace );
    \draw
        ( \wirestartwidth , 3 * \wirespace ) 
        to [short, *-*,  i=$ $] ( 0, 3 * \wirespace )
        to [short, C, l=$C_1$, label distance=-1mm, i=$ $] ( -\componentwidth/2 , 3 * \wirespace ) 
        to ( -\componentwidth/2 , 2 * \wirespace );
    \draw
        (0, 2 * \wirespace ) 
        to (\wirestartwidth, 2 * \wirespace ) 
        to (-\componentwidth/2, 2 * \wirespace )
        to [short, *-, i=$ $, C=$C_2$, label distance=-1mm]  (-\componentwidth, 2 * \wirespace )
        to [short, *-, i=$ $] (-\componentwidth, \wirespace/2 ) ;  
    \node at (-\componentwidth/2, 2 * \wirespace - 0.25 )  {5};
    \node at (\wirestartwidth, 3 * \wirespace - 0.25 )  {4};
    \node at (0 - 0.2, 5 * \wirespace  )  {2};
    
    \node at ( 0, 3 * \wirespace-0.25 )  {7}; 
    \node at (-\componentwidth -0.2, 2 * \wirespace ) {8};   
    \draw (-\componentwidth, \wirespace ) node[ground, style={scale=2}] {};  
    
    \foreach \i in {2,3,4,5,6} { 
        \pgfmathtruncatemacro{\compi}{7-\i}
        \draw ( \wirestartwidth, \i * \wirespace ) 
        to (-0.5 * \componentwidth + 0.3 + 2.4, \i * \wirespace)
        to [short, i=$y_\compi$] (\boxwidth/7 + 2.2/5 * \midspace, \i * \wirespace );
    }
\end{circuitikz}
}

\newcommand{\StaticInterconnectF}{
    \begin{circuitikz}[scale=.7, every node/.style={scale=0.8}][american voltages]
        \pgfmathsetmacro{\w}{6}
        \pgfmathsetmacro{\h}{4}
        \pgfmathsetmacro{\boxwidth}{3}
        \pgfmathsetmacro{\midspace}{\w - \boxwidth}
        \pgfmathsetmacro{\wirespace}{\h/6}

        \circuitbox{\w}{\h/2}{\boxwidth/2}{\h/2}{\huge $\partial f$}
        \circuitbox{0}{\h/2}{\boxwidth/2}{\h/2}{}

        \ctikzset{bipoles/length=.8cm}
        \foreach \i in {1,2,3,4,5} { 
            \pgfmathtruncatemacro{\compi}{6-\i}
            \draw ( \boxwidth/2 , \i * \wirespace ) to [short, -*, i=$y_\compi^\star$] ( \boxwidth/2 + \midspace, \i * \wirespace );
        }

        \draw (\w, 0) node[ground, style={scale=2}] {};        
        
        \pgfmathsetmacro{\netleft}{\boxwidth/10}
        \draw
            (\boxwidth/2  , 5 * \wirespace ) 
            to (-\boxwidth/10, 5 * \wirespace ) 
            to (-\boxwidth/10, 3 * \wirespace ) 
            to (\boxwidth/2 , 3 * \wirespace );
        \draw
            (\boxwidth/2  , 4 * \wirespace ) 
            to (\netleft, 4 * \wirespace ) 
            to (\netleft, 2 * \wirespace ) 
            to (\boxwidth/2 , 2 * \wirespace );
        \draw
            (\boxwidth/2  , \wirespace ) 
            to (\netleft, \wirespace ) ;
        \foreach \i in {1,2,3,4,5} {
            \node at ( \boxwidth/2 + \midspace + 0.3 , 6 * \wirespace - \i * \wirespace ) {  $x^{\star}_{\i}$ };
        }
    
        \node at (0, \h + \h/10 ) {  \textbf{Static interconnect} };
    \end{circuitikz}
}

\newcommand{\DynamicInterconnectFnew}{
\ctikzset{bipoles/length=.8cm}
\begin{circuitikz}[scale=.6, every node/.style={scale=0.7}][american voltages]
    \pgfmathsetmacro{\w}{8}
    \pgfmathsetmacro{\h}{6}
    \pgfmathsetmacro{\boxwidth}{\w/2}
    \pgfmathsetmacro{\midspace}{\w - \boxwidth}
    \pgfmathsetmacro{\wirespace}{\h/6}
    \pgfmathsetmacro{\wirestartwidth}{\boxwidth/5}
    \pgfmathsetmacro{\componentwidth}{3}
    \node at (-0.5 * \componentwidth + 0.3, 7 * \wirespace + 0.5 ) { \large \textbf{Dynamic interconnect} };
    \circuitbox{-0.5 * \componentwidth + 0.3}{4 * \wirespace}{2.4}{3 * \wirespace}{ }
    \draw
        ( - \componentwidth, 5 * \wirespace )
        to [short, *-] ( - \componentwidth, 6 * \wirespace )
        to [short, R] ( - \componentwidth/2, 6 * \wirespace )
        to [short] ( - \componentwidth/2, 6 * \wirespace );        
    \draw
        ( \wirestartwidth , 6 * \wirespace ) 
        to ( - \componentwidth/2 , 6 * \wirespace ) 
        to [short, *-] ( - \componentwidth/2 , 5 * \wirespace )  
        to [short, L] ( - \componentwidth, 5 * \wirespace ) 
        to ( - \componentwidth, 4 * \wirespace );
    \draw
        ( \wirestartwidth , 4 * \wirespace )
        to [short, -*] ( - \componentwidth/2, 4 * \wirespace )
        to [short, L]  (  - \componentwidth, 4 * \wirespace ) ; 
    \draw
        ( \wirestartwidth , 5 * \wirespace )
        to ( 0 , 5 * \wirespace )
        to [short, *-] ( 0 , 4 * \wirespace )
        to ( 0 , 3 * \wirespace );
    \draw
        ( \wirestartwidth , 3 * \wirespace ) 
        to [short, *-*] ( 0, 3 * \wirespace )
        to [short, C] ( -\componentwidth/2 , 3 * \wirespace ) 
        to ( -\componentwidth/2 , 2 * \wirespace );
    \draw
        (0, 2 * \wirespace ) 
        to (\wirestartwidth, 2 * \wirespace ) 
        to (-\componentwidth/2, 2 * \wirespace )
        to [short, C, *-]  (-\componentwidth, 2 * \wirespace )
        to [short, *-] (-\componentwidth, \wirespace/2 ) ;
    \draw (-\componentwidth, \wirespace ) node[ground, style={scale=2}] {};

    \circuitbox{\w}{4 * \wirespace}{2.3}{3 * \wirespace}{ \huge $\partial f$}

    \foreach \i in {2,3,4,5,6} {
        \pgfmathtruncatemacro{\compi}{7-\i}
        \draw ( \wirestartwidth , \i * \wirespace ) to [short, -*, i=$y_\compi(t)$] ( \w - 2.3, \i * \wirespace );
    }   
    
    \draw (\w, \wirespace ) node[ground, style={scale=2}] {};

    \node at (\boxwidth/2 + \midspace + 0.3, 6 * \wirespace )  {$x_1(t)$};    
    \node at (\boxwidth/2 + \midspace + 0.3, 5 * \wirespace )  {$x_2(t)$};    
    \node at (\boxwidth/2 + \midspace + 0.3, 4 * \wirespace )  {$x_3(t)$};    
    \node at (\boxwidth/2 + \midspace + 0.3, 3 * \wirespace )  {$x_4(t)$}; 
    \node at (\boxwidth/2 + \midspace + 0.3, 2 * \wirespace )  {$x_5(t)$}; 
    
\end{circuitikz}
}

\newcommand{\CircuitPgExtraMain}{
\begin{circuitikz}[scale=.6, every node/.style={scale=0.7}][american voltages]
    \pgfmathsetmacro{\w}{9}
    \pgfmathsetmacro{\h}{3}
    \pgfmathsetmacro{\boxwidth}{2}
    \pgfmathsetmacro{\midspace}{\w - \boxwidth}
    \pgfmathsetmacro{\wcompsize}{\w/3}
    \pgfmathsetmacro{\hcompsize}{3*\h/4}
    \pgfmathsetmacro{\hmargin}{2*\h/5}

    \large

    \foreach \index in {l, j} {
        \IfStrEq{\index}{j}
            { \pgfmathsetmacro{\sign}{-1} }
            { \pgfmathsetmacro{\sign}{1} }
    \pgfmathsetmacro{\heigt}{ (1+\sign) * ( \hmargin + 1/2 * \hcompsize ) }

    \wiremulseparate{-3 * \wcompsize + \boxwidth}{\heigt}{0.2}{0.4}{\large $m$}
    
    \circuitbox{\wcompsize}{\heigt}{\boxwidth/2}{\boxwidth/2}{\Large $\nabla h_\index$}
    \circuitbox{-3 * \wcompsize }{\heigt}{\boxwidth/2}{\boxwidth/2}{\Large $\partial f_\index$}

    \draw ( -2 * \wcompsize, \heigt ) node[yshift=\sign * 10] {$\tilde{x}_\index$}
        to ( -3 * \wcompsize + \boxwidth/2, \heigt );

    \draw ( -2 * \wcompsize, \heigt ) 
        to [short,  R, l_=$R$, *-] ( - \wcompsize, \heigt ) node[yshift=\sign *10] {$e_\index$} 
        to [short,  R, l_=$-R$, *-] ( 0, \heigt ) node[yshift=\sign *10] {$x_\index$} 
        to [short, i=$ $, *-] ( \wcompsize - \boxwidth/2, \heigt );
    }      

    \draw ( 0, \hmargin )
    to [short]( \wcompsize/4, \hmargin )
    to [short,  R, l_=$R_{jl}$] ( \wcompsize/4, \hmargin + \hcompsize)
    to ( 0, \hmargin + \hcompsize);
        
    \draw ( 0, 0 )
    to [short, i=$ $] ( 0, \hmargin )
    to [short] ( -\wcompsize/4, \hmargin )
    to [short,  L, l=$L_{jl}$] ( -\wcompsize/4, \hmargin + \hcompsize)
    to ( 0, \hmargin + \hcompsize)
    to [short, -*] ( 0, 2 * \hmargin + \hcompsize );

\end{circuitikz}
}

\newcommand{\CircuitDADMMMain}{
\begin{circuitikz}[scale=.6, every node/.style={scale=0.7}][american voltages]
        \pgfmathsetmacro{\w}{6}
        \pgfmathsetmacro{\h}{4}
        \pgfmathsetmacro{\boxheight}{2}
        \pgfmathsetmacro{\boxwidth}{2}
        \pgfmathsetmacro{\wirespace}{1}
        \pgfmathsetmacro{\rlspace}{4}

        \foreach \i in {-1,0,1,2} { 
            \draw
            ( 0 + \i * \rlspace,  0) 
            to ( 2 * \wirespace + \i * \rlspace, 0) 
            to [short, R, l=$R$] ( 2 * \wirespace + \i * \rlspace, 3*\h/5) 
            to ( 0 + \i * \rlspace, 3*\h/5);
        };
        \foreach \i in {-1,2} {
            \draw ( 0 + \i * \rlspace ,  3*\h/5) to [short, L, l=$L$] ( 0 + \i * \rlspace,  0); 
        };
        \draw ( 0 ,  3*\h/5) to [short, L, l=$L$ , i<=${i_{\mL}}_{jl}$] ( 0 ,  0); 
        \draw ( 0 + \rlspace ,  3*\h/5) to [short, L, l=$L$, i<=${i_{\mL}}_{lj}$] ( 0 + \rlspace,  0); 
        
        \foreach \i in {-1,0,1,2} { 
            \draw ( \i * \rlspace + \wirespace, 0) to ( \i * \rlspace + \wirespace, -\h/4); 
        };
        \draw (\wirespace + \rlspace, -\h/4) 
        to [short, -*] ( \wirespace + \rlspace/2, -\h/4) node[below, yshift=-2] {  $e_{jl}$ }
        to (\wirespace, -\h/4);
        \draw (-\rlspace + \wirespace, -\h/4) to (-\rlspace - \wirespace, -\h/4);
        \draw (2 *\rlspace + \wirespace, -\h/4) to (2 *\rlspace + 3* \wirespace, -\h/4);
        
        \draw
            ( - \boxheight/2, \h  + \boxheight/2 )
            to [short, -*, i<=$ $] ( - \boxheight/2 , 3*\h/5+ \boxheight/2) ;
        \draw
            [short, *-]  ( \w + \boxheight/2  , 3*\h/5 + \boxheight/2 ) 
            to [short, i=$ $] ( \w + \boxheight/2, \h  + \boxheight/2 );

        \foreach \i in {-1,1} { 
            \draw
            ( \wirespace + \i * \rlspace, 3*\h/5)
            to ( \wirespace + \i * \rlspace, 3*\h/5+ \boxheight/2)
            to ( \wirespace + \i * \rlspace +  \rlspace, 3*\h/5+ \boxheight/2)
            to ( \wirespace + \i * \rlspace +  \rlspace, 3*\h/5);
            \node at ( \wirespace + \i * \rlspace + \rlspace/2 , 3*\h/10 ) {\Large $\cdots$};
        };

        \wiremulv{ -\boxheight/2}{7*\h/8 + \boxheight/2}{0.2}{\large $m$}
        \wiremulv{\w + \boxheight/2}{7*\h/8 + \boxheight/2}{0.2}{\large $m$}
            
        \node at ( -\boxheight/2, 3*\h/5+ \boxheight/2 -0.3 ) {  $x_j$ };
        \node at ( \w + \boxheight/2, 3*\h/5+ \boxheight/2 -0.3 ) {  $x_l$ };

        \circuitbox{- \boxheight/2}{\h + \boxheight}{\boxwidth/2}{\boxheight/2}{\Large $\partial f_j$}
        \circuitbox{\w + \boxheight/2}{\h + \boxheight}{\boxwidth/2}{\boxheight/2}{\Large $\partial f_l$}
        
    \end{circuitikz}
}

\newcommand{\CircuitExampleMain}{
\begin{circuitikz}[scale=.6, every node/.style={scale=0.7}][american voltages]
    \ctikzset{label/align = straight}
    \pgfmathsetmacro{\w}{11}
    \pgfmathsetmacro{\h}{3}
    \pgfmathsetmacro{\boxheight}{2}
    \pgfmathsetmacro{\boxwidth}{2}
    \pgfmathsetmacro{\compwidth}{3}
    \pgfmathsetmacro{\midspace}{\w - \boxwidth}
    \pgfmathsetmacro{\wirespace}{\h/4}
    
    \circuitbox{\w}{0}{\boxwidth/2}{\boxheight/2}{\Large $\partial f$}
    \draw
        ( 0 , 0)
        to [short, C, l=$C_2$, *-*, i<=$ $] ( \compwidth , 0 ) %node[above] {$e_2$}
        to [short, -*, R, l=$R_2$, i<=$ $] ( 2 * \compwidth, 0 ) %node[above] {$e_1$}
        to [short, C, l=$C_1$, i=$ $]( 3 * \compwidth, 0 )
        to [short, -*] ( \boxwidth/2 + \midspace , 0 ) node[above, xshift=-6pt] {$x$};
    \draw
        ( \compwidth , 0 )
        to [short, -*, R, l_=$R_1$, i<=$ $] ( 1.5 * \compwidth, -{sqrt(3)} * \compwidth/2 ) %node[below, xshift=7pt] {$e_3$} 
        to [short, -*, R, l_=$R_3$, i=$ $] ( 2 * \compwidth, 0 );
    \draw
        ( 0 , 0) %node[above] {$e_4$} 
        to ( 0 , -4 * \wirespace ) node[ground] {};
    \draw
        ( 1.5 * \compwidth, -{sqrt(3)} * \compwidth/2 ) to 
        ( 1.5 * \compwidth, -4 * \wirespace ) node[ground] {};
    \wiremulseparate{ \boxwidth/2 + \midspace - \midspace/8}{0}{0.2}{0.35}{$m$}
\end{circuitikz}
}

\newcommand{\CircuitDADMMCMain}{
\begin{center} 
    \begin{circuitikz}[scale=.6, every node/.style={scale=0.7}][american voltages]
        \pgfmathsetmacro{\w}{6}
        \pgfmathsetmacro{\h}{4}
        \pgfmathsetmacro{\boxheight}{2}
        \pgfmathsetmacro{\boxwidth}{2}
        \pgfmathsetmacro{\wirespace}{1}
        \pgfmathsetmacro{\rlspace}{4}

        \foreach \i in {-1} { 
            \draw
            ( \wirespace + \i * \rlspace, 3*\h/5+ \boxheight/2) to 
            ( \wirespace + \i * \rlspace +  \rlspace, 3*\h/5+ \boxheight/2) to 
            ( \wirespace + \i * \rlspace +  \rlspace, 3*\h/5);
            \node at ( \wirespace + \i * \rlspace + \rlspace/2 , 3*\h/10 ) {\Large $\cdots$};
        };

        \foreach \i in {0,1} { 
            \draw ( 0 + \i * \rlspace,  0) 
            to ( 2 * \wirespace + \i * \rlspace, 0) 
            to [short, R, l=$R$] ( 2 * \wirespace + \i * \rlspace, 3*\h/5) 
            to ( 0 + \i * \rlspace, 3*\h/5);
        };
        
        \draw ( 0 ,  3*\h/5) to [short, L, l=$L$ , i<=${i_{\mL}}_{45}$] ( 0 ,  0); 
        \draw ( 0 + \rlspace ,  3*\h/5) to [short, L, l=$L$, i<=${i_{\mL}}_{54}$] ( 0 + \rlspace,  0); 
        
        \foreach \i in {0,1} { 
            \draw ( \i * \rlspace + \wirespace, 0) to ( \i * \rlspace + \wirespace, -\h/4); 
        };
        \draw (\wirespace, -\h/4) to (\wirespace + \rlspace, -\h/4);
        \draw ( \wirespace + \rlspace/2, -\h/4 ) node[above] {  $e_{45}$ } 
            to [short, *-] ( \wirespace + \rlspace/2, -3*\h/8 + 0.1 )
            to [short, C, l=$C$] ( \wirespace + \rlspace/2, -3*\h/4 ) node[ground] {};
        
        \draw
            ( - \boxheight/2, \h  + \boxheight/2 )
            to [short, -*, i<=$ $] ( - \boxheight/2 , 3*\h/5+ \boxheight/2) ;
        \draw
            [short, *-]  ( \wirespace + \rlspace  , 3*\h/5 + \boxheight/2 ) node[left] {$x_5$}
            to [short, i=$ $] ( \wirespace + \rlspace, \h  + \boxheight/2 );

        \draw
            ( \wirespace + \rlspace, 3*\h/5)
            to ( \wirespace +  \rlspace, 3*\h/5+ \boxheight/2);

        \wiremulv{ -\boxheight/2}{7*\h/8 + \boxheight/2}{0.2}{$m$}
        \wiremulv{ \wirespace + \rlspace }{7*\h/8 + \boxheight/2}{0.2}{$m$}
            
        \node at ( -\boxheight/2, 3*\h/5+ \boxheight/2 -0.4 ) {  $x_4$ };
        \circuitbox{- \boxheight/2}{\h + \boxheight}{\boxwidth/2}{\boxheight/2}{\Large $\partial f_4$}
        \circuitbox{ \wirespace + \rlspace }{\h + \boxheight}{\boxwidth/2}{\boxheight/2}{\Large $\partial f_5$}
        
    \end{circuitikz}
    \end{center}
}

\newcommand{\mTerminal}{
\begin{circuitikz}[scale=.7, every node/.style={scale=0.85}][american voltages] %[scale=.6, every node/.style={scale=0.9}][american voltages]
    \pgfmathsetmacro{\w}{4}
    \pgfmathsetmacro{\h}{2}
    \draw ( 0 , 0 )  to  (\w,0);
    \wiremulseparate{\w/2}{0}{\w/14}{\w/8}{\Large$m$}
\end{circuitikz}
}

\newcommand{\MoerauEnvelope}{
\begin{center}
    \begin{subfigure}{.5\textwidth}
        \centering
        \begin{circuitikz}[scale=.6, every node/.style={scale=0.9}][american voltages]
            \pgfmathsetmacro{\w}{8}
            \pgfmathsetmacro{\h}{2.2}
            \pgfmathsetmacro{\boxwidth}{2}
            \pgfmathsetmacro{\midspace}{\w - \boxwidth}
            \pgfmathsetmacro{\wirespace}{\h/4}
            \circuitbox{\w}{\h/2}{\boxwidth/2}{\boxwidth/2}{\Large $\partial f$}
            \draw ( \boxwidth/2 , 2 * \wirespace )  node[above] {$x$}
                to [short,  R, l^=$R$, *-*, i_>=$ $] ( \boxwidth/2 + \midspace/2, 2 * \wirespace ) node[above] {$\tilde{x}$}
                to [short, i=$ $] ( \boxwidth/2 + \midspace, 2 * \wirespace );
            \wiremulseparate{\boxwidth/2 + 7*\midspace/8}{2 * \wirespace}{0.2}{0.4}{$m$}
        \end{circuitikz}
    \end{subfigure}%
    \begin{subfigure}{.05\textwidth}
        \centering
        \raisebox{0.4cm}{\LARGE $\Leftrightarrow$}
    \end{subfigure}%
    \begin{subfigure}{.5\textwidth}
        \centering
        \begin{circuitikz}[scale=.6, every node/.style={scale=0.9}][american voltages]
            \pgfmathsetmacro{\w}{8}
            \pgfmathsetmacro{\h}{2.2}
            \pgfmathsetmacro{\boxwidth}{2}
            \pgfmathsetmacro{\midspace}{\w - \boxwidth}
            \pgfmathsetmacro{\wirespace}{\h/4}
            \circuitbox{\w}{\h/2}{\boxwidth/2}{\boxwidth/2}{\Large $\nabla^{R} \!f$}
            \draw ( \boxwidth/2 , 2 * \wirespace ) node[above] {$x$} to [short, *-, i=$ $] ( \boxwidth/2 + \midspace, 2 * \wirespace );
            \wiremulseparate{\boxwidth/2 + 7*\midspace/8}{2 * \wirespace}{0.2}{0.4}{$m$}
        \end{circuitikz}
    \end{subfigure}
    \end{center}
}

\newcommand{\MoreauEnvelopeVertical}{
    \begin{center} 
    \begin{circuitikz}[scale=.7, every node/.style={scale=0.85}][american voltages]
        \pgfmathsetmacro{\w}{8}
        \pgfmathsetmacro{\h}{2.2}
        \pgfmathsetmacro{\boxwidth}{2}
        \pgfmathsetmacro{\midspace}{\w - \boxwidth}
        \pgfmathsetmacro{\wirespace}{\h/4}
        \circuitbox{\w}{\h/2}{\boxwidth/2}{\boxwidth/2}{\Large $\partial f$}
        
        \foreach \i in {2} { 
            \draw ( \boxwidth/2 , \i * \wirespace ) 
            to [short,  R, l^=$R$, *-*, i_>=$ $] ( \boxwidth/2 + \midspace/2, \i * \wirespace )
            to [short, i=$ $] ( \boxwidth/2 + \midspace, \i * \wirespace );
        }

        \node at ( \boxwidth/2 , 2 * \wirespace + 0.3  ) {  $x$ };
        \node at ( \boxwidth/2 + \midspace/2, 2 * \wirespace + 0.3  ) {  $\tilde{x}$ };

        \wiremul{\boxwidth/2 + 7*\midspace/8}{2 * \wirespace}{0.2}{$m$}
        
    \end{circuitikz}
    \end{center}

    \begin{center} 
    \begin{circuitikz}[scale=.7, every node/.style={scale=0.85}][american voltages]
        \pgfmathsetmacro{\w}{8}
        \pgfmathsetmacro{\h}{2.2}
        \pgfmathsetmacro{\boxwidth}{2}
        \pgfmathsetmacro{\midspace}{\w - \boxwidth}
        \pgfmathsetmacro{\wirespace}{\h/4}
        
        \circuitbox{\w}{\h/2}{\boxwidth/2}{\boxwidth/2}{\Large $\nabla^{R} f$}
        \foreach \i in {2} { 
            \draw ( \boxwidth/2 , \i * \wirespace ) to [short, *-, i=$ $] ( \boxwidth/2 + \midspace, \i * \wirespace );
        }

        \node at ( \boxwidth/2 , 2 * \wirespace + 0.3  ) {  $x$ };
        \wiremul{\boxwidth/2 + 7*\midspace/8}{2 * \wirespace}{0.2}{$m$}
        
    \end{circuitikz}
    \end{center}
}

\newcommand{\PreMoreauEnvelope}{
    \begin{center} 
    \begin{circuitikz}[scale=.7, every node/.style={scale=0.85}][american voltages]
        \pgfmathsetmacro{\w}{8}
        \pgfmathsetmacro{\h}{2.2}
        \pgfmathsetmacro{\boxwidth}{2}
        \pgfmathsetmacro{\midspace}{\w - \boxwidth}
        \pgfmathsetmacro{\wirespace}{\h/4}
        
        \circuitbox{\w}{\h/2}{\boxwidth/2}{\boxwidth/2}{\Large $\partial \tilde{f}$}
        \foreach \i in {2} { 
            \draw ( \boxwidth/2 , \i * \wirespace ) to [short, *-, i=$ $] ( \boxwidth/2 + \midspace, \i * \wirespace );
        }

        \node at ( \boxwidth/2 , 2 * \wirespace + 0.3  ) {  $x$ };
        \wiremul{\boxwidth/2 + 7*\midspace/8}{2 * \wirespace}{0.2}{$m$}
        
    \end{circuitikz}
    \end{center}

    \begin{center} 
    \begin{circuitikz}[scale=.7, every node/.style={scale=0.85}][american voltages]
        \pgfmathsetmacro{\w}{8}
        \pgfmathsetmacro{\h}{2.2}
        \pgfmathsetmacro{\boxwidth}{2}
        \pgfmathsetmacro{\midspace}{\w - \boxwidth}
        \pgfmathsetmacro{\wirespace}{\h/4}
        \circuitbox{\w}{\h/2}{\boxwidth/2}{\boxwidth/2}{\Large $\nabla f$}
        \foreach \i in {2} { 
            \draw ( \boxwidth/2 , \i * \wirespace ) 
            to [short,  R, l^=$-R$, *-*, i_>=$ $] ( \boxwidth/2 + \midspace/2, \i * \wirespace )
            to [short, i=$ $] ( \boxwidth/2 + \midspace, \i * \wirespace );
        }
        
        \node at ( \boxwidth/2 , 2 * \wirespace + 0.3  ) {  $x$ };
        \node at ( \boxwidth/2 + \midspace/2, 2 * \wirespace + 0.3  ) {  $\tilde{x}$ };

        \wiremul{\boxwidth/2 + 7*\midspace/8}{2 * \wirespace}{0.2}{$m$}
    \end{circuitikz}
    \end{center}
}

\newcommand{\CircuitGD}{
    \begin{center} 
    \begin{circuitikz}[scale=.7, every node/.style={scale=0.85}][american voltages]
        \pgfmathsetmacro{\w}{6}
        \pgfmathsetmacro{\h}{3}
        \pgfmathsetmacro{\boxheight}{2}
        \pgfmathsetmacro{\boxwidth}{2}
        \pgfmathsetmacro{\midspace}{\w - \boxwidth}
        \pgfmathsetmacro{\wirespace}{\h/4}
        
        \circuitbox{\w}{\h/2}{\boxwidth/2}{\boxheight/2}{\Large $\partial f$}
        \draw
            ( -\boxwidth/4 , \wirespace ) node[ground] {}
            to [short, -*] ( -\boxwidth/4 , 2 * \wirespace );
        \foreach \i in {2} { 
            \draw 
            ( \boxwidth/2 + \midspace/3, \i * \wirespace )
            to [short, C, l_=$C$, *-, i=$ $] ( \boxwidth/4 , \i * \wirespace )
            to (-\boxwidth/4, 2 * \wirespace);
            \draw ( \boxwidth/2 + \midspace/3, \i * \wirespace ) to [short, i=$ $] ( \boxwidth/2 + \midspace, \i * \wirespace );
        }
        \node at ( - \boxwidth/4 , 2 * \wirespace + 0.3  ) {  $e$ };
        \node at ( \boxwidth/2 + \midspace/3, 2 * \wirespace + 0.3  ) {  $x$ };
        
        \wiremul{\boxwidth/2 + 7*\midspace/8}{2 * \wirespace}{0.2}{$m$}
    \end{circuitikz}
    \end{center}
}

\newcommand{\CitcuitNesterov}{
    \begin{center} 
    \begin{circuitikz}[scale=.7, every node/.style={scale=0.85}][american voltages]
        \pgfmathsetmacro{\w}{8}
        \pgfmathsetmacro{\h}{6}
        \pgfmathsetmacro{\boxheight}{4}
        \pgfmathsetmacro{\boxwidth}{2}
        \pgfmathsetmacro{\midspace}{\w - \boxwidth}
        \pgfmathsetmacro{\wirespace}{\h/4}
        
        \circuitbox{\w+\boxwidth}{\h/2}{\boxwidth/2}{\boxwidth/2}{\Large $\nabla f$}

        \draw
            ( 0 , 1.5 * \wirespace ) node[ground] {}
            to [short, -*] ( 0 , 2 * \wirespace);

        \pgfmathsetmacro{\endofC}{\boxwidth/2 + \midspace/4}

        \draw ( \endofC, 2 * \wirespace )
        to [short, C, l_=$C$, -, i=$ $] ( \boxwidth/6 , 2 * \wirespace ) 
        to (0, 2 * \wirespace);

        \draw  ( \endofC, 2 * \wirespace ) 
        to [short, *-] ( \endofC, 2 * \wirespace + \wirespace/2 ) 
        to [short, R, l=$R$, i=$ $] ( \boxwidth/2 + 2*\midspace/3, 2 * \wirespace + \wirespace/2 )
        to ( \boxwidth/2 + 2*\midspace/3, 2 * \wirespace );
        
        \draw  ( \endofC, 2 * \wirespace ) 
        to ( \endofC, 2 * \wirespace - \wirespace/2 ) 
        to [short, L, l_=$L$, i=$ $] ( \boxwidth/2 + 2*\midspace/3, 2 * \wirespace - \wirespace/2 )
        to ( \boxwidth/2 + 2*\midspace/3, 2 * \wirespace );
        
        \draw ( \boxwidth/2 + 2*\midspace/3, 2 * \wirespace )
        to [short, R, *-, l=$-R$, i=$ $] ( \w, 2 * \wirespace )
        to [short, -*] ( \w +\boxwidth/2, 2 * \wirespace );
        
        \wiremul{\w +\boxwidth/8}{2 * \wirespace}{0.2}{$m$}

        \node at ( \endofC - 0.4, 2 * \wirespace - 0.25  ) {$e$ };
        \node at ( \boxwidth/2 + 2*\midspace/3 + 0.4 , 2 * \wirespace - 0.25  ) {  $x^{+}$ };
        \node at (  \w +\boxwidth/2 - 0.3 , 2 * \wirespace - 0.25 ) {  $x$ };
        
    \end{circuitikz}
    \end{center}
}

\newcommand{\CircuitPPM}{
    \begin{center} 
    \begin{circuitikz}[scale=.7, every node/.style={scale=0.85}][american voltages]
        \pgfmathsetmacro{\w}{6}
        \pgfmathsetmacro{\h}{3}
        \pgfmathsetmacro{\boxheight}{2}
        \pgfmathsetmacro{\boxwidth}{2}
        \pgfmathsetmacro{\midspace}{\w - \boxwidth}
        \pgfmathsetmacro{\wirespace}{\h/4}
        
        \circuitbox{\w}{\h/2}{\boxwidth/2}{\boxheight/2}{\Large $\partial f$}
        \draw
            ( -\boxwidth , 1 * \wirespace ) node[ground] {}
            to [short, -*] ( -\boxwidth , 2 * \wirespace);
        \draw 
            ( -\boxwidth, 2 * \wirespace)
            to [short, C, l=$C$, -*, i<=$ $] ( \boxwidth/2 , 2 * \wirespace ); 
        \draw ( \boxwidth/2, 2 * \wirespace ) 
        to [short, R, l=$R$, i=$ $] ( \boxwidth/2 + 3*\midspace/4, 2 * \wirespace )
        to [short] ( \boxwidth/2 + \midspace , 2 * \wirespace );

        \wiremul{\boxwidth/2 + \midspace - \midspace/8}{2 * \wirespace}{0.2}{$m$}
    
        \node at ( \boxwidth/2, 2 * \wirespace + 0.3 ) {  $x$ };
    \end{circuitikz}
    \end{center}
}

\newcommand{\CircuitPPMMoreau}{
    \begin{center} 
    \begin{circuitikz}[scale=.7, every node/.style={scale=0.85}][american voltages]
        \pgfmathsetmacro{\w}{6}
        \pgfmathsetmacro{\h}{3}
        \pgfmathsetmacro{\boxheight}{2}
        \pgfmathsetmacro{\boxwidth}{2}
        \pgfmathsetmacro{\midspace}{\w - \boxwidth}
        \pgfmathsetmacro{\wirespace}{\h/4}
        
        \circuitbox{\w}{\h/2}{\boxwidth/2}{\boxheight/2}{\Large $\nabla ^{R}f$}
        \draw
            ( -\boxwidth , 1 * \wirespace ) node[ground] {}
            to [short, -*] ( -\boxwidth , 2 * \wirespace);
        \draw 
            ( \boxwidth/2 , 2 * \wirespace )
            to [short, C, l_=$C$, *-, i=$ $] ( -\boxwidth , 2 * \wirespace ) 
            to ( -\boxwidth, 2 * \wirespace);
        \draw ( \boxwidth/2, 2 * \wirespace ) 
        to [short, i=$ $] ( \boxwidth/2 + 3*\midspace/4, 2 * \wirespace )
        to [short] ( \boxwidth/2 + \midspace , 2 * \wirespace );

        \wiremul{\boxwidth/2 + \midspace - \midspace/4}{2 * \wirespace}{0.2}{$m$}
        \node at ( \boxwidth/2, 2 * \wirespace + 0.3 ) {  $x$ };
    
    \end{circuitikz}
    \end{center}
}

\newcommand{\CircuitProxGrad}{
\begin{circuitikz}[scale=.7, every node/.style={scale=0.85}][american voltages]
    \pgfmathsetmacro{\w}{5}
    \pgfmathsetmacro{\h}{3}
    \pgfmathsetmacro{\boxwidth}{2}
    \draw
        (0,0) node[ground] {}
        (0,0) to [C, l=$C$] (0,\h/4)
        to [short, i<=$ $] (0,\h/2);
    \draw
        (0,\h/2)
        to [R, *-*, l=$-R$] (-\w/2, \h/2)
        to [short, i=$ $] (-\w/2, 9*\h/8);
    \draw
        (0,\h/2)
        to (\w/2, \h/2)
        to [short, i=$ $] (\w/2, 9*\h/8);

    \wiremulv{-\w/2}{\h}{0.2}{$m$}
    \wiremulv{\w/2}{\h}{0.2}{$m$}
    
    \circuitbox{-\w/2}{9*\h/8 + \boxwidth/2}{\boxwidth/2}{\boxwidth/2}{\Large $\nabla f$}
    \circuitbox{\w/2}{9*\h/8 + \boxwidth/2}{\boxwidth/2}{\boxwidth/2}{\Large $\nabla^R g$}
    \node 
        at (-\w/2 , \h/2 - 0.3 ) {$x$};
    \node 
        at (0, \h/2 + 0.3) {$e$};
\end{circuitikz}
}

\newcommand{\CircuitPrimalDecomposition}{
    \begin{center} 
    \begin{circuitikz}[scale=.7, every node/.style={scale=0.85}][american voltages]
        \pgfmathsetmacro{\w}{6}
        \pgfmathsetmacro{\h}{4}
        \pgfmathsetmacro{\boxheight}{2}
        \pgfmathsetmacro{\boxwidth}{2}
        \pgfmathsetmacro{\wirespace}{1}
        \pgfmathsetmacro{\middleheight}{\h*1/2}
        
        \draw 
            ( 0, \h  )
            to [short, -*, i<=$ $] ( 0 , \middleheight) 
            to [short, -*] ( \w , \middleheight) 
            to [short, i=$ $] ( \w, \h  );
        \draw 
            (\w/2,0) node[ground] {}
            (\w/2, \middleheight)
            to [short, *-, C, l=$C$, i=$ $] (\w/2, -0.2);

        \wiremulv{0}{7*\h/8}{0.2}{$m$}
        \wiremulv{\w}{7*\h/8}{0.2}{$m$}
            
        \node at ( 0, \middleheight - 0.3 ) {  $x_1$ };
        \node at ( \w, \middleheight - 0.3 ) {  $x_N$ };
        \node at ( \w/2, \middleheight + 0.3 ) {  $e$ };

        \circuitbox{0}{\h + \boxheight/2}{\boxwidth/2}{\boxheight/2}{\Large $\partial f_1$}
        \node at ( \w/2, \h + \boxheight/2 ){\Large $\cdots$};
        \circuitbox{\w}{\h + \boxheight/2}{\boxwidth/2}{\boxheight/2}{\Large $\partial f_N$}
        
    \end{circuitikz}
    \end{center}
}

\newcommand{\CitcuitDualDecomposition}{
    \begin{center} 
    \begin{circuitikz}[scale=.7, every node/.style={scale=0.85}][american voltages]
        \pgfmathsetmacro{\w}{6}
        \pgfmathsetmacro{\h}{4}
        \pgfmathsetmacro{\boxheight}{2}
        \pgfmathsetmacro{\boxwidth}{2}
        \pgfmathsetmacro{\wirespace}{1}
        \pgfmathsetmacro{\middleheight}{\h/5}
        \pgfmathsetmacro{\boxcenterheight}{\h + \boxheight/2}

        \circuitbox{0}{\boxcenterheight}{\boxwidth/2}{\boxheight/2}{\Large $\partial f_1$}
        \node at ( \w/2, \boxcenterheight ){\Large $\cdots$};
        \circuitbox{\w}{\boxcenterheight}{\boxwidth/2}{\boxheight/2}{\Large $\partial f_N$}
        
        \draw 
            ( 0, \h  )
            to ( 0, 4.5 *\h/5  )
            to [short, *-] ( 0, 3.5 *\h/5  )
            to [short, -, i<=$ $, L, l=$L$] ( 0 , \middleheight) 
            to [short, -*] ( \w/2 , \middleheight) 
            to [short, -] ( \w , \middleheight) ;
        \draw 
            ( \w, \h  )
            to ( \w, 4.5  * \h/5  )
            to [short, *-]  ( \w, 3.5  * \h/5  )
            to [short, i<=$ $, L, l=$L$] ( \w , \middleheight);

        \node at ( -0.4, 4.5 *\h/5 ) {  $x_1$ };
        \node at ( \w - 0.4, 4.5 *\h/5 ) {  $x_N$ };
        \node at ( \w/2, \middleheight - 0.3 ) {  $e$ };

        \wiremulv{0}{3*\h/4}{0.2}{$m$}
        \wiremulv{\w}{3*\h/4}{0.2}{$m$}
        
    \end{circuitikz}
    \end{center}
}

\newcommand{\CircuitProximalDecomposition}{
    \begin{center} 
    \begin{circuitikz}[scale=.7, every node/.style={scale=0.85}][american voltages]
        \pgfmathsetmacro{\w}{6}
        \pgfmathsetmacro{\h}{4}
        \pgfmathsetmacro{\boxheight}{2}
        \pgfmathsetmacro{\boxwidth}{2}
        \pgfmathsetmacro{\wirespace}{1}
        
        \draw 
            ( 0, \h  )
            to [short, -*, i<=$ $] ( 0 , 3*\h/5) 
            to ( -\wirespace ,  3*\h/5) 
            to [short, R, l_=$R$] ( -\wirespace ,  0) 
            to [short, -*] ( \w/2 ,  0) node[below, yshift=-2] {$e$}
            to ( \w +\wirespace, 0) 
            to  [short, R, l_=$R$]  ( \w +\wirespace, 3*\h/5) 
            to [short, -*]  ( \w , 3*\h/5) 
            to [short, i=$ $] ( \w, \h  );
        \draw
            ( -\wirespace ,  3*\h/5)
            to ( \wirespace ,  3*\h/5)
            to [short, L, l_=$L$] ( \wirespace , 0);
        \draw
            ( \w + \wirespace ,  3*\h/5)
            to ( \w - \wirespace ,  3*\h/5)
            to [short, L, l=$L$] ( \w - \wirespace , 0);

        \wiremulv{0}{7*\h/8}{0.2}{$m$}
        \wiremulv{\w}{7*\h/8}{0.2}{$m$}
            
        \node at ( 0, 2*\h/4  ) {  $x_1$ };
        \node at ( \w, 2*\h/4  ) {  $x_N$ };

        \circuitbox{0}{\h + \boxheight/2}{\boxwidth/2}{\boxheight/2}{\Large $\partial f_1$}
        \node at ( \w/2, \h + \boxheight/2 ){\Large $\cdots$};
        \circuitbox{\w}{\h + \boxheight/2}{\boxwidth/2}{\boxheight/2}{\Large $\partial f_N$}
        
    \end{circuitikz}
    \end{center}
}

\newcommand{\CircuitDRS}{
\begin{circuitikz}[scale=.7, every node/.style={scale=0.85}][american voltages]
    \pgfmathsetmacro{\w}{4}
    \pgfmathsetmacro{\h}{4}
    \pgfmathsetmacro{\boxheight}{2}
    \pgfmathsetmacro{\boxwidth}{2}
    \pgfmathsetmacro{\midspace}{\w - \boxwidth}
    \pgfmathsetmacro{\wirespace}{\h/4}

    \draw 
        ( 0, \h/2 )
        to [short, *-*, L, l=$L$, i<=$ $] ( \w , \h/2 ) ;
    \draw 
        ( 0, \h/2 )
        to ( 0, \h/6 ) 
        to [short, R, l=$R$, i<=$ $] ( \w, \h/6)
        to ( \w, \h/2 );
    \foreach \i in {0, 1} { 
        \draw ( \i * \w, \h/2 ) to [short, i=$ $] (\i * \w, \h );
        \wiremulv{ \i * \w }{ 7 * \h/8}{0.2}{$m$}
    }

    \circuitbox{0}{\h+\boxwidth/2}{\boxwidth/2}{\boxheight/2}{\Large $\partial g$}
    \circuitbox{\w}{\h+\boxwidth/2}{\boxwidth/2}{\boxheight/2}{\Large $\partial f$}

    \node at ( 0 - 0.4, \h/2  ) {  $x_1$ };
    \node at ( \w + 0.4, \h/2  ) {  $x_2$ };

\end{circuitikz}
}

\newcommand{\CircuitDYS}{
    \begin{center} 
    \begin{circuitikz}[scale=.7, every node/.style={scale=0.85}][american voltages]
        \pgfmathsetmacro{\w}{10}
        \pgfmathsetmacro{\h}{4}
        \pgfmathsetmacro{\boxheight}{2}
        \pgfmathsetmacro{\boxwidth}{2}
        
        \draw 
            ( 0, \h  )
            to [short, -*, i<=$ $] ( 0 , 2*\h/5) 
            to ( 0 , 0) 
            to [short, R, l_=$R$, i=$ $] ( \w , 0) 
            to [short, -*]  ( \w , 2*\h/5) 
            to [short, i=$ $] ( \w, \h  );
        \draw
            ( 0 , 2*\h/5)
            to [short, L, l=$L$, -*, i<=$ $] ( \w/2 , 2*\h/5) 
            to [short, i=$ $] ( \w/2, \h );
        \draw
            ( \w/2 , 2*\h/5)
            to [short, R, -*, l=$-S$, i=$ $] ( 3*\w/4 , 2*\h/5)
            to [short, R, l=$S$] ( \w , 2*\h/5);

        \wiremulv{0}{7*\h/8}{0.2}{$m$}
        \wiremulv{\w/2}{7*\h/8}{0.2}{$m$}
        \wiremulv{\w}{7*\h/8}{0.2}{$m$}
            
        \node at ( -0.4, 2*\h/5  ) {  $x_3$ };
        \node at ( \w/2 , 2*\h/5 - 0.3  ) {  $x_2$ };
        \node at ( 3*\w/4 , 2*\h/5 - 0.3  ) {  $e$ };
        \node at ( \w + 0.4, 2*\h/5  ) {  $x_1$ };
        \circuitbox{0}{\h + \boxheight/2}{\boxwidth/2}{\boxheight/2}{\Large $\partial g$}
        \circuitbox{\w/2}{\h + \boxheight/2}{\boxwidth/2}{\boxheight/2}{\Large $\nabla h$}
        \circuitbox{\w}{\h + \boxheight/2}{\boxwidth/2}{\boxheight/2}{\Large $\partial f$}
        
    \end{circuitikz}
    \end{center}
}

\newcommand{\CircuitDGD}{
    \begin{center} 
    \begin{circuitikz}[scale=.7, every node/.style={scale=0.85}][american voltages]
        \pgfmathsetmacro{\w}{8}
        \pgfmathsetmacro{\h}{4}
        \pgfmathsetmacro{\boxheight}{2}
        \pgfmathsetmacro{\boxwidth}{2}

        \foreach \index in {l,j} {
            \IfStrEq{\index}{l}
            { \pgfmathsetmacro{\i}{1} }
            { \pgfmathsetmacro{\i}{2} }

            \circuitbox{\i * \w/2 - \w/2}{ \h + \boxheight/2}{\boxwidth/2}{\boxheight/2}{\Large $\nabla f_\index$}
            
            \draw 
            ( \i * \w/2 - \w/2 , \h  ) 
            to [short,  -*, i<=$ $] ( \i * \w/2 - \w/2 , \h/2) node[below, xshift=-6pt] {$x_\index$}
            to [short, C, l=$C$, i=$ $] 
            ( \i * \w/2 - \w/2 , 0 ) node[ground] {};
            
            \wiremulv{\i * \w/2 - \w/2}{ \h - 0.4 }{0.2}{$m$}
        };
        
        \draw
            ( 0 , \h/2)
            to [short, R, l=$R_{jl}$, -*, i<=$ $] ( \w/2 ,\h/2) ;
    \end{circuitikz}
    \end{center}
}

\newcommand{\CircuitDiffusion}{
    \begin{center} 
    \begin{circuitikz}[scale=.7, every node/.style={scale=0.85}][american voltages]
        \pgfmathsetmacro{\w}{8}
        \pgfmathsetmacro{\h}{4}
        \pgfmathsetmacro{\boxheight}{2}
        \pgfmathsetmacro{\boxwidth}{2}

        \foreach \index in {l,j} {
            \IfStrEq{\index}{l}
            { \pgfmathsetmacro{\i}{1} }
            { \pgfmathsetmacro{\i}{2} }

            \circuitbox{\i * \w/2 - \w/2}{ 4*\h/3 + \boxheight/2}{\boxwidth/2}{\boxheight/2}{\Large $\nabla f_\index$}
            
            \draw 
            ( \i * \w/2 - \w/2 , 4*\h/3  ) node[below, xshift=-6pt] {$x_\index$}
            to [short,*-] ( \i * \w/2 - \w/2 , 7*\h/6  )
            to [short, R, l=$-R$, -*, i<=$ $] ( \i * \w/2 - \w/2 , \h/2) node[below, xshift=-6pt] {$e_\index$}
            to [short, C, l=$C$, i=$ $] 
            ( \i * \w/2 - \w/2 , 0 ) node[ground] {};
            
            \wiremulv{\i * \w/2 - \w/2}{ 5*\h/4 - 0.2 }{0.2}{$m$}
        };
        
        \draw
            ( 0 , \h/2)
            to [short, R, l=$R_{jl}$, -*, i<=$ $] ( \w/2 ,\h/2) ;
    \end{circuitikz}
    \end{center}
}

\newcommand{\CircuitDADMM}{
\begin{center} 
    \begin{circuitikz}[scale=.7, every node/.style={scale=0.85}][american voltages]
        \pgfmathsetmacro{\w}{6}
        \pgfmathsetmacro{\h}{4}
        \pgfmathsetmacro{\boxheight}{2}
        \pgfmathsetmacro{\boxwidth}{2}
        \pgfmathsetmacro{\wirespace}{1}
        \pgfmathsetmacro{\rlspace}{4}

        \foreach \i in {-1,0,1,2} { 
            \draw
            ( 0 + \i * \rlspace,  0) 
            to ( 2 * \wirespace + \i * \rlspace, 0) 
            to [short, R, l=$R$] ( 2 * \wirespace + \i * \rlspace, 3*\h/5) 
            to ( 0 + \i * \rlspace, 3*\h/5);
        };
        \foreach \i in {-1,2} {
            \draw ( 0 + \i * \rlspace ,  3*\h/5) to [short, L, l=$L$] ( 0 + \i * \rlspace,  0); 
        };
        \draw ( 0 ,  3*\h/5) to [short, L, l=$L$ , i<=${i_{\mL}}_{jl}$] ( 0 ,  0); 
        \draw ( 0 + \rlspace ,  3*\h/5) to [short, L, l=$L$, i<=${i_{\mL}}_{lj}$] ( 0 + \rlspace,  0); 
        
        \foreach \i in {-1,0,1,2} { 
            \draw ( \i * \rlspace + \wirespace, 0) to ( \i * \rlspace + \wirespace, -\h/4); 
        };
        \draw (\wirespace, -\h/4) to (\wirespace + \rlspace, -\h/4);
        \node at ( \wirespace + \rlspace/2, -\h/4 - 0.4 ) {  $e_{jl}$ };
        \draw ( \wirespace + \rlspace/2, -\h/4 ) to [short, *-] ( \wirespace + \rlspace/2, -\h/4 );
        \draw (-\rlspace + \wirespace, -\h/4) to (-\rlspace - \wirespace, -\h/4);
        \draw (2 *\rlspace + \wirespace, -\h/4) to (2 *\rlspace + 3* \wirespace, -\h/4);
        
        \draw
            ( - \boxheight/2, \h  + \boxheight/2 )
            to [short, -*, i<=$ $] ( - \boxheight/2 , 3*\h/5+ \boxheight/2) ;
        \draw
            [short, *-]  ( \w + \boxheight/2  , 3*\h/5 + \boxheight/2 ) 
            to [short, i=$ $] ( \w + \boxheight/2, \h  + \boxheight/2 );

        \foreach \i in {-1,1} { 
            \draw
            ( \wirespace + \i * \rlspace, 3*\h/5)
            to ( \wirespace + \i * \rlspace, 3*\h/5+ \boxheight/2)
            to ( \wirespace + \i * \rlspace +  \rlspace, 3*\h/5+ \boxheight/2)
            to ( \wirespace + \i * \rlspace +  \rlspace, 3*\h/5);
            \node at ( \wirespace + \i * \rlspace + \rlspace/2 , 3*\h/10 ) {\Large $\cdots$};
        };

        \wiremulv{ -\boxheight/2}{7*\h/8 + \boxheight/2}{0.2}{$m$}
        \wiremulv{\w + \boxheight/2}{7*\h/8 + \boxheight/2}{0.2}{$m$}
            
        \node at ( -\boxheight/2, 3*\h/5+ \boxheight/2 -0.4 ) {  $x_j$ };
        \node at ( \w + \boxheight/2, 3*\h/5+ \boxheight/2 -0.4 ) {  $x_l$ };

        \circuitbox{- \boxheight/2}{\h + \boxheight}{\boxwidth/2}{\boxheight/2}{\Large $\partial f_j$}
        \circuitbox{\w + \boxheight/2}{\h + \boxheight}{\boxwidth/2}{\boxheight/2}{\Large $\partial f_l$}
        
    \end{circuitikz}
    \end{center}
}

\newcommand{\CircuitPgExtra}{
\begin{circuitikz}[scale=.7, every node/.style={scale=0.85}][american voltages]
\pgfmathsetmacro{\w}{9}
    \pgfmathsetmacro{\h}{3}
    \pgfmathsetmacro{\boxwidth}{2}
    \pgfmathsetmacro{\midspace}{\w - \boxwidth}
    \pgfmathsetmacro{\wcompsize}{\w/3}
    \pgfmathsetmacro{\hcompsize}{3*\h/4}
    \pgfmathsetmacro{\hmargin}{2*\h/5}

    \large

    \foreach \index in {l, j} {
        \IfStrEq{\index}{j}
            { \pgfmathsetmacro{\sign}{-1} }
            { \pgfmathsetmacro{\sign}{1} }
    \pgfmathsetmacro{\heigt}{ (1+\sign) * ( \hmargin + 1/2 * \hcompsize ) }

    \wiremulseparate{-3 * \wcompsize + \boxwidth}{\heigt}{0.2}{0.4}{\large $m$}
    
    \circuitbox{\wcompsize}{\heigt}{\boxwidth/2}{\boxwidth/2}{\Large $\nabla h_\index$}
    \circuitbox{-3 * \wcompsize }{\heigt}{\boxwidth/2}{\boxwidth/2}{\Large $\partial f_\index$}

    \draw ( -2 * \wcompsize, \heigt ) node[yshift=\sign * 10] {$\tilde{x}_\index$}
        to ( -3 * \wcompsize + \boxwidth/2, \heigt );

    \draw ( -2 * \wcompsize, \heigt ) 
        to [short,  R, l_=$R$, *-] ( - \wcompsize, \heigt ) node[yshift=\sign *10] {$e_\index$} 
        to [short,  R, l_=$-R$, *-] ( 0, \heigt ) node[yshift=\sign *10] {$x_\index$} 
        to [short, i=$ $, *-] ( \wcompsize - \boxwidth/2, \heigt );
    }      

    \draw ( 0, \hmargin )
    to [short]( \wcompsize/4, \hmargin )
    to [short,  R, l_=$R_{jl}$] ( \wcompsize/4, \hmargin + \hcompsize)
    to ( 0, \hmargin + \hcompsize);
        
    \draw ( 0, 0 )
    to [short, i=$ $] ( 0, \hmargin )
    to [short] ( -\wcompsize/4, \hmargin )
    to [short,  L, l=$L_{jl}$] ( -\wcompsize/4, \hmargin + \hcompsize)
    to ( 0, \hmargin + \hcompsize)
    to [short, -*] ( 0, 2 * \hmargin + \hcompsize );

\end{circuitikz}
}

\newcommand{\CircuitPgExtraSimple}{
    \begin{circuitikz}[scale=.7, every node/.style={scale=0.85}][american voltages]
        \pgfmathsetmacro{\w}{9}
        \pgfmathsetmacro{\h}{3}
        \pgfmathsetmacro{\boxwidth}{2}
        \pgfmathsetmacro{\midspace}{\w - \boxwidth}
        \pgfmathsetmacro{\wcompsize}{\w/3}
        \pgfmathsetmacro{\hcompsize}{3*\h/4}
        \pgfmathsetmacro{\hmargin}{2*\h/5}
        \large

        \foreach \index in {l, j} {
            \IfStrEq{\index}{j}
                { \pgfmathsetmacro{\sign}{-1} }
                { \pgfmathsetmacro{\sign}{1} }
                
        \pgfmathsetmacro{\heigt}{ (1+\sign) * ( \hmargin + 1/2 * \hcompsize ) }
        
        \circuitbox{\wcompsize}{\heigt}{\boxwidth/2}{\boxwidth/2}{\Large $\nabla h_\index$}
        \circuitbox{ - \wcompsize }{\heigt}{\boxwidth/2}{\boxwidth/2}{\Large $\partial f_\index$}

        \draw ( - \wcompsize + \boxwidth/2 , \heigt )
            to ( 0, \heigt ) node[yshift=\sign*10] {$x_\index$}
            to [short, i=$ $, *-] ( \wcompsize - \boxwidth/2, \heigt );

        }
        
        \draw ( 0, 0 )
            to [short, i=$ $] ( 0, \hmargin )
            to [short] ( -\wcompsize/4, \hmargin )
            to [short,  L, l=$L_{jl}$] ( -\wcompsize/4, \hmargin + \hcompsize)
            to ( 0, \hmargin + \hcompsize)
            to [short, -*] ( 0, 2 * \hmargin + \hcompsize );
        \draw ( 0, \hmargin )
            to [short]( \wcompsize/4, \hmargin )
            to [short,  R, l_=$R_{jl}$] ( \wcompsize/4, \hmargin + \hcompsize)
            to ( 0, \hmargin + \hcompsize);

    \end{circuitikz}
    \end{center}
}

\newcommand{\CircuitExample}{
\begin{circuitikz}[scale=.6, every node/.style={scale=0.7}][american voltages]
\ctikzset{label/align = straight}
    \pgfmathsetmacro{\w}{11}
    \pgfmathsetmacro{\h}{3}
    \pgfmathsetmacro{\boxheight}{2}
    \pgfmathsetmacro{\boxwidth}{2}
    \pgfmathsetmacro{\compwidth}{3}
    \pgfmathsetmacro{\midspace}{\w - \boxwidth}
    \pgfmathsetmacro{\wirespace}{\h/4}
    
    \circuitbox{\w}{0}{\boxwidth/2}{\boxheight/2}{\Large $\partial f$}
    \draw
        ( 0 , 0)
        to [short, C, l=$C_2$, *-*, i<=$ $] ( \compwidth , 0 ) node[above] {$e_2$}
        to [short, -*, R, l=$R_2$, i<=$ $] ( 2 * \compwidth, 0 ) node[above] {$e_1$}
        to [short, C, l=$C_1$, i=$ $]( 3 * \compwidth, 0 )
        to [short, -*] ( \boxwidth/2 + \midspace , 0 ) node[above, xshift=-6pt] {$x$};
    \draw
        ( \compwidth , 0 )
        to [short, -*, R, l_=$R_1$, i<=$ $] ( 1.5 * \compwidth, -{sqrt(3)} * \compwidth/2 ) node[below, xshift=7pt] {$e_3$} 
        to [short, -*, R, l_=$R_3$, i=$ $] ( 2 * \compwidth, 0 );
    \draw
        ( 0 , 0) node[above] {$e_4$} 
        to ( 0 , -4 * \wirespace ) node[ground] {};
    \draw
        ( 1.5 * \compwidth, -{sqrt(3)} * \compwidth/2 ) to 
        ( 1.5 * \compwidth, -4 * \wirespace ) node[ground] {};
    \wiremulseparate{ \boxwidth/2 + \midspace - \midspace/8}{0}{0.2}{0.4}{$m$}
\end{circuitikz}
}

\newcommand{\CircuitDADMMC}{
\begin{center} 
    \begin{circuitikz}[scale=.7, every node/.style={scale=0.85}][american voltages]
        \pgfmathsetmacro{\w}{6}
        \pgfmathsetmacro{\h}{4}
        \pgfmathsetmacro{\boxheight}{2}
        \pgfmathsetmacro{\boxwidth}{2}
        \pgfmathsetmacro{\wirespace}{1}
        \pgfmathsetmacro{\rlspace}{4}

        \foreach \i in {-1,0,1} { 
            \draw
            ( 0 + \i * \rlspace,  0) 
            to ( 2 * \wirespace + \i * \rlspace, 0) 
            to [short, R, l=$R$] ( 2 * \wirespace + \i * \rlspace, 3*\h/5) 
            to ( 0 + \i * \rlspace, 3*\h/5);
        };
        \foreach \i in {-1} {
            \draw ( 0 + \i * \rlspace ,  3*\h/5) to [short, L, l=$L$] ( 0 + \i * \rlspace,  0); 
        };
        \draw ( 0 ,  3*\h/5) to [short, L, l=$L$ , i<=${i_{\mL}}_{45}$] ( 0 ,  0); 
        \draw ( 0 + \rlspace ,  3*\h/5) to [short, L, l=$L$, i<=${i_{\mL}}_{54}$] ( 0 + \rlspace,  0); 
        
        \foreach \i in {-1,0,1} { 
            \draw ( \i * \rlspace + \wirespace, 0) to ( \i * \rlspace + \wirespace, -\h/4); 
        };
        \draw (\wirespace, -\h/4) to (\wirespace + \rlspace, -\h/4);
        \draw (-\rlspace + \wirespace, -\h/4) to (-\rlspace - \wirespace, -\h/4);
        \draw ( \wirespace + \rlspace/2, -\h/4 ) node[above] {  $e_{45}$ } 
            to [short, *-] ( \wirespace + \rlspace/2, -3*\h/8 + 0.1 )
            to [short, C, l=$C$] ( \wirespace + \rlspace/2, -3*\h/4 ) node[ground] {};
        
        \draw
            ( - \boxheight/2, \h  + \boxheight/2 )
            to [short, -*, i<=$ $] ( - \boxheight/2 , 3*\h/5+ \boxheight/2) ;
        \draw
            [short, *-]  ( \wirespace + \rlspace  , 3*\h/5 + \boxheight/2 ) node[left] {$x_5$}
            to [short, i=$ $] ( \wirespace + \rlspace, \h  + \boxheight/2 );

        \foreach \i in {-1} { 
            \draw
            ( \wirespace + \i * \rlspace, 3*\h/5)
            to ( \wirespace + \i * \rlspace, 3*\h/5+ \boxheight/2)
            to ( \wirespace + \i * \rlspace +  \rlspace, 3*\h/5+ \boxheight/2)
            to ( \wirespace + \i * \rlspace +  \rlspace, 3*\h/5);
            \node at ( \wirespace + \i * \rlspace + \rlspace/2 , 3*\h/10 ) {\Large $\cdots$};
        };

        \draw
            ( \wirespace + \rlspace, 3*\h/5)
            to ( \wirespace +  \rlspace, 3*\h/5+ \boxheight/2);

        \wiremulv{ -\boxheight/2}{7*\h/8 + \boxheight/2}{0.2}{$m$}
        \wiremulv{ \wirespace + \rlspace }{7*\h/8 + \boxheight/2}{0.2}{$m$}
            
        \node at ( -\boxheight/2, 3*\h/5+ \boxheight/2 -0.4 ) {  $x_4$ };
        \circuitbox{- \boxheight/2}{\h + \boxheight}{\boxwidth/2}{\boxheight/2}{\Large $\partial f_4$}
        \circuitbox{ \wirespace + \rlspace }{\h + \boxheight}{\boxwidth/2}{\boxheight/2}{\Large $\partial f_5$}
        
    \end{circuitikz}
    \end{center}
}

\newcommand{\PgExtraParallelC}{
    \begin{circuitikz}[scale=.7, every node/.style={scale=0.85}][american voltages]
        \pgfmathsetmacro{\w}{9}
        \pgfmathsetmacro{\h}{3}
        \pgfmathsetmacro{\boxwidth}{2}
        \pgfmathsetmacro{\midspace}{\w - \boxwidth}
        \pgfmathsetmacro{\wcompsize}{\w/3}
        \pgfmathsetmacro{\hcompsize}{3*\h/4}
        \pgfmathsetmacro{\hmargin}{2*\h/5}

        \large

        \foreach \index in {l, j} {
            \IfStrEq{\index}{j}
                { \pgfmathsetmacro{\sign}{-1} }
                { \pgfmathsetmacro{\sign}{1} }
        \pgfmathsetmacro{\height}{ (1+\sign) * ( \hmargin + 1/2 * \hcompsize ) }

        \wiremulseparate{-3 * \wcompsize + \boxwidth}{\height}{0.2}{0.4}{\large $m$}
        \circuitbox{\wcompsize}{\height}{\boxwidth/2}{\boxwidth/2}{\Large $\nabla h_\index$}
        \circuitbox{-3 * \wcompsize + \boxwidth/4}{\height}{\boxwidth/2}{\boxwidth/2}{\Large $\partial f_\index$}

        \draw ( -2 * \wcompsize, \height ) node[yshift=\sign * 10] {$\tilde{x}_\index$}
            to ( -3 * \wcompsize + \boxwidth/2 + \boxwidth/4, \height );

        \draw ( -2 * \wcompsize, \height ) 
            to [short,  R, l_=$R$, *-] ( - \wcompsize, \height ) node[yshift=\sign *10] {$e_\index$} 
            to [short,  R, l_=$-R$, *-] ( 0, \height ) node[yshift=\sign *10] {$x_\index$} 
            to [short, i=$ $, *-] ( \wcompsize - \boxwidth/2, \height );
        }      

        \draw ( 0, \hmargin )
        to [short]( 0, \hmargin )
        to [short,  R, l=$R_{jl}$] ( 0, \hmargin + \hcompsize)
        to ( 0, \hmargin + \hcompsize);

        \draw ( 0, \hmargin )
        to [short]( \wcompsize/2, \hmargin )
        to [short,  C, l_=$C_{jl}$] ( \wcompsize/2, \hmargin + \hcompsize)
        to ( 0, \hmargin + \hcompsize);
            
        \draw ( 0, 0 )
        to [short, i=$ $] ( 0, \hmargin )
        to [short] ( -\wcompsize/2, \hmargin )
        to [short,  L, l=$L_{jl}$] ( -\wcompsize/2, \hmargin + \hcompsize)
        to ( 0, \hmargin + \hcompsize)
        to [short, -*] ( 0, 2 * \hmargin + \hcompsize );

    \end{circuitikz}
}

\newcommand{\CircuitPgExtraSlides}{

\begin{circuitikz}[scale=.45, every node/.style={scale=0.6}][american voltages]
    \pgfmathsetmacro{\w}{9}
    \pgfmathsetmacro{\h}{3}
    \pgfmathsetmacro{\boxwidth}{2}
    \pgfmathsetmacro{\midspace}{\w - \boxwidth}
    \pgfmathsetmacro{\wcompsize}{\w/3}
    \pgfmathsetmacro{\hcompsize}{3*\h/4}
    \pgfmathsetmacro{\hmargin}{2*\h/5}

    \large

    \foreach \index in {l, j} {
        
        \IfStrEq{\index}{j}
            { \pgfmathsetmacro{\sign}{-1} }
            { \pgfmathsetmacro{\sign}{1} }
    \pgfmathsetmacro{\heigt}{ (1+\sign) * ( \hmargin + 1/2 * \hcompsize ) }

    \wiremulseparate{-3 * \wcompsize + \boxwidth}{\heigt}{0.2}{0.4}{\large $m$}
    
    \circuitbox{\wcompsize}{\heigt}{\boxwidth/2}{\boxwidth/2}{\Large $\nabla h_\index$}
    \circuitbox{-3 * \wcompsize }{\heigt}{\boxwidth/2}{\boxwidth/2}{\Large $\partial f_\index$}

    \draw ( -2 * \wcompsize, \heigt ) node[yshift=\sign * 10] {$\tilde{x}_\index$}
        to ( -3 * \wcompsize + \boxwidth/2, \heigt );

    \draw ( -2 * \wcompsize, \heigt ) 
        to [short,  R, l_=$R$, *-] ( - \wcompsize, \heigt ) node[yshift=\sign *10] {$e_\index$} 
        to [short,  R, l_=$-R$, *-] ( 0, \heigt ) node[yshift=\sign *10] {$x_\index$} 
        to [short, i=$ $, *-] ( \wcompsize - \boxwidth/2, \heigt );
    }      

    \draw ( 0, \hmargin )
    to [short]( \wcompsize/4, \hmargin )
    to [short,  R, l_=$R_{jl}$] ( \wcompsize/4, \hmargin + \hcompsize)
    to ( 0, \hmargin + \hcompsize);
        
    \draw ( 0, 0 )
    to [short, i=$ $] ( 0, \hmargin )
    to [short] ( -\wcompsize/4, \hmargin )
    to [short,  L, l=$L_{jl}$] ( -\wcompsize/4, \hmargin + \hcompsize)
    to ( 0, \hmargin + \hcompsize)
    to [short, -*] ( 0, 2 * \hmargin + \hcompsize );

\end{circuitikz}
}

\usepackage{circuitikz}
\usepackage{ifthen}
\usepackage{xstring}

\usepackage{cancel}
\usepackage{enumitem}   

\usepackage{natbib}
\usepackage{amsthm,amssymb}

\usepackage{wrapfig}
\usepackage{subcaption}
\usepackage{arydshln}
\usepackage{mathtools}
\usepackage{blkarray}

\title{Optimization Algorithm Design via Electric Circuits}

\author{
Stephen P. Boyd\\
Stanford University
\And
Tetiana Parshakova$^*$\\
Stanford University
\And
Ernest K. Ryu\\
UCLA
\And
Jaewook J. Suh\thanks{Lead authors (author list ordered alphabetically)}\\
Rice University\\
}

\begin{document}

\maketitle

\begin{abstract}
We present a novel methodology for convex optimization algorithm design using ideas from electric RLC circuits. Given an optimization problem, the first stage of the methodology is to design an appropriate electric circuit whose continuous-time dynamics converge to the solution of the optimization problem at hand. Then, the second stage is an automated, computer-assisted discretization of the continuous-time dynamics, yielding a provably convergent discrete-time algorithm. Our methodology recovers many classical (distributed) optimization algorithms and enables users to quickly design and explore a wide range of new algorithms with convergence guarantees.
\end{abstract}

\section{Introduction}
In the classical literature of optimization theory, optimization algorithms are designed with the goal of establishing fast worst-case convergence guarantees. However, these methods, designed with the pessimistic framework of worst-case analysis, often exhibit slow practical performance. In the modern machine learning literature, optimizers are designed with the goal of obtaining fast empirical performance on a set of practical problems of interest. However, these methods, designed without consideration of the feasibility of a convergence analysis, tend to be much more difficult to analyze theoretically, and such methods sometimes even fail to converge under nice idealized assumptions such as convexity \cite{KingmaBa2015_adam,Reddi2018}.

In this work, we present a novel methodology for convex optimization algorithm design using ideas from electric RLC  (resistor-inductor-capacitor) circuits and the performance estimation problem \cite{DroriTeboulle2014_performance,TaylorHendrickxGlineur2017_smooth}. (To clarify, our proposal does not involve building a physical circuit.) Specifically, our methodology provides a quick and systematic recipe for designing new, provably convergent optimization algorithms, including distributed optimization algorithms. The ease of the methodology enables users to quickly explore a wide range of algorithms with convergence guarantees.

\paragraph{Optimization problem formulation.}
We consider the standard-form optimization problem
\BEQ\label{e-dist-opt-primal}
\begin{array}{ll}
\mbox{minimize}&
    f(x)\\
    \mbox{subject to}& x\in \mathcal{R}(E^\intercal ),
\end{array}
\EEQ
where $x \in \reals^m$ is the optimization variable, $f\colon \reals^m \to \reals \cup \{ \infty \}$ is closed, convex, 
and proper, and $E\in \reals^{n\times m}$.
Assume we have $n$ nets $N_1,\dots,N_n$ forming a partition of $\{1,\dots,m\}$.
More specifically, we let $E\in \reals^{n\times m}$ be a \emph{selection matrix} defined as
\BEQ\label{e-selection-matrix}
E_{ij}=\left\{
\begin{array}{ll}
+1&\text{if $j\in N_i$ }\\
0&\text{otherwise.}
\end{array}
\right.
\EEQ
% (Therefore, $m\ge n\ge 1$.)
Our goal is to find a primal-dual solution satisfying the KKT conditions \cite[Theorem 28.3]{rockafellar1997convex}
\BEQ
\label{pd-sol}
 y\in \partial f(x),\,\quad
x\in \mathcal{R}(E^\intercal),\quad y\in \mathcal{N}(E).
\EEQ
As we show through examples, this standard-form problem \eqref{e-dist-opt-primal} conveniently models many optimization problem setups of practical interest.

\newpage

In the analysis and design of optimization algorithms, a standard approach is to consider a continuous-time model of a given algorithm, corresponding to the limit of small stepsizes \cite{Polyak1964_methods, HelmkeMoore1996_optimization, Alvarez2000_minimizing, WangElia2010_control, SuBoydCandes2014_differential, KricheneBayenBartlett2015_accelerated, KiaCortesMartinez2015_distributed, ShiDuJordanSu2021_understanding}. Our work is based on the key observation that such continuous-time models can be interpreted as RLC circuits connected to the subdifferential operator $\partial f$, which we interpret as a nonlinear resistor. We expand on this observation and propose a general methodology for designing optimization algorithms by designing RLC circuits that relax to the nets defined by $E$.

\paragraph{Example.}
Problem~\eqref{e-dist-opt-primal} represents a general form of distributed optimization,
where the constraints enforce consensus among the primal
variables.
An example is the so-called consensus problem~\cite[\S 7]{boyd2011distributed}
\[
\begin{array}{ll}
\underset{{x_1, \ldots, x_N\in \textbf{R}^{m/N}}}{\mbox{minimize}}&
    f_1(x_1) + \cdots + f_N(x_N)\\
    \mbox{subject to}& x_1 = \cdots = x_N,
\end{array}
\]
where $x = (x_1, \ldots, x_N)$ is the decision variable, 
the objective function $f(x) = f_1(x_1)+\cdots +f_N(x_N)$ is block-separable, and
$E^\intercal = (I, \ldots, I) \in \reals^{m \times m/N}$.
Refer to sections \S\ref{appendix-centralized} and \S\ref{appendix-decentralized},
for an overview of classical splitting methods and decentralized methods
for solving \eqref{e-dist-opt-primal}.

\subsection{Preliminaries and Notation}
We generally follow the standard definitions and notations of convex optimization \cite{boyd2004convex, Nesterov2004_introductory, BauschkeCombettes2017_convex, Nesterov2018_lectures, RyuYin2022_largescale}. 
% Throughout this paper, we denote $\reals^n$ as the underlying Euclidean space with Euclidean norm $\|\cdot\|$ and inner product $\left\langle \cdot , \cdot \right\rangle$.
Consider the extended-valued function $f\colon\reals^n \to \reals \cup \{\infty\}$.
We say $f$ is closed if its epigraph %$\text{epi} f = \{ (x, \tau) \in \reals^n \times \reals \mid f(x) \le \tau  \}$ 
is closed set in $\reals^{n+1}$ and proper if its value is finite somewhere. 
We say $f$ is CCP if it is closed, convex, and proper. 
For $R>0$, we say $f$ is $R$-smooth if $f$ is finite and differentiable everywhere and $\| \nabla f(x) - \nabla f(y)  \| \le R \| x - y  \|$ for all $x, y \in \reals^n$.
% \[
%     \| \nabla f(x) - \nabla f(y)  \| \le R \| x - y  \|, \qquad \forall x, y \in \reals^n.
% \]
For $\mu>0$, we say $f$ is $\mu$-strongly convex if $f(x) - (\mu/2) \|x\|^2$ is convex. 
Let $f^*(y) = \sup_{x\in \textbf{R}^n} \{ \langle y , x \rangle - f(x)  \}$ denote the Fenchel conjugate of $f$.
% Let $f$ be a CCP function on $\reals^n$.
For $R>0$ and a CCP $f$, define the $R$-Moreau envelope of $f$ as 
 ${}^R\!f(x) = \inf_{z\in \textbf{R}^n}\big\{f(z)+\tfrac{1}{2R}\|z-x\|^2\big\}$.
One can show \citep[Proposition 13.24]{BauschkeCombettes2017_convex} that the $R$-Moreau envelope is given by 
${}^R\!f=\big(f^*+\tfrac{R }{2}\|\cdot\|^2\big)^*$.
% \[
%  {}^R\!f=\big(f^*+\tfrac{R }{2}\|\cdot\|^2\big)^*.
% \]
% where, to clarify, superscript $*$ denotes the Fenchel conjugate.
If $f$ is $1/R$-smooth, we can define \cite[Theorem 18.15]{BauschkeCombettes2017_convex} the $R$-pre-Moreau envelope of $f$ as
\[
\tilde{f}=\big(f^*-\tfrac{R }{2}\|\cdot\|^2\big)^*,
\]
which is defined such that ${}^{R}(\tilde{f})=f$.

Due to the limited space, we defer the review of prior works to \S\ref{s:prior_works} of the appendix.

\subsection{Contributions}
Our work presents two technical novelties, one in continuous time and the other in discrete time. The first is the observation that many standard optimization algorithms can be interpreted as discretizations of electric RLC circuits connected to the subdifferential operator $\partial f$. The second is the use of the performance estimation problem to obtain an automated recipe for discretizing convergent continuous-time dynamics into convergent discrete-time algorithms, and we provide code implementing our automatic discretization methodology.

By combining these two insights, we provide a quick and systematic methodology for designing new, provably convergent optimization algorithms, including distributed optimization algorithms. 
We provide an open-source package that implements automatic discretization
of our circuits:
\begin{quote}
\centering
\url{https://github.com/cvxgrp/optimization_via_circuits}
\end{quote}   
% We also identified dynamic interconnects and discretized them using our package for gradient descent, 
% Nesterov acceleration, proximal point method,
% proximal gradient method, primal decomposition, dual decomposition,
% Davis-Yin splitting, Douglas-Rachford splitting, decentralized gradient descent,
% diffusion, decentralized ADMM and PG-EXTRA.

\section{Continuous-time optimization with circuits} \label{sec-continuous-optimization-with-circuit}

% \begin{wrapfigure}{r}{0.465\linewidth} \vspace{-8mm}
% \begin{center}
%     \StaticInterconnect
% \end{center} \vspace{-2mm}
%     % \caption{Example of a static interconnect with $m=5$, $N_1=\{1,3\}$, $N_2=\{2,4\}$, $N_3=\{5\}$. }
%     \caption{Example of a static interconnect, $m=5$, $N_1=\{1,3\}$, $N_2=\{2,4\}$, $N_3=\{5\}$. }
%     \label{fig-static-ic} \vspace{-10mm}
% \end{wrapfigure} 

\subsection{Interconnects}
We now describe two types of electric circuits that we call \emph{static} and \emph{dynamic interconnects}. Both interconnects have $m$ terminals, and we will later connect them to the $m$ inputs of $\partial f$.

\paragraph{Static interconnect.}
The static interconnect is a set of (ideal) wires connecting $m$ terminals and forming $n$ nets. See Figure~\ref{fig-static-ic} for an example.
Let $x \in\reals^m$ be a vector of terminal potentials and $y\in \reals^m$
be a vector of currents leaving the terminals.
Using matrix $E\in\reals^{n\times m}$ as defined in \eqref{e-selection-matrix}, we can express Kirchhoff's voltage law (KVL) as $x\in \mathcal{R}(E^\intercal )$
and Kirchhoff's current law (KCL) as $y\in \mathcal{N}(E)$. %\vspace{1mm}
\begin{wrapfigure}{r}{0.465\linewidth}
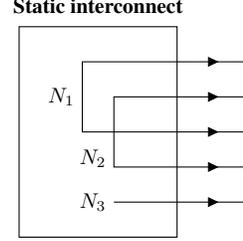
 
 \vspace{-4mm}
\begin{center}
    \StaticInterconnect
\end{center} %\vspace{-2mm}
    \caption{Example of a static interconnect, $m=5$, $N_1=\{1,3\}$, $N_2=\{2,4\}$, $N_3=\{5\}$. }
    \label{fig-static-ic} \vspace{-2mm}
\end{wrapfigure} 
In other words, the static interconnect enforces the V-I relationship
\BEQ
(x,y)\in\mathcal{R}(E^\intercal )\times  \mathcal{N}(E).
\label{eq:static-interconnect}
\EEQ

\paragraph{Dynamic interconnect.}
The dynamic interconnect is an RLC circuit with $m$ terminals and $1$ ground node.
We assume all inductances and capacitances have values in $(0,\infty)$ while the resistances have values in $[0,\infty)$. (A $0$-ohm resistor is an ideal wire. We do not permit ideal wire loop.) Each RLC component has two (scalar-valued) terminals: the $+$ and $-$ terminals.

Denote the number of nodes in the RLC circuit by $\tau$. Connect nodes $1, 2, \ldots, m$ to terminals $1, 2, \ldots, m$, and let the last node, node $\tau$, be the ground node. (This implies $\tau \geq m+1$.)
Denote the number of RLC components by $\sigma$. We describe the topology with a reduced node incidence matrix (with the bottom row corresponding to the ground node removed)
$A\in \reals^{(\tau - 1)\times \sigma}$ defined as
\BEAS
A_{ij}&=&\left\{
\begin{array}{ll}
+1&\text{if } \text{node $i$ connects to $+$ terminal of component $j$} \\
-1&\text{if node $i$ connects to $-$ terminal of component $j$}\\
0&\text{otherwise}.
\end{array}
\right.
\EEAS
See Figure~\ref{fig-RLC-example} for an example.

\begin{figure}  [H]
    \centering    
    \begin{subfigure}[c]{0.4\textwidth}
    \centering
    \begin{center} 
        % \DynamicInterconnect
        \DynamicInterconnectNew
    \end{center}
    \end{subfigure}
    % \hfill
    \hspace{3mm}
    \begin{subfigure}[c]{0.35\textwidth}
        \centering 
        \begin{center}
        \begin{align*}
        &A=\!\!\! %\\&
        {%\small
        \begin{blockarray}{c@{\rcm}c@{\rcm}c@{\rcm}c@{\rcm}c@{\rcm}c@{\rcm}c@{\rcm}c}
        % \begin{blockarray}{cccccccc}
            & R_1 & R_2 & R_3 & L_1 & L_2 & C_1 & C_2 \\
            \begin{block}{c@{\hspace{3mm}}[c@{\rcm}c@{\rcm}c@{\rcm}c@{\rcm}c@{\rcm}c@{\rcm}c]}
            % \begin{block}{c[ccccccc]}
            1 & +1 & 0  & 0  & -1 & 0  & 0  & 0   \\
            2 & 0  & +1  & 0  & 0  & 0 & 0  & 0   \\
            3 & 0  & 0 & 0  & 0  & -1  & 0  & 0   \\
            4 & 0  & 0  & +1 & 0  & 0  & 0  & 0   \\
            5 & 0  & 0  & 0  & 0  & 0  & -1 & +1  \\
            6 & -1 & 0  & 0  & +1 & +1 & 0  & 0   \\
            7 & 0  & -1 & -1 & 0  & 0  & +1 & 0   \\
            % 8 & 0  & 0  & 0  & 0  & 0  & 0  & -1  \\
            \end{block}
        \end{blockarray}
        }
        \end{align*}
        \end{center}
    \end{subfigure}
    % \vspace{-0.1cm}    
    \caption{Example of a dynamic interconnect with $\tau=8$ nodes, $\sigma=7$ RLC components, $m=5$ terminals, and $1$ ground node. Reduced node incidence matrix $A$ is provided. ($R_2$ and $R_3$ are $0$-ohm resistors.)
    This dynamic interconnect is admissible with respect to the static interconnect of Figure~\ref{fig-static-ic}.
    }
    \label{fig-RLC-example}
\end{figure}

The ground node is designated to have $0$ potential, and the \emph{potential} of any node is the potential relative to ground.
The \emph{voltage} across a component is the difference of potentials between the $+$ and $-$ terminals.
The \emph{current} through a component is defined as the current flowing from the $+$ terminal to the $-$ terminal.

Let $x\in \reals^m$ be the potentials at the $m$ terminals, which are connected to nodes $1,\dots,m$, and $y\in \reals^m$ be the currents leaving the terminals.
Denote the node potential vector with the ground node excluded 
(since the potential at ground is $0$)
by 
\[
\begin{bmatrix}
x\\
 e
\end{bmatrix}  \in \reals^{\tau-1}.
\]
So, $e\in \reals^{\tau-1-m}$ denotes the potentials at the non-terminal nodes.
Denote the vector of voltages by $v \in \reals^\sigma$
and the vector of currents by $i \in \reals^{\sigma}$.
% Connect nodes $1, 2, \ldots, m$ to terminals $1, 2, \ldots, m$.
Then, the currents and voltages of the dynamic interconnect satisfy the following V-I relations
\begin{gather*}
\text{(i) } Ai = \begin{bmatrix}
-y\\
0\end{bmatrix}\quad \text{(KCL)}\qquad\qquad
\text{(ii) } v = A^\intercal \begin{bmatrix}
x\\
 e
\end{bmatrix} \quad \text{(KVL)}\\
\!\!\!\!
\text{(iii) }  v_{\mathcal{R}} = D_\mathcal{R} i_{\mathcal{R}} \quad\text{(Resistor)}\qquad
\text{(iv) }v_{\mathcal{L}} = D_\mathcal{L} \frac{d}{dt} i_{\mathcal{L}} \quad \text{(Inductor)}\qquad
\text{(v) }i_{\mathcal{C}} = D_\mathcal{C} \frac{d}{dt} v_{\mathcal{C}} \quad \text{(Capacitor)}
\end{gather*}
where $D_\mathcal{R}$, $D_\mathcal{L}$, and $D_\mathcal{C}$ are diagonal matrices respectively with resistances, inductances, and capacitances values in the diagonals.
% Note that the relations for inductors and capacitors are dynamic. 

% \begin{enumerate}[label=(\roman*)]
% %     \item \textbf{KCL:} $Ai = \begin{bmatrix}
% % -y\\
% % 0_{\tau-1-m}
% % \end{bmatrix} \in \reals^{\tau-1}$.

% % \item \textbf{KVL:} $v = A^\intercal e \in \reals^\sigma$,
% % $e = \begin{bmatrix}
% % x\\
% % \tilde e
% % \end{bmatrix} \in \reals^{\tau-1}$.
% % \item \textbf{Inductor relation:} $v_{\mathcal{L}} = D_\mathcal{L} \dot i_{\mathcal{L}} \in \reals^{|L|}$.
% % \item \textbf{Capacitor relation:} $i_{\mathcal{C}} = D_\mathcal{C} \dot v_{\mathcal{C}} \in \reals^{|C|}$.
% % \item \textbf{Resistor relation:} $v_{\mathcal{R}} = D_\mathcal{R} i_{\mathcal{R}} \in \reals^{|R|}$.
%     \item $Ai = \begin{bmatrix}
% -y\\
% 0\end{bmatrix}$\quad (KCL)
% \item $v = A^\intercal \begin{bmatrix}
% x\\
%  e
% \end{bmatrix} $\quad (KVL)
% \item  $v_{\mathcal{R}} = D_\mathcal{R} i_{\mathcal{R}} $\quad (Resistor)
% \item $v_{\mathcal{L}} = D_\mathcal{L} \frac{d}{dt} i_{\mathcal{L}}$ \quad (Inductor)
% \item  $i_{\mathcal{C}} = D_\mathcal{C} \frac{d}{dt} v_{\mathcal{C}} $\quad (Capacitor)
% \end{enumerate}

\paragraph{Admissibility.}
When an RLC circuit reaches equilibrium, voltages across inductors and currents through capacitors are $0$.
We say a dynamic interconnect is \emph{admissible} if it relaxes to the static interconnect at equilibrium. Mathematically, this condition is expressed as
\[
\Big\{(x,y) \,\Big|\, Ai = \begin{bmatrix}
-y\\
0
\end{bmatrix}, v = A^\intercal \begin{bmatrix}
x\\
e
\end{bmatrix}, v_{\mathcal{R}} = D_\mathcal{R} i_{\mathcal{R}} , v_{\mathcal{L}}=0, i_{\mathcal{C}}=0\Big \}
= \mathcal{R}(E^\intercal) \times \mathcal{N}(E).
\]
% When an RLC circuit reaches equilibrium, $v_{\mathcal{L}}=0$ and $i_{\mathcal{C}}=0$ hold. 
As an example, the dynamic interconnect of Figure~\ref{fig-RLC-example} is admissible with respect to the static interconnect of Figure~\ref{fig-static-ic}.

% \newpage

\subsection{Composing interconnects with $\partial f$}
% Let $X^\star\times Y^\star$ be the primal-dual solution set for the primal problem \eqref{e-dist-opt-primal}. 
% Suppose strong duality holds (which can be ensured by checking Slater's constraint qualification 
% on the primal or the dual problem).

We view the subdifferential operator $\partial f$ as an $m$-terminal electric device that is also grounded. Let $x\in \reals^m$ be the potentials at the $m$ terminals (excluding ground) and $y\in \reals^m$ be the currents flowing into the $m$ terminals. The $\partial f$ operator enforces the V-I relation
\[
y \in \partial f(x).
\]
% , \ie, the V-I characteristic is
% \[
% \{ (x, y) \mid y \in \partial f(x)\}.
% \]
% Therefore, this device acts as a nonlinear resistor.
% \begin{wrapfigure}{r}{0.53\linewidth} \vspace{-5mm}
%     \begin{center} 
%     \StaticInterconnectF
%     \end{center} 
%     \vspace{-0.2cm}
%     \caption{
%    The static interconnect of Figure~\ref{fig-static-ic} connected with $\partial f$.
%    The potentials at the $m$ terminals is an optimal $x^\star\in \reals^m$ solving \eqref{e-dist-opt-primal}. 
%    }
%     \label{fig-static-ic-f} \vspace{-5mm}
% \end{wrapfigure}

We connect the $m$ terminals of $\partial f$ to the $m$ terminals of the static and dynamic interconnects. Immediately, connecting the static interconnect with $\partial f$ enforces the V-I relations \eqref{eq:static-interconnect} and $y \in \partial f(x)$, which combine to be the optimality condition \eqref{pd-sol}. Therefore, the potentials at the $m$ terminals as a vector in $\reals^m$ is an optimal $x^\star\in \reals^m$ solving \eqref{e-dist-opt-primal}. To clarify, connecting the static interconnect with $\partial f$ leads to a \emph{static} circuit in the sense that the potential $x$ and current $y$ do not depend on time.

\begin{figure} [H]
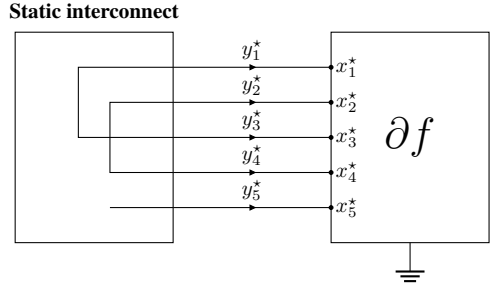

    \begin{center} 
    \StaticInterconnectF
    \end{center} 
    % \vspace{-0.2cm}
    \caption{
   The static interconnect of Figure~\ref{fig-static-ic} connected with $\partial f$.
   The potentials at the $m$ terminals is an optimal $x^\star\in \reals^m$ solving \eqref{e-dist-opt-primal}. 
   }
    \label{fig-static-ic-f}
\end{figure}

% \begin{figure}[ht]
%     \StaticInterconnectF
%     \vspace{-0.2cm}
%     \caption{
%    The static interconnect of Figure~\ref{fig-static-ic} connected with $\partial f$.
%    The potentials at the $m$ terminals as a vector in $\reals^m$ is an optimal $x^\star\in \reals^m$ solving \eqref{e-dist-opt-primal}.
%     }
%     \label{fig-static-ic-f}
% \end{figure} 

Next, we compose (connect) the dynamic interconnect with $\partial f$. Due to capacitors and inductors, this circuit is \emph{dynamic} in the sense that the voltages $v(t)$ and $x(t)$ and currents $i(t)$ and $y(t)$ depend on time, although we often omit explicitly writing the $t$-dependence for notational convenience.
Then, the V-I relations of the dynamic interconnect combined with $y\in \partial f(x)$ leads to the V-I relation
\begin{align}    
\label{e-dyn-ic-partial-f}
\Big\{(v,i) \,\Big|\, &
y \in \partial f(x),\,\,
Ai = \begin{bmatrix}
-y\\
0
\end{bmatrix},\,\, v = A^\intercal \begin{bmatrix}
x\\
 e
\end{bmatrix},\\
&\quad v_{\mathcal{R}} = D_\mathcal{R} i_{\mathcal{R}},\,\,
v_{\mathcal{L}}=D_\mathcal{L} \frac{d}{dt} i_{\mathcal{L}},\,\,
i_{\mathcal{C}}=D_\mathcal{C}\frac{d}{dt} v_{\mathcal{C}},\,\,t\in(0,\infty)\Big\},
\nonumber
\end{align}
where $v(t)=(v_\mathcal{R}(t),v_\mathcal{L}(t),v_\mathcal{C}(t))\in \reals^\sigma$, $i(t)=(i_\mathcal{R}(t),i_\mathcal{L}(t),i_\mathcal{C}(t))\in \reals^\sigma$, %\\
$e(t)\in \reals^{\tau-m-1}$, $x(t)\in \reals^{m}$, and $y(t)\in \reals^{m}$ for $t\in [0,\infty)$.

\begin{figure}[ht]
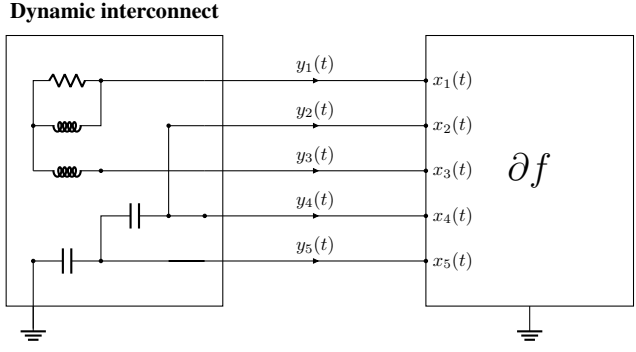

\begin{center}
% \DynamicInterconnectF
\DynamicInterconnectFnew
\end{center}
\caption{
 The dynamic interconnect of Figure~\ref{fig-RLC-example} connected with $\partial f$. The potentials at the $m$ terminals satisfy $x(t)\rightarrow x^\star$ for an optimal $x^\star\in \reals^m$ solving \eqref{e-dist-opt-primal} under the conditions of Theorem~\ref{thm:convergence}.}
\label{fig-dynamic-ic-f}
\end{figure} 

Under appropriate conditions, the dynamics \eqref{e-dyn-ic-partial-f} is mathematically well-posed in the sense that there exist unique Lipschitz-continuous curves $v(t)$, $i(t)$, $x(t)$, and $y(t)$ satisfying the V-I relation \eqref{e-dyn-ic-partial-f} as formalized in the following Theorem~\ref{thm-well-posedness}.
The proof, which utilizes the machinery of monotone operator theory \cite{aubin2012differential,  BauschkeCombettes2017_convex, RyuYin2022_largescale}, is provided in \S\ref{s-well-posedness} of the appendix.
% In the appendix, we formalize this well-posedness statement as Theorem~\ref{thm-well-posedness}.

% \paragraph{Well-posedness.}
% A priori, it is not clear whether the dynamics of dynamic interconnect composed with $\partial f$ is mathematically well-posed in the sense that a solution exists and is unique. In the Theorem~\ref{thm-well-posedness}, we show that it is. 

\begin{theorem}
\label{thm-well-posedness} 
Assume $f$ is $\mu$-strongly convex and $M$-smooth. 
Suppose $(v^0, i^0, x^0, y^0)$ satisfy
\[
Ai^0 = \begin{bmatrix} -y^0\\ 0 \end{bmatrix}, \quad
            v^0 = A^\intercal \begin{bmatrix} x^0\\  e \end{bmatrix}, \quad
            v_{\mathcal{R}}^0 = D_\mathcal{R} i_{\mathcal{R}}^0, \quad
            y^0 = \nabla f(x^0).
\]  
Then there is a unique Lipschitz continuous curve  
$(v, i, x, y) \colon [0,\infty) \to \reals^\sigma \times \reals^\sigma \times \reals^m \times \reals^m$
satisfying the conditions in \eqref{e-dyn-ic-partial-f} and the initial condition $(v(0), i(0), x(0), y(0)) = (v^0, i^0, x^0, y^0)$.
\end{theorem}

\paragraph{Equillibrium yields a primal-dual solution.}
With the dynamic interconnect composed with $\partial f$, we generically expect the circuit state $(v(t), i(t), x(t), y(t))$ to converge (relax) to an equilibrium state.
% in which the voltage and currents of the capacitors and inductors do not change and all time derivatives are $0$.
The admissibility condition ensures that at such an equilibrium, $(x,y)$ will be a primal-dual solution. We formally state this fact as Theorem~\ref{thm-equilibrium} of the Appendix.

\subsection{Energy dissipation}

Let $(v^\star,i^\star,x^\star,y^\star)$ be an equilibrium of an admissible dynamic interconnect composed with $\partial f$. 
Since the voltages across resistors and inductors and the currents through capacitors are zero under equilibrium, we have
\[
v^\star = (v_{\mathcal{R}}^\star, v_{\mathcal{L}}^\star, v_{\mathcal{C}}^\star)= (0, 0, v_{\mathcal{C}}^\star), \qquad 
i^\star= (i_{\mathcal{R}}^\star, i_{\mathcal{L}}^\star, i_{\mathcal{C}}^\star)= (0, i_{\mathcal{L}}^\star, 0).
\]
(We formally show this in Theorem~\ref{thm-equilibrium} of the appendix.)
Define the energy of the circuit at time $t$ as
\BEQ\label{e-total-energy}
\mathcal{E}(t)=\frac{1}{2}\|v_{\mathcal{C}}(t)-v_{\mathcal{C}}^\star\|^2_{D_\mathcal{C}} +
\frac{1}{2}\|i_{\mathcal{L}}(t)-i_{\mathcal{L}}^\star\|^2_{D_\mathcal{L}}, 
\EEQ
% {\color{red} which is related to the Lyapunov function considered in \cite{Willems1971_generation}}. 
which is a dissipative (non-increasing) quantity:
\begin{align}
\frac{d}{dt}   \mathcal{E}
&=
\langle v_{\mathcal{C}}-v_{\mathcal{C}}^\star,
i_{\mathcal{C}} -i_{\mathcal{C}}^\star
\rangle 
+
\langle i_{\mathcal{L}}-i_{\mathcal{L}}^\star,
v_{\mathcal{L}}-v^\star_\mathcal{L}
\rangle  \nonumber\\
&=
-
\|i_{\mathcal{R}}\|_{D_\mathcal{R}}^2
-
\langle x - x^\star,
y - y^\star
\rangle 
\le 0. \label{eq-cont-power}
\end{align}
Here, we use $i^\star_{\mathcal{C}}=0$ and $v^\star_\mathcal{L}=0$ and the fact that the power dissipated by the resistors and $\partial f$ must come from the energy stored in the capacitors and inductors. 
% Considering integration from $0$ to $\infty$, we may expect $\lim_{t\to\infty} \mathcal{E}(t) = 0$.
% Then using LaSalle's principle, we may expect $\lim_{t\to\infty} \mathcal{E}(t) = 0$. 
This dissipativity property leads to the following continuous-time convergence.
\begin{theorem}
\label{thm:convergence}
    Assume $f\colon \reals^m \to \reals $ is strongly convex and smooth.
    Assume the dynamic interconnect is admissible, and let $(x^\star, y^\star)$ be a primal-dual solution pair.  
    Let $(v(t),i(t),x(t),y(t))$ be a curve satisfying \eqref{e-dyn-ic-partial-f}.
    Then,     %\vspace{-0.03in}
    \[
        \lim_{t\to\infty} (x(t), y(t)) = (x^\star, y^\star).
    \]
\end{theorem}

Theorem~\ref{thm:convergence} largely follows as a corollary of Theorem~\ref{thm-well-posedness}. The formal proof is provided in  \S\ref{s-energy-dissipation} of the appendix.
In \S\ref{sec-auto-discretization}, we present a systematic framework for finding discretized versions of Theorem~\ref{thm:convergence} the corresponding discretized algorithms.

% The energy function \eqref{e-total-energy} used in the proof is related to the Lyapunov function considered in \cite{Willems1971_generation}. 
% However, the dissipativity theory presented in \cite{Willems1971_generation, Willems1972_dissipative} does not directly apply to our setup. 
% In our setup, we allow cases where $v_C, i_L$ oscillate, for example, a circuit with a
% disconnected $L-C$ loop.

% discrete version of this 
% technique, 
% establishing a systematic framework for finding discrete time convergent algorithms. 

% These steps form the foundation of our proof technique, which is broadly applicable to 
% the dynamic interconnects under our consideration.
% The formalized statement for continuous dynamics is presented below, 
% with its proof provided in \S\ref{s-energy-dissipation}.

\section{Circuits for classical algorithms} \label{sec-circuits-classical-algorithms}

In this section, we present circuits recovering the classical Nesterov acceleration, decentralized ADMM, and PG-EXTRA.
For additional examples and detailed
derivations, refer to \S\ref{appendix-centralized} 
and \S\ref{appendix-decentralized} of the appendix, where we provide circuits and analyses of classical algorithms 
such as gradient descent~\cite{Cauchy1847_methode}, proximal point method \cite{Rockafellar1976_monotone},
proximal gradient method~\cite{CombettesWajs2005_signal}, 
primal decomposition~\cite{Geoffrion1970_primal, Silverman1972_primal}, dual decomposition~\cite{Everett1963_generalized, Lemarechal2001_lagrangian, Fisher2004_lagrangian, SontagGlobersonJaakkola2011_introduction},
Douglas--Rachford splitting~\cite{PeacemanRachford1955_numerical,DouglasRachford1956_numerical,LionsMercier1979_splitting}, 
Davis--Yin splitting~\cite{DavisYin2017_threeoperator}, 
decentralized gradient descent~\cite{NedicOzdaglar2009_distributed, YuanLingYin2016_convergence}, and diffusion~\cite{CattivelliLopesSayed2007_diffusion,CattivelliSayed2010_diffusion}.

\begin{wrapfigure}{r}{0.305\linewidth}
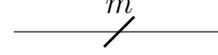
 
% \vspace{-14mm}
% \vspace{-8mm}
\begin{center}
\mTerminal
\end{center} \vspace{-2mm}
   \caption{Multi-wire notation.}
   \label{fig:multiwire}
% \vspace{-10mm}
\vspace{-2mm}
\end{wrapfigure}

\paragraph{Multi-wire notation.}
We start by quickly introducing the multi-wire notation depicted in Figure~\ref{fig:multiwire}. When optimizing $f\colon\reals^m\rightarrow\reals$ and using the $m$-terminal device $\partial f$, we will 
often use dynamic interconnects that have the same RLC circuit across each net, \ie, the dynamic interconnect consists of $m$ identical copies of the same RLC circuit for the $m$ coordinates of $x\in \reals^m$. In this case, we use the diagonal-line notation depicted in Figure~\ref{fig:multiwire}.

% \subsection{Moreau envelope} \label{sec-resistors-Moreau-envelope}
\paragraph{Moreau envelope.}
We use the following simple identity throughout this work:
$\partial f$ composed with a resistor is equivalent to $\nabla{}^R\!f(x)$.
% For $R>0$, define the Moreau envelope of $f\colon\reals^m\to\reals\cup\{\infty\}$ with parameter $R$ as
% \[
% {}^R\!f(x) = \inf_{z \in \reals^m} \Big( f(z) + \frac{1}{2R} \|z - x\|^2_2\Big).
% \]
% \vspace{-0.08in}
\begin{figure}[H]
    \MoerauEnvelope
\end{figure}
% \vspace{-0.25in}
\vspace{-0.15in}
\noindent
To clarify, the equivalence means the two circuits impose the same V-I relation on the $m$ pins of $x$. To see this, note $\big[\partial f(\tilde{x})=\tfrac{1}{R}(x-\tilde{x}) \big]\Leftrightarrow \big[\tilde{x} = \prox_{Rf}(x)\big]$ and use the identity for the gradient of the Moreau envelope to conclude
\[
\nabla { }^R\! f(x) = \frac{1}{R}(x - \prox_{Rf}(x)) = \frac{1}{R}(x - \tilde x).
\]
See \S\ref{sec-resistors-Moreau-envelope} of the appendix for further details.

\subsection{Nesterov acceleration} \label{s-nesterov}
Let $f\colon\reals^m\to\reals$ be a $1/R$-smooth convex function. 
Then, the circuit corresponding to the classical Nesterov acceleration is 
given below.
\begin{figure}[H]
    \centering
    \begin{subfigure}{.4\textwidth}
    % \ctikzset{bipoles/length=1.2cm}
    % \centering
    \begin{circuitikz}[scale=.6, every node/.style={scale=0.7}][american voltages]
        \pgfmathsetmacro{\w}{6}
        \pgfmathsetmacro{\h}{6}
        \pgfmathsetmacro{\boxheight}{4}
        \pgfmathsetmacro{\boxwidth}{2}
        \pgfmathsetmacro{\compwidth}{\w/3}
        \pgfmathsetmacro{\wirespace}{\h/4}
        
        \circuitbox{\w+\boxwidth}{0}{\boxwidth/2}{\boxwidth/2}{\Large $\nabla f$}  
        
        \wiremul{\w+\boxwidth/8}{0}{0.2}{$m$}  
        
        \draw
            ( 0 , -0.5 * \wirespace ) node[ground] {}
            to [short, -*] ( 0 ,0);
        
        \draw ( 0, 0 )
        to [short, C, l=$C$] (\compwidth, 0);    
        
        \draw  ( \compwidth, 0 ) 
        to [short, *-] ( \compwidth, 0 + \wirespace/2 ) 
        to [short, R, l=$R$] ( 2 * \compwidth, 0 + \wirespace/2 )
        to ( 2 * \compwidth, 0 );
        
        \draw  ( \compwidth, 0 ) 
        to ( \compwidth, 0 - \wirespace/2 ) 
        to [short, L, l_=$L$] ( 2 * \compwidth, 0 - \wirespace/2 )
        to ( 2 * \compwidth, 0 );
        
        \draw ( 2 * \compwidth, 0 ) 
        to [short, R, *-, l=$-R$] ( 3 * \compwidth, 0 ) 
        to [short] ( 3 * \compwidth, 0 )
        to [short, -*] ( 3 * \compwidth + \boxwidth/2, 0 ) node[above, xshift=-5] {  $x$ };              
    \end{circuitikz}
    \end{subfigure}
    \begin{subfigure}{.09\textwidth}
        \centering
        \raisebox{0.7cm}{\LARGE $\Leftrightarrow$}
    \end{subfigure}%
    \begin{subfigure}{.5\textwidth}
        \centering
        \begin{circuitikz}[scale=.6, every node/.style={scale=0.7}][american voltages]
        \pgfmathsetmacro{\w}{6}
        \pgfmathsetmacro{\h}{6}
        \pgfmathsetmacro{\boxheight}{4}
        \pgfmathsetmacro{\boxwidth}{2}
        \pgfmathsetmacro{\compwidth}{\w/3}
        \pgfmathsetmacro{\wirespace}{\h/4}
        
        \circuitbox{\w+\compwidth+\boxwidth}{0}{\boxwidth/2}{\boxwidth/2}{\Large $\partial \tilde{f}$}  
        
        \wiremul{\w+\compwidth+\boxwidth/8}{0}{0.2}{$m$}  
        
        \draw
            ( 0 , -0.5 * \wirespace ) node[ground] {}
            to [short, -*] ( 0 ,0);
        
        \draw ( 0, 0 )
        to [short, C, l=$C$] (\compwidth, 0);    
        
        \draw  ( \compwidth, 0 ) 
        to [short, *-] ( \compwidth, 0 + \wirespace/2 ) 
        to [short, R, l=$R$] ( 2 * \compwidth, 0 + \wirespace/2 )
        to ( 2 * \compwidth, 0 );
        
        \draw  ( \compwidth, 0 ) 
        to ( \compwidth, 0 - \wirespace/2 ) 
        to [short, L, l_=$L$] ( 2 * \compwidth, 0 - \wirespace/2 )
        to ( 2 * \compwidth, 0 );
        
        \draw ( 2 * \compwidth, 0 )
        to [short, R, *-, l=$-R$] ( 3 * \compwidth, 0 ) node[above] {  $x$ }
        to [short, R, *-, l=$R$] ( 4 * \compwidth, 0 ) 
        to [short] ( 4 * \compwidth, 0 )
        to [short, -*] ( 4 * \compwidth + \boxwidth/2, 0 );
          
    \end{circuitikz}
    \end{subfigure}
\end{figure}
The use of a \emph{negative} resistor $-R$ may seem unconventional, but the fact that this circuit is stable is easier to see if we consider the equivalent circuit with the pre-Moreau envelope $\tilde f$, \ie, $\tilde f$ is the convex function such that ${}^R\! \tilde f = f$. 
To clarify, negative resistors  
satisfy the same V-I relations of the standard resistors but with a negative slope. 
Negative resistors have also been considered in \cite{VichikBorrelli2014_solving}. 
% Negative resistor has long history, perhaps date backs to \cite{Herold1935_negative}.}

The V-I relations of this circuit lead to the ODE
\[
    \frac{d^2}{dt^2} x + \frac{R}{L} \frac{d}{dt} x 
    + \left( \frac{1}{C} - \frac{R^2}{L} \right) \frac{d}{dt} \nabla f(x)
    + \frac{R}{LC} \nabla f(x) = 0. 
\]
If we set $R = \sqrt{L/C}$, which can be interpreted as an instance of critical damping \cite{YangArorabravermanZhao2018_physical, ZhangOrvietoDaneshmand2021_rethinking, chen2023underdamped}, $L = \frac{1}{8 \mu \sqrt{\mu}}$, and $C = 2 \sqrt{\mu}$, we recover the Nesterov ODE \cite{WilsonRechtJordan2021_lyapunov}
\[
    \frac{d^2}{dt^2} x + 2\sqrt{\mu} \frac{d}{dt} x  +  \nabla f(x) = 0. 
\]
We also quickly point out that other choices of parameters lead to the high-resolution ODE introduced in \cite{ShiDuSuJordan2019_acceleration}. See \S \ref{s-nesterov-acceleration} of the appendix for further details.

% \[
%     \frac{d^2}{dt^2} {x} + 2\sqrt{\mu} \frac{d}{dt}{x} + \sqrt{s} \frac{d}{dt} \nabla f(x) + (1 + \sqrt{\mu s} ) \nabla f(x) = 0. 
% \]
% With other choices, we recover high-res ODE

% As an immediate consequence, if we set $R = \frac{1}{4 \mu }$, $L_i = \frac{1}{8 \mu \sqrt{\mu}}$, $C_i = 2 \sqrt{\mu}$, we recover the low-resolution ODE of NAG-SC
% \[
%     \frac{d^2}{dt^2} x + 2\sqrt{\mu} \frac{d}{dt} x 
%     +  \nabla f(x) = 0. 
% \]

\subsection{Decentralized ADMM}

% In a decentralized optimization setup, we are given a graph $G$ 
% specifying the communication pattern between agents.
% This means that each agent is constrained to communicate only to its
% direct neighbors with respect to the edges of the graph.
% specifying the communication pattern between agents.
% We provide the full description of the standard decentralized optimization notation and definitions in \S XXX of the appendix.
% Define $N_i$ to be the neighbors of $i$ in graph $G$. {\color{red} XXX conflicts with net notation XXX}

Let $f_1, \ldots, f_N\colon\reals^{m}\to\reals\cup\{\infty\}$ be CCP functions.
Consider a decentralized optimization setup with graph $G$.
We provide the full description of the decentralized setup and notations in \S \ref{appendix-decentralized} of the appendix.
Define $\Gamma_j$ to be the neighbors of $j$ in graph $G$. 
For simplicity, we only illustrate the circuit related to nodes $j$ and $l$, where $j$ and $l$ are directly connected through an edge in the graph $G$.

The circuit corresponding to decentralized ADMM~\cite{GlowinskiMarroco1975_lapproximation, GabayMercier1976_dual, Gabay1983_chapter, WeiOzdaglar2013_o1k, shi2014linear} is given below.
\vspace{0.1in}
% \begin{figure}[H]
% \vspace{-0.2cm}
    \begin{center} 
    \CircuitDADMMMain
    \end{center}
    \vspace{-0.2cm}
    % \vspace{-0.8cm}
% \end{figure}
In the following, the left column presents the dynamics of the continuous-time circuit
and the right column presents the discretization with stepsize $L/R$, recovering the standard decentralized ADMM:
\begin{center}{
\small 
\begin{tabular}{l|l}
\parbox{0.3\linewidth}{%
\begin{align*}
a_j &= \frac{1}{|\Gamma_j|} \sum_{l \in \Gamma_j} (R{i_{\mL}}_{jl} + e_{jl}) \\
x_j &= \prox_{(R/|\Gamma_j|)f_j}\left ( a_j\right ) \\
e_{jl} &= \frac{1}{2}(x_j+x_l) \\ 
\frac{d}{dt} {i_{\mL}}_{jl} &= \frac{1}{L}(e_{jl} - x_j)
\end{align*}}
&
\parbox{0.3\linewidth}{%
\begin{align*}
a_j^{k+1} &= \frac{1}{|\Gamma_j|} \sum_{l \in \Gamma_j} (R{i^k_{\mL}}_{j,l} + e^k_{jl}) \\
x_j^{k+1} &= \prox_{(R/|\Gamma_j|)f_j}\left (a_j^{k+1}\right ) \\
e_{jl}^{k+1} &= \frac{1}{2}(x_j^{k+1}+x_l^{k+1}) \\ 
{i_{\mL}^{k+1}}_{jl} &= {i_{\mL}^k}_{jl} + \frac{1}{R}(e_{jl}^{k+1} - x_j^{k+1})
\end{align*}}
\end{tabular} }
\end{center}
for every node $j=1, \ldots, N$ and every edge $(j,l)$ in graph $G$.

\subsection{PG-EXTRA}
Let $f_1, \ldots, f_N\colon\reals^{m}\to\reals\cup\{\infty\}$ be CCP functions and
$h_1, \ldots, h_N\colon\reals^{m}\to\reals$ be convex $M$-smooth functions.
Consider a decentralized optimization setup with graph $G$.
The circuit corresponding to PG-EXTRA \cite{ShiLingWuYin2015_proximal} is given below.

\begin{figure}[ht]
    \begin{center} 
    \CircuitPgExtraMain
    \end{center}
    % \vspace{-0.5cm}
    \vspace{-0.3cm}
\end{figure}

Define the mixing matrix $W\in \reals^{N\times N}$ with 
\[
W_{jl} = \left\{ \begin{array}{ll} 
    1 - \sum_{l \in \Gamma_j}\frac{R}{R_{jl}} & \text{if }j=l \\
    \frac{R}{R_{jl}} & \text{if } j \neq l, \quad l \in \Gamma_j \\
    0 & \text{otherwise}.
    \end{array}\right. 
\]
In the following, the left column presents the V-I relations for the continuous-time circuit
and the right column presents the discretization with stepsize $\tfrac{1}{2}$, recovering the standard PG-EXTRA:
% \vspace{-0.2in}
\vspace{-2mm}
\begin{center}{
\small
\begin{tabular}{l|l}
\parbox{0.3\linewidth}{%
\begin{align*}
x_j &= \prox_{Rf_j}\left (\sum_{l=1}^N W_{jl} x_l - R\nabla h_j(x_j) - w_j\right ) \\
\frac{d}{dt} w_j &= x_j - \sum_{l=1}^N W_{jl}x_l
\end{align*}}
&
\parbox{0.3\linewidth}{%
\begin{align*}
x_j^{k+1} &= \prox_{Rf_j}\left (\sum_{l=1}^N W_{jl} x_l^k - R\nabla h_j(x_j^k) - w_j^k\right ) \\
w_j^{k+1} &= w_j^k + \frac{1}{2}(x_j^k - \sum_{l=1}^N W_{jl}x_l^k)
\end{align*}}
\end{tabular} }
\end{center}
for every node $j=1, \ldots, N$ and every edge $(j,l)$ in graph $G$.

% \BEAS\label{e-pgextra-vi-continuous}
% x^{k+1} &=& \prox_{Rf}\left (W x^k - R\nabla h(x^k) - w^k\right ) \\
% w^{k+1} &=& w^k + \frac{1}{2}(I - W)x^k,
% \EEAS
% using the decentralized notation of \cite[\S11.3]{RyuYin2022_largescale}.

\section{Automatic discretization}\label{sec-auto-discretization}
% Recall that any electric circuit has
% energy function,
% \ie, a total (shifted) energy stored in all inductors and capacitors~\eqref{e-total-energy}.
% Using energy dissipation analysis we show that (under the assumptions of Theorem~\ref{thm:convergence})
% dynamic interconnect necessarily 
% reaches equilibrium in continuous time.
We discretize the continuous-time dynamics given by the circuit with an admissible dynamic interconnect using a two-stage Runge--Kutta method with parameters $\alpha, \beta$ and stepsize $h>0$.
The explicit form of the discretization is stated in \S\ref{appendix-automatic-discr} of the appendix.
Let $\{(v^k, i^k, x^k, y^k)\}_{k=1}^\infty$ be the iterates generated by the discretized algorithm.
Then the energy stored in the circuit at time $t=kh$ is 
\[
\mathcal{E}_k = \frac{1}{2}\|v_\mathcal{C}^{k}-v^\star_\mathcal{C}\|^2_{D_\mathcal{C}}+
\frac{1}{2}\|i_\mathcal{L}^k-i^\star_\mathcal{L}\|^2_{D_\mathcal{L}}.
\]
To guarantee convergence of the discretized algorithm, we search for discretization parameters 
that ensure the $\mathcal{E}_1,\mathcal{E}_2,\dots$ sequence is dissipative in the following sense.
Specifically, we say the algorithm or the discretization is \emph{sufficiently dissipative} if there is an $\eta>0$ such that
\BEQ\label{e-suff-diss}
\mathcal{E}_{k+1}- \mathcal{E}_k 
 + \eta\langle x^k-x^\star, y^k-y^\star\rangle 
\leq 0,
\EEQ 
holds for all $k=1, 2, \ldots$. 
This requirement is analogous to the ``sufficient decrease'' conditions in optimization~\cite{boyd2004convex, NocedalWright1999_numerical}. 
The following Lemma~\ref{lem:auto-discr}, which proof we provide in \S\ref{appendix-automatic-discr} of the appendix, states that sufficient dissipativity ensures convergence under suitable conditions.
\begin{lemma}    
\label{lem:auto-discr}
Assume $f\colon \reals^m \to \reals \cup \{ \infty \}$ is a strictly convex function 
and the dynamic interconnect is admissible.
If the two-stage Runge--Kutta discretization, as explicitly stated in \S \ref{appendix-automatic-discr} of the appendix, generates a discrete-time sequence
$\{(v^k, i^k, x^k, y^k)\}_{k=1}^\infty$ satisfying the sufficient dissipativity condition \eqref{e-suff-diss}, then $x^k$ converges to a primal solution.
\end{lemma}

We find such a discretization with the following automated methodology.
Given a discretization characterized by $(\alpha,\beta,h)$, the dissipativity condition \eqref{e-suff-diss} for a given $\eta>0$ is implied if the optimal value of the following optimization problem is non-positive:
\BEQ\label{eq-worst-case-descent}
\begin{array}{ll}
\mbox{maximize} &
    \mathcal{E}_{2} - \mathcal{E}_1  +  \eta
\langle x^1-x^\star, y^1-y^\star\rangle\\
    \mbox{subject to}
& \mathcal{E}_s = \frac{1}{2}\|v_\mathcal{C}^{s}-v^\star_\mathcal{C}\|^2_{D_\mathcal{C}} +
\frac{1}{2}\|i_\mathcal{L}^s-i^\star_\mathcal{L}\|^2_{D_\mathcal{L}}, \quad s \in \{1, 2\} \\
& (v^1, i^1, x^1, y^1) \text{ is feasible initial point} \\
& (v^2, i^2, x^2, y^2) \text{ is generated by discrete optimization method from initial point} \\
& f \in \mathcal{F},
\end{array}
\EEQ
where $f, v^1, i^1, x^1, y^1, v^\star, i^\star, x^\star, y^\star$
are the decision variables
and $\mathcal{F}$ is a family of functions (\eg, $L$-smooth convex) that the algorithm is to be applied to.
Here, we are using the fact that \eqref{e-suff-diss} is homogeneous with respect to $k$ (\ie, \eqref{e-suff-diss} essentially has no $k$-dependence), and therefore it is sufficient to verify the condition for $k=1$ but for all feasible initial points $(v^1, i^1, x^1, y^1)$.
It turns out that \eqref{eq-worst-case-descent} can be solved exactly as a semidefinite program (SDP) for many commonly considered function classes $\mathcal{F}$.
This technique was initially proposed as the  performance estimation
problem (PEP) \cite{DroriTeboulle2014_performance, 
TaylorHendrickxGlineur2017_smooth}, a computer-aided methodology for
constructing convergence proofs of first-order optimization methods.
See, \eg, PEPit~\cite{pepit2022} package that implements PEP in Python.

Further, \eqref{eq-worst-case-descent} can be posed as a nonconvex quadratically constrained quadratic problem (QCQP) with only a few tens of
variables and such problems can be solved exactly with spatial branch-and-bound algorithms \cite{Gurobi,locatelli2013global,horst2013global,liberti2008introduction,DasGuptaVanParysRyu2023_branchandbound}.

% XXX we use Ipopt not QCQP BnB XXX
In conclusion, we can solve a non-convex QCQP to find a provably convergent discretization of the continuous-time circuit with an admissible dynamic interconnect.
We use the Ipopt~\cite{wachter2006implementation,andersson2019casadi} solver. 
Further details are provided in \S\ref{appendix-automatic-discr} of the appendix.

\paragraph{Example.}
Consider the following example circuit for 
the minimization of a convex function $f$. Let $R_1=R_2=R_3=1$, and $C_1=C_2=10$.
    \begin{center} 
    \CircuitExampleMain
    \end{center}
%     \vspace{-0.5cm}
With our automatic discretization methodology, we find the sufficiently dissipative parameters
\[
 \eta=6.66, \qquad h=6.66, \quad \alpha=0, \quad \beta=1.
\]
The resulting provably convergent algorithm is 
\BEAS
x^k &=& \prox_{(1/2) f}(z^k ),\quad  y^k=2(z^k-x^k)\\
w^{k+1} &=& w^k - 0.33(y^k + 3w^k) \\
z^{k+1} &=& z^k - 0.16(5 y^k + 3w^k).
\EEAS
is provably convergent%
\footnote{Our pipeline has a final verification stage that numerically checks whether point returned by the Ipopt solver is indeed feasible for the small QCQP. Strictly speaking, our theoretical convergence guarantee relies on the correctness of this numerical verification of feasibility.}
under the condition that $f$ is strictly convex, see \S\ref{appendix-ciropt} for details.

% The V-I relations for the circuit (left column) and convergent discretized method
% found by our method (right) are displayed below.
% \begin{center}{
% \small
% \begin{tabular}{l|l}
% \parbox{0.3\linewidth}{%
% \begin{align*}
% x &= \prox_{(R/2) f}\left (z\right ),\quad  y=\frac{2}{R}(z-x)\\
% \frac{d}{dt} e_2 &= - \frac{1}{2CR}(R y + 3e_2) \\
% \frac{d}{dt} z &= - \frac{1}{4CR}(5R y + 3e_2)
% \end{align*}}
% &
% \parbox{0.3\linewidth}{%
% \begin{align*}
% x^k &= \prox_{(R/2) f}\left (z^k\right ),\quad  y^k=\frac{2}{R}(z^k-x^k)\\
% e_2^{k+1} &= e_2^k - \frac{h}{2CR}(R y^k + 3e_2^k) \\
% z^{k+1} &= z^k - \frac{h}{4CR}(5R y^k + 3e_2^k).
% \end{align*}}
% \end{tabular} }
% \end{center}

% \newpage

\begin{wrapfigure}{r}{0.41\linewidth} \vspace{-8mm}
\vspace{-0.2in}
\begin{center}
        \includegraphics[width=0.4\textwidth]{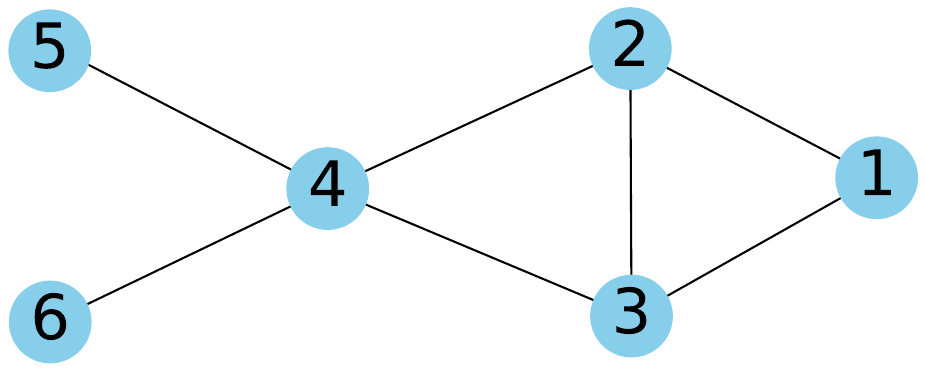}
\end{center} \vspace{-2mm}
    % \caption{Example of a static interconnect with $m=5$, $N_1=\{1,3\}$, $N_2=\{2,4\}$, $N_3=\{5\}$. }
    \caption{Underlying graph $G$.}
    \label{fig-graph} %\vspace{-10mm}
\end{wrapfigure}

\section{Experiments} \label{sec-experiment}
In this section, we use our methodology to obtain a new algorithm and experiment with it on a specific problem instance.
Consider a decentralized optimization problem with a communication graph $G$ with $N=6$ nodes and $7$ edges, as shown in Figure~\ref{fig-graph}.
Specifically, we consider the optimization problem
\[
\begin{array}{ll}
 \underset{x \in \textbf{R}^{100}}{\mbox{minimize}}&
    \sum_{i \in \{4, 5\}} \left (\| x - b_i\|_2 + \|x - b_i\|_2^2\right) + \sum_{i \notin \{4, 5\}} \| x - b_i\|_2,
\end{array}
\]
where each agent $i \in \set{ 1, \dots, 6 }$ holds the vector $b_i \in \reals^{100}$.
To leverage the strong convexity of $f_4$ and $f_5$, we propose a modification to the DADMM circuit described in \S\ref{s-dadmm}. 
Given that a circuit with a capacitor and inductor corresponds to a momentum method (see \S\ref{s-nesterov}), 
and momentum is known to accelerate convergence for strongly convex functions~\cite{Polyak1964_methods}, we add a capacitor to $e_{45}$ to DADMM circuit as shown in the left column of Figure~\ref{fig:dadmm_c_n6}.
We then discretize the circuit and refer the the resulting algorithm DADMM+C.
We apply DADMM+C to the decentralized optimization problem and observe a speedup as shown in the right columns of Figure~\ref{fig:dadmm_c_n6}. 
The relative error for DADMM+C decreases to $10^{-10}$ in 
$66$ iterations, for DADMM in $87$ iterations
and for P-EXTRA in $294$ iterations.
For further details, see \S\ref{sec-dadmm-c} of the appendix.

% We compare the convergence of DADMM+C with DADMM and P-EXTRA
% in Figure~\ref{fig:dadmm_c_n6}. 

\begin{figure} [H]
    \centering
    \!\!\!\!\!\!\!\!\!\!\!\!\!\!\!\!
    \begin{subfigure}[c]{0.51\textwidth}
        \centering
\CircuitDADMMCMain
    \end{subfigure}
    \hfill
    \!\!\!\!\!\!\!\!\!\!\!\!\!\!\!\!\!\!
    \begin{subfigure}[c]{0.55\textwidth}
        \centering
        \vspace{-0.2in}
        \includegraphics[width=1\textwidth]{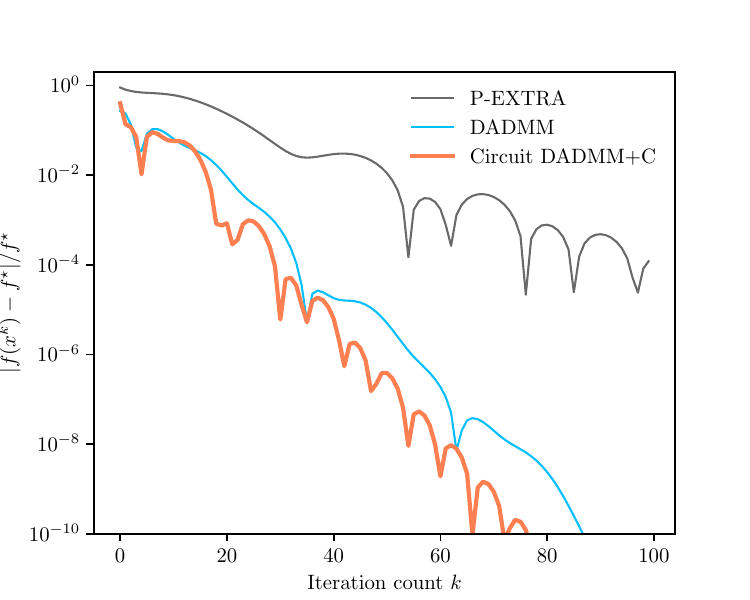}
    \end{subfigure}
    \caption{(Left) Circuit of DADMM+C. Compared to the DADMM circuit of \S\ref{s-dadmm}, the DADMM+C circuit has an additional capacitor.
    (Right) Relative error $\left|f(x^k) - f^\star\right|/{f^\star}$ vs.\ iteration count.
    }
    \label{fig:dadmm_c_n6}
\end{figure}

Further, we define a general version of the DADMM+C method for any connected graph and establish a general convergence proof in Lemma~\ref{lem:dadmm-c-convergence} of in \S\ref{ss:dadmm-c-convergence} of the appendix. This convergence analysis demonstrates how to use our methodology to discover a new family of methods with a classical convergence proof. Finally, we provide another set of similar experiments in \S\ref{s:pg-extra-c} of the appendix.

\section{Conclusion}
In this work, we present a novel approach to optimization algorithm design using ideas from electric RLC circuits. The continuous-time RLC circuit models combined with the automatic discretization method provide a foundation for designing algorithms that inherently possess convergence guarantees. Further, we provide code implementing the automatic discretization. Our framework opens the door to future research by applying this methodology to a broader range of optimization problems and extending the problem to other setups, such as the stochastic optimization setup.

% \newpage

\section*{Acknowledgments and Disclosure of Funding}

This work was supported by the Samsung Science and Technology Foundation (Project Number SSTF-BA2101-02), the National Research Foundation of Korea (NRF) grant funded by the Korean government (No.RS-2024-00421203, RS-2024-00406127),
and the Oliger Memorial Fellowship. We thank Hangjun Cho for the helpful discussions on the continuous-time analysis. We also thank anonymous reviewers for the highly constructive feedback.

% \newpage

\bibliography{references}
\bibliographystyle{abbrv}

\newpage

% \section*{Supplementary Material}

% Authors may wish to optionally include extra information (complete proofs, additional experiments and plots) in the appendix. All such materials should be part of the supplemental material (submitted separately) and should NOT be included in the main submission.

\appendix
\section{Prior works}
\label{s:prior_works}

\paragraph{Distributed optimization as RLC circuits.} 
This work started as a lecture for the Stanford University EE 364b class given in 2010 \cite{Boyd2010_distributed}.
The lecture proposed the idea of relating distributed optimization algorithms to the dynamics of RLC circuits.
% for example, relating proximal with resistor can be novel idea of professor Boyd
Different from the prior studies \cite{Dennis1959_mathematical, Wilson1986_quadratic, ChuaLin1984_nonlinear, VichikBorrelli2014_solving, SawantNguyenLiuPoonDhople2024_hybridcomputing},
that consider solving specific optimization problems through implementing physical circuits, our focus is on using insights from circuit theory to design new algorithms, without any consideration of implementing physical circuits.
The follow-up works \cite{YuAcikmese2020_rc, YuAcikmese2020_rlc, AgarwalFisckoKarPileggiSinopoli2023_equivalent, AgarwalPileggi2023_equivalent},
have built upon this setup \cite{Boyd2010_distributed}.

\paragraph{Optimization algorithms from continuous-time dynamics.}
Relating continuous-time dynamics described by ordinary differential equation (ODE) with optimization algorithm is a technique with a long history \cite{Bruck1975_asymptotic, HelmkeMoore1996_optimization, Alvarez2000_minimizing, SchroppSinger2000_dynamical, Fiori2005_quasigeodesic}.  
The continuous-time dynamics related to Polyak's heavy ball method \cite{Polyak1964_methods} were studied by \cite{AttouchAlvarez2000_heavy, AttouchGoudouRedont2000_heavy, AlvarezAttouchBolteRedont2002_secondorder, AttouchCzarnecki2002_asymptotic}. 
The ODE model for Nesterov acceleration \cite{Nesterov1983_method} was introduced by \cite{SuBoydCandes2014_differential, SuBoydCandes2016_differential},
analyses for generalized cases were followed by \cite{ApidopoulosAujolDossal2018_differential, AttouchChbaniPeypouquetRedont2018_fast, AttouchChbaniRiahi2019_rate}, and the ODE model for Nesterov acceleration for strongly convex function (NAG-SC) was introduced in \cite{WilsonRechtJordan2021_lyapunov}. 
Together with \cite{KricheneBayenBartlett2015_accelerated}, the studies by  \cite{SuBoydCandes2014_differential, SuBoydCandes2016_differential} initiated continuous-time analyses of accelerated first-order methods and inspired much follow-up works such as \cite{WibisonoWilsonJordan2016_variational, AttouchPeypouquet2019_convergence, MuehlebachJordan2021_optimization, DiakonikolasJordan2021_generalized, EvenBerthierBachFlammarionHendrikxGaillardMassoulieTaylor2021_continuized, WilsonRechtJordan2021_lyapunov, AttouchLaszlo2020_newtonlike, BotHulett2022_seconda, SuhRohRyu2022_continuoustime, KimYang2023_unifying, SuhParkRyu2023_continuoustime, BotCsetnekNguyen2023_fast}. 
As a further refined continuous-time model preserving more information from the discretization, the high-resolution ODE for NAG-SC was introduced in  \cite{ShiDuJordanSu2021_understanding}, and was further developed by \cite{Lu2022_os}. 

In addition to accelerated methods, various topics and methods in optimization have been studied in a continuous-time framework. 
Continuous-time dynamics related to splitting methods were studied by
\cite{AttouchPeypouquetRedont2014_dynamical, AbbasAttouch2015_dynamical, BotCsetnek2018_convergence, CsetnekMalitskyTam2019_shadow, FrancaRobinsonVidal2021_gradient, Hassan-MoghaddamJovanovic2021_proximal}. 
\cite{FrancaRobinsonVidal2018_admm} studied continuous-time dynamics of ADMM \cite{GlowinskiMarroco1975_lapproximation, GabayMercier1976_dual, Gabay1983_chapter, EcksteinBertsekas1992_douglas, BoydParikhChuPeleatoEckstein2011_distributed}, and provided an accelerated ADMM by discretizing the ODE model combined with \cite{SuBoydCandes2014_differential}. The analyses were furthermore generalized to differential inclusions by \cite{YuanZhouLiSun2019_differential, FrancaRobinsonVidal2023_nonsmooth}. 
There are numerous works of continuous-time analyses for distributed optimization,  \cite{WangElia2010_control, LuTang2012_zerogradientsum, GharesifardCortes2014_distributed, KiaCortesMartinez2015_distributed, LinRenFarrell2017_distributed} to name a few, and we refer the readers to the survey paper \cite{YangYiWuYuanWuMengHongWangLinJohansson2019_survey} for a comprehensive overview.

\paragraph{Computer-assisted analysis of optimization algorithms.}
There has been lines of work automating the analysis of optimization methods using semidefinite programs (SDP). 
One line of work is performance estimation problems (PEP) introduced by \cite{DroriTeboulle2014_performance}, which provides a systematic way to obtain worst-case performance guarantees of a given fixed-step first-order method. 
The range and technique of utilizing PEP have been further developed by \cite{TaylorHendrickxGlineur2017_smooth, TaylorBach2019_stochastic, RyuTaylorBergelingGiselsson2020_operator, MoucerTaylorBach2023_systematic, KimYang2023_convergencea}, and many efficient algorithms with tight analyses utilizing PEP are discovered \cite{KimFessler2016_optimized, Lieder2021_convergence, Kim2021_accelerated, KimFessler2021_optimizing, YoonRyu2021_accelerated, LeeKim2021_fast, ParkRyu2022_exact, GorbunovLoizouGidel2022_extragradient, taylor2023optimal, JangGuptaRyu2023_computerassisted, BarreTaylorBach2023_principled, YoonKimSuhRyu2024_optimal}.

Another line of work is an approach adapting integral quadratic constraints (IQC) \cite{MegretskiRantzer1997_system}. 
IQCs are a powerful analysis method in control theory for analyzing interconnected dynamical systems with nonlinear feedback. This approach was first adapted for analyzing first-order optimization algorithms by \cite{LessardRechtPackard2016_analysis} and followed by \cite{FazlyabRibeiroMorariPreciado2018_analysis}. 
Analyses based on IQC have lead to tight bounds for well-known algorithms \cite{NishiharaLessardRechtPackardJordan2015_general, HuLessard2017_dissipativity}. IQC has also been utilized to develop new fast algorithms with tight convergence rates \cite{VanScoyFreemanLynch2018_fastest, CyrusHuVanScoyLessard2018_robust, SundararajanVanScoyLessard2020_analysis, SimonMichalowskyEbenbauer2021_robust}.

Recently, an extension of PEP to leveraging quadratic constrained quadratic programs (QCQP) was introduced by \cite{DasGuptaVanParysRyu2023_branchandbound}. 
Treating the step-sizes as optimization variables, this work furthermore provides systematic computer-assisted methodology to optimize the step-sizes. 
Our work adapts this approach to finding appropriate discretizations. To the best of our knowledge, our proposal is the first instance of using computer-assisted methodologies to find discretizations of continuous-time dynamics.

\paragraph{Physics-bases approaches to designing optimization algorithms.} 
Optimization methods obtained by discretizing conformal Hamiltonian dynamics \cite{McLachlanPerlmutter2001_conformal} were considered by \cite{MaddisonPaulinTehODonoghueDoucet2018_hamiltonian}. 
Studying structure-preserving discretizations for conformal (dissipative) Hamiltonian systems, \cite{FrancaSulamRobinsonVidal2020_conformal, FrancaSulamRobinsonVidal2020_conformala} analyzed symplectic structure of Nesterov and heavy ball, 
and introduced Relativistic Gradient Descent (RGD) by adopting ideas from special relativity. 
Based on relativistic Born-Infeld (BI) dynamics,   
\cite{DeLucaSilverstein2022_borninfeld} considered a class of frictionless, energy-conserving system and introduced Bouncing BI (BBI) algorithm as a discretization. 

Our work is based on nonlinear resistive electric circuits, the study of which dates back to \cite{Duffin1946_nonlinear}. 
The stationary condition for nonlinear networks were considered by \cite{Millar1951_cxvi}, generalizing theorems of Maxwell \cite{Maxwell1873_treatise} for linear networks. 
The study of nonlinear resistive networks influenced the refinement of the concept of maximal monotonicity~\cite{Minty1960_monotone, Minty1961_maximal, Minty1961_solving}, which is now a fundamental concept in convex optimization. 
Well-posedness of the solutions for nonlinear networks was studied by \cite{DesoerKatzenelson1965_nonlinear, DesoerWu1974_nonlinear}, but only for one-descent nonlinear resistors. 
Recently, the study of nonlinear electrical circuits was revisited by \cite{ChaffeySepulchre2021_monotone, ChaffeyPadoan2022_circuit, ChaffeyBanertGiselssonPates2023_circuit, ChaffeySepulchre2024_monotone} using contemporary methods of convex optimization. However, their main focus was on circuits, not on designing new optimization algorithms. 
% Our work is inspired by \cite{Boyd2010_distributed}
To the best of our knowledge, our work is the first to introduce a generalized framework for designing optimization methods based on electric circuits.

\paragraph{Discretization.} 
Continuous-time analyses of optimization algorithms must eventually contend with the issue of discretizing the dynamics into a discrete-time algorithm.
Discretization of differential equations is a subject of numerical analysis, and it has a long history, even dating back to Euler \cite{Euler1755_institutiones}. 
Standard discretization schemes such as Euler, Runge--Kutta \cite{Runge1895_ueber, Kutta1901_beitrag} and symplectic integrators \cite{Vogelaere1956_methods, Ruth1983_canonical, Feng1984_difference}, have a rich body of research analyzing their convergence \cite{HairerLubichGerhard2006_geometric, Iserles2009_first} for example. However, these theories in numerical analysis primarily focus on the convergence of the discretized sequence to the trajectory of the solution flow in differential equations throughout a finite time-interval, which differs from the focus of optimization.
Therefore, directly applying standard discretization schemes from numerical analysis does not ensure convergence to the optimality criteria of interest in optimization, such as function value or optimal point convergence. 

In optimization, the study of discretization can broadly be divided into two categories. 
One involves applying standard discretization schemes or their variants, and the other provides special rules tailored to the specific dynamics of interest. 
As previously discussed, the former cases can only guarantee the convergence involving certain errors \cite{BetancourtJordanWilson2018_symplectic, FrancaJordanVidal2021_dissipative}, or introduce specific and limited cases they can cover \cite{ScieurRouletBachdAspremont2017_integration, ZhangMokhtariSraJadbabaie2018_direct, ShiDuSuJordan2019_acceleration, MuehlebachJordan2019_dynamical, ZhangSraJadbabaie2019_acceleration, SuhRohRyu2022_continuoustime, UshiyamaSatoMatsuo2023_unified}. 
The latter type of works do provide discretization rules with analytic proofs for certain families of ODEs \cite{AlvarezAttouch2001_inertiala, SuBoydCandes2016_differential, WibisonoWilsonJordan2016_variational, AttouchChbaniRiahi2019_fast, WilsonMackeyWibisono2019_accelerating, AdlyAttouch2020_finite, AttouchChbaniFadiliRiahi2020_firstorder, DiakonikolasJordan2021_generalized, BotCsetnekNguyen2023_fast}, but cannot be applied to general cases.  
Of course, both approaches have brought significant advances in obtaining new methods from continuous-time dynamics, however, it is still true that previous approaches cannot immediately applied the new ODEs that emerge from our framework. 
To the best of our knowledge, our work is the first to propose to automate the process of finding a discretized method from ODE using computer-assisted tools.
\section{Proof of Theorem~\ref{thm-well-posedness}} \label{s-well-posedness}

To prove Theorem~\ref{thm-well-posedness}, it is sufficient to consider the cases without $0$-ohm resistors and furthermore all resistor, inductance, capacitance values are $1$. 
We first state the theorem for such cases, which implies Theorem~\ref{thm-well-posedness}. 
\begin{theorem} \label{thm-well-posedness-for-reduced} 
Let $f \colon \reals^m \to \reals^m$ be a $\mu$-strongly convex and $M$-smooth function and 
$B \colon \reals^\Bdom \to \reals^\Bran$ be a matrix. %and $D_{\mR} \colon \reals^{\rho} \to \reals^{\rho}$ be a diagonal matrix with nonzero elements. 
Suppose $(v^0, i^0, x^0, y^0)$ satisfy
\begin{align} \label{eq-reduced-initial}
    \bmat{ i^0 \\ y^0 } \in \mathcal{N}(B),  \quad
    \bmat{ v^0 \\ x^0 } \in \mathcal{R}(B^{\intercal}), \quad
    v_{\mathcal{R}}^0 = i_{\mathcal{R}}^0, \quad
    y^0 = \nabla f(x^0).
\end{align}
Then there is a uniquely determined Lipschitz continuous curve  
$(v, i, x, y) \colon [0,\infty) \to \reals^{2\Bran}$
satisfies 
\begin{align} \label{eq-reduced-dynamics}
    \bmat{ i \\ y } \in \mathcal{N}(B),  \quad
    \bmat{v \\ x } \in \mathcal{R}(B^{\intercal}), \quad 
    y = \nabla f(x), \quad
    v_{\mathcal{R}} = i_{\mathcal{R}}, \quad
    v_{\mathcal{L}} = \frac{d}{dt} i_{\mathcal{L}}, \quad
    i_{\mathcal{C}} = \frac{d}{dt} v_{\mathcal{C}},
\end{align}
for all $t\in(0,\infty)$ and the initial condition $(v(0), i(0), x(0), y(0)) = (v^0, i^0, x^0, y^0)$.
\end{theorem}

\begin{lemma} \label{lem-enough-to-show-reduced}
    Theorem~\ref{thm-well-posedness-for-reduced} implies Theorem~\ref{thm-well-posedness}. 
\end{lemma}

\begin{proof}

\emph{(i) KCL, KVL and V-I relations for equivalent dynamics without $0$-ohm resistors.} \\
We first consider the equivalent dynamic interconnect without $0$-ohm resistors. 
As $0$-ohm resistors are ideal wires, from basic circuit theory we know the nodes connected by $0$-ohm resistors can be considered as a single node. 
We find the expression for KCL, KVL and V-I relations for the equivalent dynamic interconnect composed with $\partial f$. 
The equivalent expression for KCL and KVL can be considered as consequence of Tellegen's theorem in \citep[\S10.2.3]{DesoerKuh1969_basic}, however, we write the detail here to make it self-contained. 

Observe, KCL and KVL can be equivalently written as
\begin{align*}
    \bmat{ A \,\,\Big|\,\, \mat{ I_m  \\ 0 } } \bmat{ i \\ y } = 0, 
    \qquad \bmat{ v \\ x } 
    % = \bmat{ A^\intercal \\  \overline{ \mat{ I_m  \,\, 0 } } } \bmat{ x \\ e }
    = \bmat{ A \,\,\Big|\,\, \mat{ I_m  \\ 0 } }^\intercal \bmat{ x \\ e }.
\end{align*}
We furthermore restrict the values to satisfy Ohm's law for $0$-ohm resistors, \ie, the potential values of two nodes connected to a $0$-ohm resistor is identical. 

Let's first focus on KCL, the left equation. 
Suppose node $j$ and $l$ are connected with $0$-ohm resistor named as $\mR_{jl}$. 
% Without loss of generality, we may assume node $j$ is not connected to any of the ports. 
Suppose the $k$'th column of $A$ corresponds to $\mR_{jl}$. 
Eliminating $\mR_{jl}$ corresponds to eliminating the $k$'th column of $A$ and eliminating $i_{\mR_{jl}}$ from $i$. 
However, if we just directly eliminate them, as $i_{\mR_{jl}}$ may not be zero, the equation will no longer be satisfied. 
We need to keep the information that currents flowing into node $j$ (except for $-i_{\mR_{jl}}$) flows to node $l$. 
As we do not permit ideal wire loop, without loss of generality we may assume node $j$ is not the ground node.

To preserve the information, when node $l$ is not the ground node, we add the $j$'th row of {\footnotesize $\bmat{ A \,\,\Big|\,\, \mat{ I_m  \\ 0 } }$} to the $l$'th row. 
Then $k$'th component of the $l$'th row becomes $0$, thus the equation corresponding to the $l$'th row will still be satisfied after eliminating the $k$'th column and $\mR_{jl}$. 
When node $l$ is the ground node, skip the row addition. 
% Now eliminate the $j$'th row and the components (and corresponding columns) connected to both node $j$ and $l$. 
Now eliminate the $j$'th row. 
% Note that columns are eliminated only from $A$. 
Note that column is eliminated only from $A$. 

We now move on to KVL.         
Eliminating a column of $A$ and a component in $i$ corresponds to eliminating a row of $A^{\intercal}$ and a component in $v$. 
This conserves the validity of the equation. 
Next, the row operation for {\footnotesize $\bmat{ A \,\,\Big|\,\, \mat{ I_m  \\ 0 } }$} corresponds to column operation for {\footnotesize $\bmat{ A \,\,\Big|\,\, \mat{ I_m  \\ 0 } }^{\intercal} = \bmat{ A^\intercal \\  \overline{ \mat{ I_m  \,\, 0 } } }$}. 
% The values in {\footnotesize$\bmat{ x \\ e }$} corresponding to column $j$ and $l$ coincides, as they are potential values of the two nodes connected with ideal wire. 
Recall we've restricted the potential values of the nodes connected with $0$-ohm resistor to be same, values in {\footnotesize$\bmat{ x \\ e }$} corresponding to column $j$ and $l$ coincide. 
Thus when node $l$ is not the ground node, adding $j$'th column to the $l$'th column and eliminating $j$'th component in $\bmat{ x \\ e }$, will not change the values on the left hand side. 
When node $l$ is the ground node, the same argument holds by skipping the column addition. 

Repeat this process until there is no $0$-ohm resistors. 
Name the reduced matrix as $\tilde{B}$ and reduced current as $\tilde{i}$. 
Then KCL reduces to $\tilde{B} \bmat{ \tilde{i} \\ y } = 0$ and KVL reduces to $\bmat{ v \\ x }  = \tilde{B}^\intercal \bmat{ \tilde{x} \\ \tilde{e} }$, or equivalently $\bmat{ \tilde{v} \\ x } \in \mathcal{R}(\tilde{B}^{\intercal})$. 

Now name the reduced diagonal matrices $\tilde{D}_\mR$, $\tilde{D}_\mL$ and $\tilde{D}_\mC$ as the reduced matrices that without the entries corresponding to eliminated components. 
Note $\tilde{D}_\mR$ has no zero diagonal entries. 
Then KCL, KVL and V-I relations for the equivalent dynamic interconnect composed with $\partial f$ become as follows
\begin{align} \label{eq-no-0-ohm-dynamics}
    &\bmat{ \tilde{i} \\ y } \in \mathcal{N}(\tilde{B}),  \quad
    \bmat{ \tilde{v} \\ x } \in \mathcal{R}(\tilde{B}^{\intercal}), \quad
    y = \nabla f(x), \quad
    \tilde{v}_{\mathcal{R}} = \tilde{D}_\mathcal{R} \tilde{i}_{\mathcal{R}}, \quad
    \tilde{v}_{\mathcal{L}} = \tilde{D}_\mathcal{L} \frac{d}{dt} \tilde{i}_{\mathcal{L}}, \quad
    \tilde{i}_{\mathcal{C}} = \tilde{D}_\mathcal{C} \frac{d}{dt} \tilde{v}_{\mathcal{C}}.
\end{align}
As an equivalent dynamics, it is enough to prove the curve that satisfies \eqref{eq-no-0-ohm-dynamics}
and the initial condition $(\tilde{v}(0), \tilde{i}(0), x(0), y(0)) = (\tilde{v}^0, \tilde{i}^0, x^0, y^0)$ with condition  
\begin{align} \label{eq-0-ohm-initial}
    \bmat{ \tilde{i}^0 \\ y^0 } \in \mathcal{N}(\tilde{B}),  \quad
    \bmat{ \tilde{v}^0 \\ x^0 } \in \mathcal{R}(\tilde{B}^{\intercal}), \quad
    \tilde{v}_{\mathcal{R}}^0 = \tilde{D}_\mathcal{R} \tilde{i}_{\mathcal{R}}^0, \quad
    y^0 = \nabla f(x^0)
\end{align}  
is unique and Lispchitz continuous.

\emph{(ii) Sufficient to consider only the cases with $\tilde{D}_\mR$, $\tilde{D}_\mL$ and $\tilde{D}_\mC$ are identity matrices.} \\
For a dynamic interconnect composed with $\partial f$, consider the equivalent dynamics without $0$-ohm resistors. 
Let $\tilde{B}$ be the matrix in \eqref{eq-no-0-ohm-dynamics} for the dynamics, and let $\mathcal{K}$ be the number of columns of $\tilde{B}$. 
Suppose $(\tilde{v}^0, \tilde{i}^0, x^0, y^0)$ satisfy \eqref{eq-0-ohm-initial}. 
Define the diagonal matrix
\[
    P = \diag\Big(
    \sqrt{ \tilde{D}_\mathcal{R}^{-1} },
    \sqrt{ \tilde{D}_\mathcal{L}^{-1} },
    \sqrt{ \tilde{D}_\mathcal{C}}, 
    I_m
    \Big),
\]
and define $B=\tilde{B} P$. 
% Since all diagonal elements are positive, such $P$ is well-defined diagonal matrix with positive elements. 
Define ${i}^0$ and ${v}^0$ to satisfy $\bmat{ {i}^0 \\ y^0 } = P^{-1} \bmat{ \tilde{i}^0 \\ y^0 }$ and $\bmat{ {v}^0 \\ x^0 } = P \bmat{ \tilde{v}^0 \\ x^0 }$. 
Then $({v}^0, {i}^0, x^0, y^0)$ satisfies \eqref{eq-reduced-initial} since 
\BEAS
    \tilde{B} \bmat{ \tilde{i}^0 \\ y^0 } = 0 &\iff&
    B \bmat{ i^0 \\ y^0 } = (\tilde{B} P) \pr{ P^{-1} \bmat{ \tilde{i}^0 \\ y^0 } } = 0,  \\
    \qquad
     \exists z^0, \,\, \bmat{ \tilde{v}^0 \\ x^0 } = \tilde{B}^{\intercal} z^0 &\iff& 
    \exists z^0, \,\, \bmat{ v^0 \\ x^0 } = P \bmat{ \tilde{v}^0 \\ x^0 } = P\tilde{B}^{\intercal} z = Bz^0,
\EEAS
and
\[
    \tilde{v}_{\mathcal{R}}^0 = \tilde{D}_\mathcal{R} \tilde{i}_{\mathcal{R}}^0
    \iff \sqrt{ \tilde{D}_\mathcal{R} } \tilde{v}_{\mathcal{R}}^0 = \sqrt{ \tilde{D}_\mathcal{R}^{-1} } \tilde{i}_{\mathcal{R}}^0
    \iff v_{\mathcal{R}}^0 = i_{\mathcal{R}}^0.
\]
Then by Theorem~\ref{thm-well-posedness-for-reduced}, there is a Lipschitz continuous curve  
$(v, i, x, y) \colon [0,\infty) \to \reals^{2\Bran}$ that satisfies \eqref{eq-reduced-dynamics} for all $t\in(0,\infty)$ and the initial condition $(v(0), i(0), x(0), y(0)) = (v^0, i^0, x^0, y^0)$. 
Define $\tilde{i}$ and $\tilde{v}$ to satisfy $\bmat{ \tilde{i} \\ y } = P \bmat{ i \\ x }$ and $\bmat{ \tilde{v} \\ y } = P^{-1} \bmat{ v \\ y }$. 
Then $\tilde{i}$ and $\tilde{v}$ are Lipschitz continuous as well, as they are composition of linear operators and Lipschitz continuous curves. 
Furthermore, we can check \eqref{eq-no-0-ohm-dynamics} and the initial condition $(\tilde{v}(0), \tilde{i}(0), x(0), y(0)) = (\tilde{v}^0, \tilde{i}^0, x^0, y^0)$ is satisfied, with the similar argument above. 

Reversing the arguments, the uniqueness can be obtained since $P$ is invertible and thus $(v, i) \mapsto (\tilde{v},\tilde{i})$ is bijective. This concludes the proof. 
% Conversely, suppose $(\tilde{v}, x, \tilde{i}, y)$ is a curve that satisfies \eqref{eq-no-0-ohm-dynamics} with initial condition $(\tilde{v}^0, \tilde{i}^0, x^0, y^0)$ that satisfy \eqref{eq-reduced-initial}. 
% Then if we define ${i}$ and ${v}$ to satisfy $\bmat{ i \\ x } = P^{-1}\bmat{ \tilde{i} \\ y } =$ and $\bmat{ v \\ y } = P\bmat{ \tilde{v} \\ y }$ and similarly define ${i}^0$ and ${v}^0$, then with the similar argument above, $(v, i, x, y)$ becomes a curve that satisfies \eqref{eq-reduced-dynamics} and $(v(0), i(0), x(0), y(0)) = (v^0, i^0, x^0, y^0)$. 
% However, such curve is unique by Theorem~\ref{thm-well-posedness-for-reduced}, therefore it is the solution 
    
\end{proof}

By Lemma~\ref{lem-enough-to-show-reduced}, our goal has reduced to Theorem~\ref{thm-well-posedness-for-reduced}. 
We will establish the well-posedness for $v_\mC$ and $i_{\mL}$ first, then extended them to whole curve. 
The well-posedness of $v_\mC$, $i_{\mL}$ can be obtained by reducing the dynamics to a differential inclusion with a maximal monotone operator. 
We first restate the theorem in \cite{{aubin2012differential}} and its immediate implication as a remark, which we use in the proof. 
\begin{theorem}{\citep[Thm 3.2.1]{aubin2012differential}}
\label{thm-AC} \label{known-thm-for-well-posedness}
Let $\varmathbb{M}\colon\reals^n\rightrightarrows\reals^n$ be a maximal monotone operator, consider the differential inclusion
\begin{equation} \label{eq-dif-inclu-thm}
    \dot{X}(t)\in -\varmathbb{M}(X(t)),    
\end{equation}
with initial condition $X(0)=X_0 \in \dom \varmathbb{M}$. 
Then there is a unique solution $X\colon [0,\infty)\rightarrow\reals^n$ that is absolutely continuous and satisfies \eqref{eq-dif-inclu-thm} for almost all $t$. 
Moreover, 
if we denote $ \mathcal{T} = \{ t \in [0,\infty) \mid X \text{ is differentiable at } t \} $, then 
followings are true.
\begin{itemize}
    \item [(i)] Let $X(\cdot), Y(\cdot)$ are the solutions issued from $X_0, Y_0 \in \dom \varmathbb{M}$ respectively. 
    Then $\norm{X(t) - Y(t)} \le \norm{ X_0 - Y_0 }$ for all $t\ge0$. 
    \item [(ii)] For all $t\ge0$,
        $\dot{X}_{+}(t) := \lim_{h\to0+}\frac{X(t+h)-X(t)}{h}$ is well-defined and continuous from the right. 
        Note, $\dot{X}(t)=\dot{X}_{+}(t)$ for all $t \in \mathcal{T}$.
    \item [(iii)] $t \mapsto \norm{ \dot{X}_+(t) }$ is nonincreasing. 
    \item [(iv)] $\dot{X}_{+}(t) = -m(\varmathbb{M}(X(t)))$ holds for all $t\ge0$.
    % Here $m(\varmathbb{M}(X(t)))$ is the element of $\varmathbb{M}(X(t))$ minimal norm, that is, $m(\varmathbb{M}(X(t))) = \pi_{\varmathbb{M}(X(t))}(0)$. 
    Here $m(K)$ is the element of $K \subset \reals^n$ with minimal norm, that is, $m(K) = \Pi_{K}(0) = \underset{k \in K}{\argmin} \norm{ k }$. 
    Therefore $\dot{X}(t) = -m(\varmathbb{M}(X(t))) $ holds for all $t\in\mathcal{T}$, and so \eqref{eq-dif-inclu-thm} is satisfied almost everywhere.
\end{itemize}

\textbf{Remark.} 
From $(iii)$ we have $\norm{ \dot{X}_{+}(t) } \le \norm{ \dot{X}_{+}(0) } = \norm{ m(\varmathbb{M}(X_0)) }$ for all $t\ge0$, 
% and since $X(t) = \int_{0}^{t} \dot{X}_{+}(s) ds$, 
thus for $t_1, t_2 \ge 0 $ we have
\begin{align*}
    \norm{ X(t_1) - X(t_2) } 
    = \norm{ \int_{t_2}^{t_1} \dot{X}_{+}(s) ds } 
    &\le \int_{t_2}^{t_1} \norm{ \dot{X}_{+}(s) } ds  \\
    &\le \int_{t_2}^{t_1} \norm{ m(\varmathbb{M}(X_0)) } ds = | t_1 - t_2 | \norm{ m(\varmathbb{M}(X_0)) }.
\end{align*}
Therefore the theorem implies that $X$ is Lipschitz-continuous, in particular with parameter $\norm{ m(\varmathbb{M}(X_0)) }$. 

\end{theorem}

Thus our first goal is to prove the condition \eqref{eq-reduced-dynamics} can be equivalently written as
\[
    \frac{d}{dt} \bmat{ v_{\mC} \\ i_{\mL} } \in -\varmathbb{A} \bmat{ v_{\mC} \\ i_{\mL} }
\]
for some maximal monotone operator $\varmathbb{A}\colon \reals^{|\mC| + |\mL|} \rightrightarrows \reals^{|\mC| + |\mL|}$. 
We first establish an efficient reformulation of KCL and KVL. 
\begin{lemma} \label{expressing KCL KVL with skew-symmetric H}
    There is a skew-symmetric matrix $\hat{H}\colon\reals^{\sigma + m} \to \reals^{\sigma + m}$ and 
    a corresponding diagonal matrix $J\colon\reals^{\sigma + m} \to \reals^{\sigma + m}$ with entries $0$ of $1$ that satisfies 
    \[
        \bmat{ i \\ y } \in \mathcal{N}(B),  \,\,  
        \bmat{ v \\ x } \in \mathcal{R}(B^{\intercal})
        \quad \Longleftrightarrow \quad 
         \hat{u} = \hat{H} \hat{w}, 
         % \text{ \rm where }
         % \bmat{ w \\ u } =  \bmat{ Q & 0 \\ 0 & Q } \bmat{ J & I_{\sigma + m} - J \\ I_{\sigma + m} - J & J } \bmat{ v \\ x \\ i \\ y }.
    \]
    where $\hat{u}$ and $\hat{w}$ are defined as
    \[
        \hat{w} = \bmat{ J & I_{\sigma + m} - J } \bmat{ v \\ x \\ i \\ y }, \quad
        \hat{u} = \bmat{ I_{\sigma + m} - J & J } \bmat{ v \\ x \\ i \\ y }.
    \]
    Moreover, let $Q\colon\reals^{\sigma + m} \to \reals^{\sigma + m}$ be a permutation matrix, define $w = Q\hat{w}$, $u=Q\hat{u}$. 
    Then there is a skew-symmetric matrix $H$ that satisfies
    \[
        \bmat{ i \\ y } \in \mathcal{N}(B),  \,\,  
        \bmat{ v \\ x } \in \mathcal{R}(B^{\intercal})
        \quad \Longleftrightarrow \quad 
        u = H w.
    \]
\end{lemma}

\paragraph{Remark.} 
The diagonal matrix $J$ determines whether to select voltage or current for each component, to construct $\hat{w}$. 
To clarify, $w, u \in \reals^{\sigma + m}$ are the vectors that $\{w_l, u_l \}$ becomes a current and voltage pair of a component for $l=1,2,\dots,\sigma+m$. 
Such partitions of current, voltages values $w, u$ and skew-symmetric matrix $H$ were also considered in \cite{Hughes2017_passivity} with different notation. 
However, we introduce our method of constructing them here, as we will consider $H$ with a special property in Corollary~\ref{cor-w-selection} that plays a key role in the proof.

\begin{proof}
    Define $N$ and $\rB$  be matrices consisted with basis of $\mathcal{N}(B)$ and $\row(B)$ respectively. 
Then KCL and KVL can be shortly rewritten as
\[
    \bmat{ N & 0 \\ 0 & \rB } \bmat{v\\x\\i\\y} = \bmat{ 0 \\ 0 }.
\]
We now show there is a diagonal matrix $J \colon \reals^{\sigma + m} \to \reals^{\sigma + m}$ with entries $0$ or $1$, 
that makes the below square matrix invertible 
\[
    G = \bmat{ N & 0 \\ 0 & \rB \\ J & I_{\sigma + m} - J } \in \reals^{2(\sigma+m) \times 2(\sigma+m)}.
\]
Name $N_0 = \bmat{ N \\ 0 }$ and $\rB_0 = \bmat{ 0 \\ \rB }$. 
We will attach the standard basis vectors or $0$ below, and increase the index with attached number of rows. 
We proceed induction on the index, until the index becomes $\sigma + m$. 

Suppose, for $0\le k\le \sigma + m -1$, $N_k$ and $\rB_k$ satisfy the form
\begin{equation} \label{form of N_K, B_k}
    N_k = \bmat{ N \\ 0 \\ j_1 \ve_1 \\ \vdots \\ j_k \ve_k }, \quad 
    \rB_k = \bmat{ 0 \\ \rB \\ (1-j_1) \ve_1 \\ \vdots \\ (1-j_k) \ve_k },
\end{equation}
where $j_l \in \set{0,1}$ and $\ve_l \in \reals^{\sigma + m}$ is a standard basis (row) vector for $1\le l \le k$. 
We claim either $\ve_{k+1} \notin \row(N_k)$ or $\ve_{k+1} \notin \row(\rB_k)$ is true.

Proof by contradiction. Suppose not. That is, suppose $\ve_{k+1} \in \row(N_k)$ and $\ve_{k+1} \in \row(\rB_k)$. 
Then there are $\vn = (n_1, \dots, n_{\sigma+m}) \in \row(N)$, $\vr = (r_1, \dots, r_{\sigma+m}) \in \row(\rB)$ and coefficients $a_l$, $b_l$ such that 
\[
    \ve_{k+1} = \vn + \sum_{l=1}^{k} a_l j_l \ve_l = \vr + \sum_{l=1}^{k} b_l (1-j_l) \ve_l.
\]
Taking inner product with $\ve_{p}$, $1\le p \le \sigma + m$, 
we have
\[
    n_p = \begin{cases}
       -a_p & \text{ if } 1\le p \le k, \,\, j_p = 1 \\ 
       0 & \text{ if } 1\le p \le k, \,\, j_p = 0 \\ 
       1 & \text{ if } p = k+1 \\
       0 & \text{ if } k+1 < p \le \sigma + m ,
    \end{cases}
    \qquad
    r_p = \begin{cases}
       0 & \text{ if } 1\le p \le k, \,\, j_p = 1 \\ 
       -b_p & \text{ if } 1\le p \le k, \,\, j_p = 0 \\ 
       1 & \text{ if } p = k+1 \\
       0 & \text{ if } k+1 < p \le \sigma + m .
    \end{cases}
\]
Therefore
\[
    \inner{\vn}{\vr} 
    = \sum_{p=1}^{\sigma + m} n_p r_p
    = n_{k+1} r_{k+1}
    = 1.
\]
By the way, since $\vn \in \row(N) = N(B)$, $\vr \in \row(\rB) = R(N^\intercal)$ we have $\vn \perp \vr$ and so $\inner{\vn}{\vr}=0$. 
A contradiction, we conclude either $\ve_{k+1} \notin \row(N_k)$ or $\ve_{k+1} \notin \row(\rB_k)$ is true.

From the proved claim, we can extend $N_0$, $\rB_0$ to $N_{\sigma+m}$, $\rB_{\sigma+m}$ with keeping the form of \eqref{form of N_K, B_k} by repeating the process below. 
Recall, the desired form of the matrix was 
\[
    G = \bmat{ N & 0 \\ 0 & \rB \\ J & I_{\sigma + m} - J } 
\]
with diagonal matrix $J \in \reals^{(\sigma + m) \times (\sigma + m)}$ with entries $0$ or $1$. 
By the construction, we see matrix $\bmat{ N_{\sigma+m} & \rB_{\sigma+m} }$ satisfies the desired form. 
Moreover, we know the nonzero rows of $N_{\sigma+m}$ and $\rB_{\sigma+m}$ are linearly independent respectively, by their construction. 
By the form of $G$ we see if $l$-th row of $N_{\sigma+m}$ is nonzero then $l$-th row of $\rB_{\sigma+m}$ is zero and vice-versa, we conclude the rows of $G$ are linearly independent. 
Therefore, $G$ is invertible. 

Observe
\begin{equation} \label{w express all components with linear combination}
    G \bmat{ v \\ x \\ i \\ y } 
    = \bmat{ N & 0 \\ 0 & \rB \\ J & I_{\sigma + m} - J } 
    \bmat{ v \\ x \\ i \\ y } 
    = \bmat{ 0  \\ \hat{w} } 
    \quad \Longrightarrow \quad
    \bmat{ v \\ x \\ i \\ y }
    =  G^{-1} \bmat{ 0  \\ \hat{w} }.
\end{equation}
% Thus all components $v_{\mathcal{R}}, v_{\mathcal{L}}, v_{\mathcal{C}}, x, i_{\mathcal{R}}, i_{\mathcal{L}}, i_{\mathcal{C}}, y$ are expressed by $w_{\mathcal{R}}, w_{\mathcal{L}}, w_{\mathcal{C}}, w_m$ with linear combination, which was our desired result. 
We know above equation holds for arbitrarily chosen $(v,x)$, $(i,y)$ that satisfies KVL and KCL respectively. 
Observe $\dom(G) = R(B^\intercal) \times N(B)$ and from dimension theorem  we know
\[
    \dim(R(B^\intercal) \times N(B))
    = \dim(R(B^\intercal)) + \dim(N(B))
    = \sigma + m.
\]
As $G$ is invertible, we have $\dim(R(G)) = \dim(\dom(G)) = \sigma + m$. 
Therefore the values of the components of $\hat{w}$ can be arbitrary values in $\reals$. 

Rearranging the rows of $G^{-1}$, from \eqref{w express all components with linear combination} we obtain $\tilde{H} \in \reals^{2(\sigma+m) \times 2(\sigma+m)}$ that satisfies
\[
    \bmat{ \hat{w} \\ \hat{u} }
    = \tilde{H} \bmat{ 0 \\ \hat{w} }
    = \bmat{ \tilde{H}_0^{w} & \tilde{H}_{w}^{w} \\ \tilde{H}_0^{u} & \tilde{H}_{w}^{u}  } \bmat{ 0 \\ \hat{w} },
\]
where the block matrices are in $\reals^{(\sigma + m)\times(\sigma + m)}$. 
Now, naming $\hat{H} = \tilde{H}_{w}^{u}$ we get
\[
    \hat{u} = \hat{H} \hat{w}.
\]
Now to show $H$ is skew-symmetric, recall from $(v,x) \in R(B^\intercal)$ and $(i,y)\in N(B)$ we have $\inner{(v,x)}{(i,y)} =0$. 
Thus for all $\hat{w} \in \reals^{\sigma + m}$, we have
\begin{align*}
    \inner{\hat{w}}{\hat{H}\hat{w}} = \inner{\hat{w}}{\hat{u}} = \inner{(v,x)}{(i,y)} =0. 
\end{align*}
Therefore $\hat{H}$ is skew-symmetric.  

Finally, let $Q\colon\reals^{\sigma + m} \to \reals^{\sigma + m}$ be a permutation matrix. 
Define $H = Q\hat{H}Q^{\intercal}$. 
Since 
\begin{align*}
    H^\intercal = Q\hat{H}^{\intercal}Q^{\intercal} = Q(-\hat{H})Q^{\intercal} = -H, 
\end{align*}
$H$ is skew-symmetric. 
And from $Q^\intercal Q = I_{\sigma+m}$ we have
\[
    \hat{u} = \hat{H} \hat{w} 
    \quad \iff \quad 
    u = Q\hat{u} = Q \hat{H} \hat{w} = Q \hat{H} Q^{\intercal} Q \hat{w} = H w,
\]
we conclude the proof. 

\end{proof}

\begin{corollary} \label{cor-w-selection}
Recall $\hat{w}$ is composed with voltage or current values of each component. 
Integrate the values of resistors and denote as $r$, integrate $v_{\mC}, i_{\mL}$ as $\pp$ and integrate $i_{\mC}, v_{\mL}$ as $\pp_*$. 
Then we may rearrange the elements of $\hat{w}$ with certain permutation matrix $Q$, that $w=Q\hat{w}$ can be decomposed as following order
\[
    w = \bmat{  w_{\pp} \\ w_{\pp_*} \\ w_{\Rm} }, 
    \qquad \text{ where } \quad 
    w_\pp = \bmat{ w_{v_\mC} \\ w_{i_\mL} },  \quad
    w_{\pp^*} = \bmat{ w_{i_\mC} \\ w_{v_\mL} },  \quad
    w_{\Rm} = \bmat{ w_{v_{\Rm}} \\ w_{i_{\Rm}} } = \bmat{ w_{v_\mR} \\ w_{x} \\ w_{i_\mR} \\ w_{y} }. 
\]
Consider rewriting $u=Hw$ in the decomposed way as
\begin{equation} \label{eq-decomposed-H}
    \bmat{ u_{\pp_*} \\ u_{\pp} \\ u_{\Rm} }
    = \bmat{ H_{\pp}^{\pp_*} & H_{\pp_*}^{\pp_*} & H_{\Rm}^{\pp_*} 
        \\ H_{\pp}^{\pp} & H_{\pp_*}^{\pp} & H_{\Rm}^{\pp} 
        \\ H_{\pp}^{\Rm} & H_{\pp_*}^{\Rm} & H_{\Rm}^{\Rm} } 
    \bmat{ w_{\pp} \\ w_{\pp_*} \\ w_{\Rm} }.
\end{equation}
Then there is a diagonal matrix $J$ satisfies the properties considered in Lemma~\ref{expressing KCL KVL with skew-symmetric H}, that corresponding $H$ satisfies
\[
    H_{\pp_*}^{\Rm}=0, \quad
    H_{\pp_*}^{\pp}=0, \quad 
    H_{\Rm}^{\pp_*}=0.
\]
\end{corollary}

\begin{proof} [Proof]
        Name the indices as $\mC_l, \mL_k \in \set{1, \dots, \sigma + m}$ for $l\in\set{1,\dots, |\mC|}$, $k\in\set{1,\dots, |\mL|}$ that satisfy
    \[
        \ve_{\mC_l} \bmat{v\\x} = v_{\mC_l}, \qquad \ve_{\mL_k}\bmat{i\\y} = i_{\mL_k}.
    \]
    First we put $\ve_{\mC_l}$'s and $\ve_{\mL_k}$'s in $J$ as many as possible. 
    That is, determine the values of $j_{\mC_l}$'s and $j_{\mL_k}$'s to satisfy
    \begin{itemize}
        \item $\set{\ve_{\mC_l} \mid j_{\mC_l}=1}$ is linearly independent to $\row(N)$.
        \item $\set{\ve_{\mC_l} \mid j_{\mC_l}=1} \cup \set{ \ve_{\mC_{l'}}}$ is linearly dependent to $\row(N)$ for any $l'\in\set{  j_{\mC_l} \ne 1}$.
        \item $\set{\ve_{\mL_s} \mid j_{\mL_k}=0}$ is linearly independent to $\row(\tilde{B})$.
        \item $\set{\ve_{\mL_s} \mid j_{\mL_k}=0} \cup \set{ \ve_{\mL_{s'}}}$ is linearly dependent to $\row(\tilde{B})$ for any $s'\in\set{ j_{\mL_s} \ne 0 }$.
    \end{itemize}
    Next, fill the remaining $j$'s as we've done in Lemma~\ref{expressing KCL KVL with skew-symmetric H}. 
    
    Since the proof can be applied using the same argument to other cases, we will focus on a specific case. 
    Focusing on the last row of \eqref{eq-decomposed-H}, we can furthermore decompose and write as following
    \[
         u_{\Rm}
        =  \bmat{ H_{\pp}^{\Rm} & H_{\pp_*}^{\Rm} & H_{\Rm}^{\Rm} } 
            \bmat{ w_{\pp} \\ w_{\pp_*} \\ w_{\Rm} }
        \,\, \iff \,\,
        \bmat{ u_{i_\Rm} \\ u_{v_\Rm} }
        =  \bmat{ 0 & H_{i_\mL}^{i_{\Rm}} & H_{i_\mC}^{i_{\Rm}} & 0 &  0 & H_{i_{\Rm}}^{i_{\Rm}} \\ 
        H_{v_\mC}^{v_{\Rm}} & 0 & 0 &  H_{v_\mL}^{v_{\Rm}} & H_{v_{\Rm}}^{v_{\Rm}} & 0  }
    \bmat{ w_{v_{\mC}} \\ w_{i_{\mL}} \\ w_{i_{\mC}}  \\ w_{v_{\mL}} \\ w_{v_{\Rm}} \\ w_{i_{\Rm}} }.
    \]
    Note, since above equations origin from KCL and KVL (which are linear equations only with current values or voltage values), $H_\alpha^\beta=0$ if $\alpha$ is current and $\beta$ is voltage, and vice-versa. Refer \citep[Theorem 6.3]{SeshuReed1961_linear}.

    Observe $H_{\pp_*}^{\Rm} = \bmat{ H_{i_\mC}^{i_{\Rm}} & 0 \\ 0 &  H_{v_\mL}^{v_{\Rm}} }$, here we show $H_{i_\mC}^{i_{\Rm}}=0$. 
    Focusing on arbitrary $k$'th row of $H_{i_\mC}^{i_{\Rm}}$, from above equality we get
    \[
        u_{i_{\Rm_k}}
        = \bmat{ H_{i_\mL}^{i_{\Rm}} & H_{i_\mC}^{i_{\Rm}} & H_{i_{\Rm}}^{i_{\Rm}} }_k \bmat{ w_{i_{\mL}} \\ w_{i_{\mC}} \\  w_{i_{\Rm}} }
        \quad \iff \quad 
        0 = \bmat{ H_{i_\mL}^{i_{\Rm}} & H_{i_\mC}^{i_{\Rm}} & H_{i_{\Rm}}^{i_{\Rm}} }_k \bmat{ w_{i_{\mL}} \\ w_{i_{\mC}} \\  w_{i_{\Rm}} } - u_{i_{\Rm_k}},
    \]
    where the subscript $k$ means the $k$'th row of the block matrix. 
    As this is a linear equation of current values, it origins from KCL, thus there is a vector $\vr \in \row(\tilde{B})$ corresponding to this equation, \ie 
    \[
        \vr \bmat{ i \\ y } = \bmat{ H_{i_\mL}^{i_{\Rm}} & H_{i_\mC}^{i_{\Rm}} &  H_{i_{\Rm}}^{i_{\Rm}} }_k 
        \bmat{ w_{i_{\mL}} \\ w_{i_{\mC}} \\ w_{i_{\Rm}} } - u_{i_{\Rm_k}}.
    \]
    On the other hand, as $w_{i_{\mL}}, w_{i_{\mC}}, w_{i_{\Rm}}$ are consisted with the components of $i,y$ that corresponds to $j_l=0$, there are coefficient vectors $a\in \reals^{|\mL|}$, $b\in\reals^{|\mC|}$, $c\in\reals^{|\mR|+m}$ that satisfies
    \[
        \bmat{ H_{i_\mL}^{i_{\Rm}} & H_{i_\mC}^{i_{\Rm}} &  H_{i_{\Rm}}^{i_{\Rm}} }_k 
        \bmat{ w_{i_{\mL}} \\ w_{i_{\mC}} \\ w_{i_{\Rm}} }  - u_{i_{\Rm_k}}
        = \pr{ \sum_{s \in \{ j_{\mL_s}=0 \} } \!\!\!\! a_s \ve_{\mL_s} + \sum_{l \in \{j_{\mC_l}=0\} } \!\!\!\! b_l \ve_{\mC_l} + \sum_{q \in \{j_{\Rm_q}=0\} } \!\!\!\! c_q \ve_{\Rm_q} - \ve_{\Rm_k} } \bmat{ i \\ y }.
    \]
    Note $b_l$'s correspond to components of $\bmat{ H_{i_\mC}^{i_{\Rm}}}_{k}$. 
    Organizing, we have
    \[
        \vr = \sum_{s \in \{ j_{\mL_s}=0 \} } \!\!\!\! a_s \ve_{\mL_s} 
            + \sum_{l \in \{j_{\mC_l}=0\} } \!\!\!\! b_l \ve_{\mC_l} 
            + \sum_{q \in \{j_{\Rm_q}=0\} } \!\!\!\! c_q \ve_{\Rm_q} - \ve_{\Rm_k}.
    \]
    Observe that from right hand side, we can see $\vr$ is orthogonal to $\set{\ve_{\mC_{l'}} \mid j_{\mC_{l'}}=1}$.

    By the way, as $\set{\ve_{\mC_{l'}} \mid j_{\mC_{l'}}=1} \cup \set{ \ve_{\mC_{l}}}$ is linearly dependent to $\row(N)$ for all $l \in \{j_{\mC_l}=0\}$, we see
    \[
        \sum_{l \in \{j_{\mC_l}=0\} } \!\!\!\! b_l \ve_{\mC_l} \in \textbf{span} \pr{ \set{\ve_{\mC_{l'}} \mid j_{\mC_{l'}}=1} \cup \row(N)},
    \]
    so there is some coefficient vector $d \in \reals^{|\mC|}$ and $\vn \in \row(N)$ that satisfies
    \[
        \sum_{l \in \{j_{\mC_l}=0\} } \!\!\!\! b_l \ve_{\mC_l} 
        = \sum_{l' \in \{ j_{\mC_{l'}}=1 \} } \!\!\!\! d_{l'} \ve_{\mC_{l'}} + \vn.
    \]
    However, as $\vr \in \row(\tilde{B})$ and $\row(\tilde{B}) \perp \row(N)$, we have $\inner{\vr}{\vn}=0$. 
    Moreover, as $\vr$ is orthogonal to $\set{\ve_{\mC_{l'}} \mid j_{\mC_{l'}}=1}$, we conclude
    \[
        0 = \inner{\vr}{\sum_{l' \in \{ j_{\mC_{l'}}=1 \}} \!\!\!\! d_{l'} \ve_{\mC_{l'}} + \vn}
        = \inner{\vr}{\sum_{l \in \{j_{\mC_l}=0\}} \!\!\!\! b_l \ve_{\mC_l}}
        = \norm{ \sum_{l \in \{j_{\mC_l}=0\}} \!\!\!\! b_l \ve_{\mC_l} }^2
        = \sum_{l \in \{j_{\mC_l}=0\}} \!\!\!\! b_l^2 .
    \]
    Therefore, as $b_l$'s corresponds to components of $\bmat{ H_{i_\mC}^{i_{\Rm}}}_{k}$, we conclude $\bmat{ H_{i_\mC}^{i_{\Rm}}}_{k} = 0$.
    As $k$ was arbitrary, we get $H_{i_\mC}^{i_{\Rm}} = 0$. 
    Similarly we can show $H_{v_\mL}^{v_{\Rm}}=0$, and thus $H_{\pp_*}^{\Rm}=0$.     
    Repeating the same argument, we can show $H_{\pp_*}^{\pp}=0$. 
    Finally, as $H$ is skew-symmetric, we have $H_{\Rm}^{\pp_*} = -(H_{\pp_*}^{\Rm})^{\intercal} = 0$.
\end{proof}

We now move on to V-I relations of resistors. 
To express V-I relations in terms of $w$ and $u$, we adopt partial inverse. 
\paragraph{Definition.} \citep[Definition 20.42]{BauschkeCombettes2017_convex} 
Let $\varmathbb{M} \colon \reals^d \rightrightarrows \reals^d$ be a set-valued operator and let $K$ be a
closed linear subspace of $\reals^d$. 
Denote $\Pi_K \colon \reals^d \to \reals^d$ the projection onto $K$ as
\begin{align*}
    \Pi_{K}(z) = \underset{k \in K}{\argmin} \norm{ z - k }.
\end{align*}
The \textit{partial inverse} of $\varmathbb{M}$ with respect to $K$ is the
operator $\varmathbb{M}_K \colon \reals^d \rightrightarrows \reals^d$ defined by 
\begin{align*}
    \text{gra} \, \varmathbb{M}_K   = \set{ (\Pi_{K} \text{x} + \Pi_{K^{\perp}} \text{y}, \Pi_{K} \text{y} + \Pi_{K^{\perp}} \text{x}) \mid  (\text{x},\text{y}) \in \text{gra} \, \varmathbb{M} }, 
\end{align*}
\ie,
\begin{align*}
    u \in \varmathbb{M}_{K} (w)
    \iff \exists \text{x}, \text{y} \text{ such that } \text{y} \in \varmathbb{M} \text{x} \text{ and } 
    (w,u) = (\Pi_{K} \text{x} + \Pi_{K^{\perp}} \text{y}, \Pi_{K} \text{y} + \Pi_{K^{\perp}} \text{x}).
\end{align*}

We then prove important properties of the function related to V-I relations for resistors. 

\begin{lemma} \label{lem-partialF-fulldomain-continuous}
    Suppose $f$ is $\mu$-strongly convex and $M$-smooth function. 
    Let $Q_{\Rm}, H_{\Rm}^{\Rm}, J_{\Rm} \colon \reals^{|\mR|+m} \to \reals^{|\mR|+m}$ be a permutation matrix, 
    a skew-symmetric matrix, 
    a diagonal matrix with entries $1$ or $0$ respectively
    and let $K = \range(J_{\Rm})$. 
    Define $F\colon\reals^{|\mR|+m}\to\reals$ as
    \[
        F(v_{\mR}, x) = \frac{1}{2} \norm{ v_{\mR} }^2 + f(x).
    \]
    % Then $Q_\Rm (\nabla F)_K Q_\Rm^{\intercal} \colon \reals^{|\mR|+m} \rightrightarrows \reals^{|\mR|+m}$ is strongly monotone operator. 
    % Let $H_{\Rm}^{\Rm}\colon \reals^{|\mR|+m} \to \reals^{|\mR|+m}$ be a skew-symmetric matrix. 
    Then the following holds.
    \begin{itemize}
        \item [(i)] $\dom ( Q_\Rm (\nabla F)_K Q_\Rm^{\intercal} - H_{r}^{r} )^{-1} = \reals^{|\mR|+m}$.
        \item [(ii)] $( Q_\Rm (\nabla F)_K Q_\Rm^{\intercal} - H_{r}^{r} )^{-1}$ is Lipschitz continuous monotone operator. 
    \end{itemize}
    % $\dom ( Q_\Rm (\nabla F)_K Q_\Rm^{\intercal} - H_{r}^{r} )^{-1} = \reals^{|R|+m}$ and $( Q_\Rm (\nabla F)_K Q_\Rm^{\intercal} - H_{r}^{r} )^{-1}$ is Lipschitz continuous monotone operator. 
\end{lemma}

\begin{proof}
    Take $(w_r^l, u_r^l) \in \reals^{|\mR|+m}$ for $l \in \set{1,2}$, such that $u_r^l \in (Q_\Rm (\nabla F)_K Q_\Rm^{\intercal}) w_r^l$. 
    As $Q_\Rm$ is permutation matrix, we know $(Q_\Rm)^{-1} = (Q_\Rm)^{\intercal}$, 
    and thus  $Q_\Rm^{\intercal} u_r^l \in (\nabla F)_K (Q_\Rm^{\intercal} w_r^l)$. 
    % Observe from the property of $J_{\Rm}$, we know $\Pi_K(z) = J_{\Rm}z$ and $\Pi_{K^{\perp}}(z) = (I_{|\mR| + m} - J_{\Rm})z$ for all $z\in\reals^{|\mR|+m}$. 
    Then there are $\bmat{ v_{\mathcal{R}}^l \\ x^l }, \bmat{ i_{\mathcal{R}}^l \\ y^l } \in \reals^{|\mR|+m}$ such that 
    \begin{equation} \label{eq-V-I-relation-partial-inverse}
        \bmat{ i_{\mathcal{R}}^l \\ y^l }
        = \nabla F \bmat{ v_{\mathcal{R}}^l \\ x^l },
        \quad 
        \pr{ Q_{\Rm}^{\intercal} w_{\Rm}^l , Q_{\Rm}^{\intercal} u_{\Rm}^l }
        % \bmat{ Q_{\Rm}^{\intercal} w_{\Rm}^l \\ Q_{\Rm}^{\intercal} u_{\Rm}^l }
        % = \bmat{ J_{\Rm} & I_{|\mR|+m} -  J_{\Rm} \\ I_{|\mR|+m} -  J_{\Rm} & J_{\Rm} }  
        %   \bmat{ v_{\mathcal{R}}^l \\ x^l \\ i_{\mathcal{R}}^l \\ y^l } .
        = \pr{ \Pi_{K} \bmat{ v_{\mathcal{R}}^l \\ x^l  } + \Pi_{K^{\perp}} \bmat{ i_{\mathcal{R}}^l \\ y^l  } ,
            \Pi_{K} \bmat{ i_{\mathcal{R}}^l \\ y^l  } + \Pi_{K^{\perp}} \bmat{ v_{\mathcal{R}}^l \\ x^l  } }.
        % = \bmat{ \Pi_{K} \bmat{ v_{\mathcal{R}}^l \\ x^l  } + \Pi_{K^{\perp}} \bmat{ i_{\mathcal{R}}^l \\ y^l  } \\ 
        %     \Pi_{K} \bmat{ i_{\mathcal{R}}^l \\ y^l  } + \Pi_{K^{\perp}} \bmat{ v_{\mathcal{R}}^l \\ x^l  } }.
    \end{equation}
    By \citep[Proposition 20.44, (iii)]{BauschkeCombettes2017_convex}, we have 
    \[
        \inner{ Q_{\Rm}^{\intercal}(w_r^1 - w_r^2) }{ Q_{\Rm}^{\intercal}(u_r^1 - u_r^2) }  
        = \inner{ \bmat{ v_{\mathcal{R}}^1 \\ x^1 } - \bmat{ v_{\mathcal{R}}^2 \\ x^2 } }
            { \bmat{ i_{\mathcal{R}}^1 \\ y^1 } - \bmat{ i_{\mathcal{R}}^2 \\ y^2 } }  .
    \]
    Moreover, we can check
    \begin{align*}
        &\norm{ \bmat{ v_{\mathcal{R}}^1 \\ x^1 } - \bmat{ v_{\mathcal{R}}^2 \\ x^2 } }^2 
        + \norm{ \bmat{ i_{\mathcal{R}}^1 \\ y^1 } - \bmat{ i_{\mathcal{R}}^2 \\ y^2 }  }^2 \\
        &= \norm{ \Pi_{K} \pr{ \bmat{ v_{\mathcal{R}}^1 \\ x^1 } - \bmat{ v_{\mathcal{R}}^2 \\ x^2 }  } }^2
         + \norm{ \Pi_{K^{\perp}}\pr{ \bmat{ v_{\mathcal{R}}^1 \\ x^1 } - \bmat{ v_{\mathcal{R}}^2 \\ x^2 }  } }^2   \\&\quad
         + \norm{ \Pi_{K} \pr{ \bmat{ i_{\mathcal{R}}^1 \\ y^1 } - \bmat{ i_{\mathcal{R}}^2 \\ y^2 } } }^2 
         + \norm{ \Pi_{K^{\perp}}\pr{ \bmat{ i_{\mathcal{R}}^1 \\ y^1 } - \bmat{ i_{\mathcal{R}}^2 \\ y^2 } } }^2  
         = \norm{ w_r^1 - w_r^2 }^2 + \norm{ u_r^1 - u_r^2 }^2.
    \end{align*}
    
    Define $\mu_{\min} = \set{ \mu , 1 }$ and $M_{\min} = \set{M, 1}$. 
    Then we can check $\nabla F$ is $\mu_{\text{min}}$-strongly convex and $M_{\text{min}}$-smooth, we see
\BEAS
    \inner{ (i_{\mathcal{R}}^1, y^1) - (i_{\mathcal{R}}^2, y^2) }{ (v_{\mathcal{R}}^1, x^1) - (v_{\mathcal{R}}^2, x^2) }  
    &\ge&  \mu_{\text{min}} \norm{ (v_{\mathcal{R}}^1, x^1) - (v_{\mathcal{R}}^2, x^2) }^2, \\
    \inner{ (i_{\mathcal{R}}^1, y^1) - (i_{\mathcal{R}}^2, y^2) }{ (v_{\mathcal{R}}^1, x^1) - (v_{\mathcal{R}}^2, x^2) } 
    &\ge&  M_{\text{min}} \norm{ (i_{\mathcal{R}}^1, y^1) - (i_{\mathcal{R}}^2, y^2) }^2.
\EEAS
Thus
\BEAS
    \inner{ w_r^1 - w_r^2 }{ u_r^1 - u_r^2 }   
    &=& \inner{ Q_{\Rm}^{\intercal}(w_r^1 - w_r^2) }{ Q_{\Rm}^{\intercal}(u_r^1 - u_r^2) }  
    = \inner{ \bmat{ v_{\mathcal{R}}^1 \\ x^1 } - \bmat{ v_{\mathcal{R}}^2 \\ x^2 } }
            { \bmat{ i_{\mathcal{R}}^1 \\ y^1 } - \bmat{ i_{\mathcal{R}}^2 \\ y^2 } } \\
    &\ge& \frac{\mu_{\text{min}}  + M_{\text{min}}}{2} \pr{ \norm{ \bmat{ v_{\mathcal{R}}^1 \\ x^1 } - \bmat{ v_{\mathcal{R}}^2 \\ x^2 } }^2 + \norm{ \bmat{ i_{\mathcal{R}}^1 \\ y^1 } - \bmat{ i_{\mathcal{R}}^2 \\ y^2 }  }^2   } \\
    &=& \frac{\mu_{\text{min}}  + M_{\text{min}}}{2} \pr{ \norm{ w_r^1 - w_r^2 }^2 + \norm{ u_r^1 - u_r^2 }^2 } \\
    &\ge&   \frac{\mu_{\text{min}}  + M_{\text{min}}}{2} \norm{ w_r^1 - w_r^2 }^2, 
\EEAS
we see $Q_\Rm (\nabla F)_K Q_\Rm^{\intercal}$ is $\frac{\mu_{\text{min}}  + M_{\text{min}}}{2}$-strongly monotone. 
Note we can check $(\nabla F)_K$ is also strongly monotone, by considering the special case $Q_\Rm = I_{|\Rm|+m}$. 
Lastly, since $H_{r}^{r}$ is skew-symmetric, we know $\inner{H_{r}^{r} z}{z}=0$ for all $z\in\reals^{|\mR|+m}$. 
Therefore for arbitrary $(\tilde{w}_r^l, \tilde{u}_r^l) \in \reals^{|\mR|+m}$ with $l \in \set{1,2}$ such that $\tilde{u}_r^l \in \pr{  Q_\Rm (\nabla F)_K Q_\Rm^{\intercal} - H_{\Rm}^{\Rm} } \tilde{w}_r^l$, 
since ${u}_r^l = \tilde{u}_r^l + H_{\Rm}^{\Rm} \tilde{w}_r^l \in (Q_\Rm (\nabla F)_K Q_\Rm^{\intercal}) \tilde{w}_r^l$, 
\BEAS
    \inner{ \tilde{w}_r^1 - \tilde{w}_r^2 }{ \tilde{u}_r^1 - \tilde{u}_r^2 }
    &=& \inner{ \tilde{w}_r^1 - \tilde{w}_r^2 }{ {u}_r^1 - {u}_r^2 - H_{\Rm}^{\Rm} (\tilde{w}_r^1 - \tilde{w}_r^2) } \\
    &=& \inner{ \tilde{w}_r^1 - \tilde{w}_r^2 }{ {u}_r^1 - {u}_r^2 } \\
    &\ge&   \frac{\mu_{\text{min}}  + M_{\text{min}}}{2} \norm{ w_r^1 - w_r^2 }^2 
\EEAS
thus $(Q_\Rm (\nabla F)_K Q_\Rm^{\intercal} - H_{\Rm}^{\Rm})$ is also $\frac{\mu_{\text{min}}  + M_{\text{min}}}{2}$-strongly monotone.

Now since $\nabla F$ is maximal monotone, $(\nabla F)_K$ is maximal monotone by \citep[Proposition 20.44, (v)]{BauschkeCombettes2017_convex}. 
Since $(\nabla F)_K$ is strongly monotone, 
we have $\dom (\nabla F)_K = \reals^{|\mR|+m}$ by \citep[Proposition 22.11]{BauschkeCombettes2017_convex}, 
and thus $Q_\Rm (\nabla F)_K Q_\Rm^{\intercal}$ is maximal monotone by \citep[Theorem 12]{RyuYin2022_largescale}. 
Moreover, since both $Q_\Rm (\nabla F)_K Q_\Rm^{\intercal}$ and $-H_\Rm^{\Rm}$ have full domain, 
$(Q_\Rm (\nabla F)_K Q_\Rm^{\intercal} - H_\Rm^{\Rm})$ is maximal monotone by \citep[Theorem 10]{RyuYin2022_largescale}.

Organizing, $(Q_\Rm (\nabla F)_K Q_\Rm^{\intercal} - H_\Rm^{\Rm})$ is maximal monotone and strongly monotone. 
Therefore by \citep[Proposition 22.11]{BauschkeCombettes2017_convex}, 
we conclude (ii). 
Finally, observe
\BEAS
    (Q_\Rm (\nabla F)_K Q_\Rm^{\intercal} - H_\Rm^{\Rm})^{-1}
    &=& \Bigg( \underbrace{ Q_\Rm (\nabla F)_K Q_\Rm^{\intercal} - H_\Rm^{\Rm} - \frac{\mu_{\text{min}}  + M_{\text{min}}}{2} I_{|\mR|+m} }_{=:\varmathbb{M}} + \frac{\mu_{\text{min}} + M_{\text{min}}}{2} I_{|\mR|+m} \Bigg)^{-1} \\
    &=& \varmathbb{J}_{ \frac{2}{\mu_{\text{min}} + M_{\text{min}}} \varmathbb{M} }  
        \circ \frac{2}{\mu_{\text{min}} + M_{\text{min}}} I_{|\mR| + m}.
\EEAS
Since $\varmathbb{M}$ is monotone, $\frac{2}{\mu_{\text{min}} + M_{\text{min}}} \varmathbb{M}$ is also monotone, 
by \citep[Corollary 23.9]{BauschkeCombettes2017_convex} we know $\varmathbb{J}_{ \frac{2}{\mu_{\text{min}} + M_{\text{min}}} \varmathbb{M} }$ is $1$-Lipschitz continuous. 
Therefore $(Q_\Rm (\nabla F)_K Q_\Rm^{\intercal} - H_\Rm^{\Rm})^{-1}$ is $\frac{2}{\mu_{\text{min}} + M_{\text{min}}}$-Lipschitz continuous. Finally it is monotone as it is an inverse of a monotone operator, we conclude (ii).
\end{proof}

Finally, we are ready to prove Theorem~\ref{thm-well-posedness}. 
\begin{proof}[Proof of Theorem~\ref{thm-well-posedness}]
By Lemma~\ref{lem-enough-to-show-reduced} it is suffices to show Theorem~\ref{thm-well-posedness-for-reduced}. 

\emph{(i) Well-posedness and Lipschitz continuity of  $v_\mC, i_\mL$. Existence of the whole curve $(v,x,i,y)$.} \\      
    % We first show $v_\mC, i_\mL$ is well-posed. 
    Define an operator $\varmathbb{A} \colon \reals^{|{\mL}|+|{\mC}|} \rightrightarrows \reals^{|{\mL}|+|{\mC}|}$ as
    \begin{align} \label{eq-def-varmathbbA}
        \varmathbb{A} = \Bigg\{ 
            \pr{ \bmat{ {v}_{\mathcal{C}} \\ {i}_{\mathcal{L}} }, \,
                  -\bmat{ {i}_{\mathcal{C}} \\ {v}_{\mathcal{L}} } } \,\Big|&\,\, \exists \,
            {v} =  ( {v}_{\mathcal{R}}, {v}_{\mathcal{L}}, {v}_{\mathcal{C}} ), \,\, 
            {i} =  ( {i}_{\mathcal{R}}, {i}_{\mathcal{L}}, {i}_{\mathcal{C}} ), \,\, (x,y)  \\& \quad  \nonumber
           \text{ such that } \bmat{ i \\ y } \in \mathcal{N}(B),  \,\,
            \bmat{v \\ x } \in \mathcal{R}(B^{\intercal}), \,\,
            y = \nabla f(x), \,\,
            v_{\mathcal{R}} = i_{\mathcal{R}} \Bigg\}.
    \end{align}
    We prove $\varmathbb{A}$ is maximal monotone by providing an explicit expression of $\varmathbb{A}$ and apply Theorem~\ref{known-thm-for-well-posedness}. 

    From Corollary~\ref{cor-w-selection}, we know there is a diagonal matrix $J\colon\reals^{\sigma + m} \to \reals^{\sigma + m}$, a permutation matrix $Q\colon\reals^{\sigma + m} \to \reals^{\sigma + m}$ and 
    a corresponding skew-symmetric matrix $H\colon\reals^{\sigma + m} \to \reals^{\sigma + m}$ that satisfies
    \begin{align} \label{eq-reduced-H}
        \bmat{ i \\ y } \in \mathcal{N}(B),  \,\,
        \bmat{v \\ x } \in \mathcal{R}(B^{\intercal}) 
        \iff
        \bmat{ u_{\pp_*} \\ u_{\pp} \\ u_{\Rm}  }
        = \bmat{ H_{\pp}^{\pp_*} & H_{\pp_*}^{\pp_*} & H_{\Rm}^{\pp_*} 
            \\ H_{\pp}^{\pp} & 0 & 0
            \\ H_{\pp}^{\Rm} & 0 & H_{\Rm}^{\Rm} } 
        \bmat{ w_{\pp} \\ w_{\pp_*} \\ w_{\Rm} }
    \end{align}
    where $w$ and $u$ are defined as in Lemma~\ref{expressing KCL KVL with skew-symmetric H}. 

    Define $F\colon \reals^{|\mR|+m} \to \reals$ as $F({v}_{\mathcal{R}}, x) = \frac{1}{2} \norm{ {v}_{\mathcal{R}} }^2 + f(x)$. 
    % then we know $\bmat{ i_{\mathcal{R}} \\ y } = \nabla F \bmat{ v_{\mathcal{R}} \\ x }$. 
    Then it is straight forward that
    \[
        y = \nabla f(x), \,\,  v_{\mathcal{R}} = i_{\mathcal{R}}
        \iff \bmat{ i_{\mathcal{R}} \\ y } = \nabla F \bmat{ v_{\mathcal{R}} \\ x }.
    \]
    By the construction of $w$, $u$, there is a diagonal matrix $J_{\Rm} \colon \reals^{|\mR|+m} \to \reals^{|\mR|+m}$ with entries $1$ or $0$ and a permutation matrix $Q_\Rm \colon \reals^{|\mR|+m} \to \reals^{|\mR|+m}$ that satisfies
    \[
        \bmat{ Q_\Rm^{-1} w_\Rm \\ Q_\Rm^{-1} u_\Rm }
        = \bmat{ Q_\Rm^{\intercal} w_\Rm \\ Q_\Rm^{\intercal} u_\Rm }
        = \bmat{ J_\Rm & I_{|\mR| + m} \\  I_{|\mR| + m} & J_\Rm } 
            \bmat{ v_{\mathcal{R}} \\ x \\ i_{\mathcal{R}} \\ y }.
    \]
    Define $K_{\Rm}=\range(J_{\Rm})$. Note, we can check $\Pi_{K_{\Rm}}(z) = J_{\Rm}z$ and $\Pi_{K^{\perp}_{\Rm}}(z) = (I_{|\mR| + m} - J_{\Rm})z$ for all $z\in\reals^{|\mR|+m}$. 
    Recalling \eqref{eq-V-I-relation-partial-inverse} in Lemma~\ref{lem-partialF-fulldomain-continuous} we see
    \[
        \bmat{ i_{\mathcal{R}} \\ y } = \nabla F \bmat{ v_{\mathcal{R}} \\ x }
        \iff 
        u_{r} = ( Q_\Rm (\nabla F)_{K_\Rm} Q_\Rm^{\intercal} ) w_{r} .
    \]
    From 
    \[
        (Q_\Rm (\nabla F)_{K_\Rm} Q_\Rm^{\intercal})  w_{\Rm}
        = u_{\Rm}
        = H_{\pp}^{\Rm} w_{\pp} + H_{\Rm}^{\Rm} w_{\Rm},
    \]
    we get expression for $w_\Rm$ in terms of $w_\pp$
    \begin{equation} \label{eq-expressing-resistors-with-CL}
        w_{\Rm}
        = \pr{ Q_\Rm (\nabla F)_{K_\Rm} Q_\Rm^{\intercal}  - H_{\Rm}^{\Rm} }^{-1}  H_{\pp}^{\Rm} w_{\pp}.
    \end{equation}
    % Note $w_{\pp_*}$ is not used while expressing the resistor values. 
    % Therefore we conclude all components can be expressed with $w_{\pp}$ and $w_{\pp_*}$. 
    Now focusing on the expression for $\bmat{ u_{\pp_*} \\ u_{\pp} }$, 
    since $H_{\pp}^{\Rm} = -(H_{\Rm}^{\pp_*})^{\intercal}$ as $H$ is skew-symmetric, 
    eliminating $w_\Rm$ by applying \eqref{eq-expressing-resistors-with-CL} we see 
    % \begin{equation} \label{eq-u_p-u_p*-by-w_p-w_p*}
    %     \begin{aligned}
    %     \bmat{ u_{\pp_*} \\ u_{\pp} }
    %     &= \bmat{ H_{\pp}^{\pp_*} & H_{\pp_*}^{\pp_*} & H_{\Rm}^{\pp_*} 
    %         \\ H_{\pp}^{\pp} & 0 & 0 } 
    %     \bmat{ w_{\pp} \\ w_{\pp_*} \\ w_{\Rm} } \\
    %     &=  - \underbrace{\left( - \bmat{ H_{\pp}^{\pp_*} & H_{\pp_*}^{\pp_*}  \\ H_{\pp}^{\pp} & 0  }
    %          + \bmat{  H_{\pp}^{\Rm} & 0 }^{\intercal} \pr{ Q_\Rm (\nabla F)_{K_\Rm} Q_\Rm^{\intercal}  - H_{r}^{r} }^{-1} \bmat{  H_{\pp}^{\Rm} & 0 } \right) }_{\varmathbb{B}}  \bmat{ w_{\pp} \\ w_{\pp_*} }.
    % \end{aligned}
    % \end{equation}
    \BEA \label{eq-u_p-u_p*-by-w_p-w_p*}
        \bmat{ u_{\pp_*} \\ u_{\pp} }
        &=& \bmat{ H_{\pp}^{\pp_*} & H_{\pp_*}^{\pp_*} & H_{\Rm}^{\pp_*} 
            \\ H_{\pp}^{\pp} & 0 & 0 } 
        \bmat{ w_{\pp} \\ w_{\pp_*} \\ w_{\Rm} } \\
        &=&  - \underbrace{\left( - \bmat{ H_{\pp}^{\pp_*} & H_{\pp_*}^{\pp_*}  \\ H_{\pp}^{\pp} & 0  }
             + \bmat{  H_{\pp}^{\Rm} & 0 }^{\intercal} \pr{ Q_\Rm (\nabla F)_{K_\Rm} Q_\Rm^{\intercal}  - H_{r}^{r} }^{-1} \bmat{  H_{\pp}^{\Rm} & 0 } \right) }_{\varmathbb{B}}  \bmat{ w_{\pp} \\ w_{\pp_*} }. \nonumber
    \EEA
    Since $H$ is skew-symmetric, its principal minors {\small $-\bmat{ H_{\pp}^{\pp_*} & H_{\pp_*}^{\pp_*}  \\ H_{\pp}^{\pp} & 0  }$} and $-H_{r}^{r}$ are skew-symmetric and so maximal monotone. 
    Furthermore, by Lemma~\ref{lem-partialF-fulldomain-continuous} we have $(  Q_\Rm (\nabla F)_{K_\Rm} Q_\Rm^{\intercal} - H_{r}^{r} )^{-1}$ is maximal monotone and $\dom ( (\nabla F)^{1,-1} - H_{r}^{r} )^{-1} = \reals^{|\mR|+m}$. 
    Invoking \citep[Theorem~11, 12]{RyuYin2022_largescale}, we conclude $\varmathbb{B}$ is maximal monotone.
    
    Organizing, we see
    \begin{align*}
        \bmat{ {i}_{\mathcal{C}} \\ {v}_{\mathcal{L}} } \in -\varmathbb{A} \bmat{ {v}_{\mathcal{C}} \\ {i}_{\mathcal{L}} }   
        \iff&\bmat{ i \\ y } \in \mathcal{N}(B),  \,\,
            \bmat{v \\ x } \in \mathcal{R}(B^{\intercal}), \,\,
            y = \nabla f(x), \,\,
            v_{\mathcal{R}} = i_{\mathcal{R}} \\
        \iff&
        \bmat{ u_{\pp_*} \\ u_{\pp} } = -\varmathbb{B} \bmat{ w_{\pp} \\ w_{\pp_*} },
        \text{ where }
        \bmat{ w \\ u  } =  \bmat{ Q & 0 \\ 0 & Q } \bmat{ J & I_{\sigma + m} - J  \\ I_{\sigma + m} - J  & J } \bmat{ v \\ x \\ i \\ y }.
    \end{align*}
    Therefore there is a diagonal matrix $J_{\pp, \pp^*} \colon \reals^{|\mL| + |\mC|} \to \reals^{|\mL| + |\mC|}$ with entries $1$ or $0$ 
    and a permutation matrix $Q_{\pp, \pp^*} \colon \reals^{|\mL| + |\mC|} \to \reals^{|\mL| + |\mC|}$ that satisfies
    \begin{align*}
        \bmat{ {i}_{\mathcal{C}} \\ {v}_{\mathcal{L}} } \in -\varmathbb{A} \bmat{ {v}_{\mathcal{C}} \\ {i}_{\mathcal{L}} }    
        &\iff \bmat{ u_{\pp_*} \\ u_{\pp} } = -\varmathbb{B} \bmat{ w_{\pp} \\ w_{\pp_*} }, %\\ &\quad
        % \text{ where }
        % \bmat{ Q_{\pp, \pp^*}^{\intercal} & 0 \\ 0 & Q_{\pp, \pp^*}^{\intercal} } 
        % \bmat{ w_{\pp} \\ w_{\pp_*} \\ u_{\pp_*} \\ u_{\pp} } 
        % =  \bmat{ J_{\pp, \pp^*} & I_{|\mL| + |\mC|} - J_{\pp, \pp^*}  \\ I_{|\mL| + |\mC|} - J_{\pp, \pp^*}  & J_{\pp, \pp^*} } \bmat{ v_{\mC} \\ i_{\mL} \\ i_{\mC} \\ v_{\mL} } .
        % \\&\iff \bmat{ \hat{u}_{\pp_*} \\ \hat{u}_{\pp} } = -Q_{\pp, \pp^*}^{\intercal} \varmathbb{B} \, Q_{\pp, \pp^*} \bmat{ \hat{w}_{\pp} \\ \hat{w}_{\pp_*} },  \\ &\qquad
        % \text{ where }
        % \pr{ \bmat{ \hat{w}_{\pp} \\ \hat{w}_{\pp_*} } , \bmat{ \hat{u}_{\pp_*} \\ \hat{u}_{\pp} } } 
        % =   \pr{ \Pi_{K_{\pp, \pp^*}} \bmat{ v_{\mC} \\ i_{\mL} } 
        %     + \Pi_{K_{\pp, \pp^*}^{\perp}} \bmat{ i_{\mC} \\ v_{\mL} },
        %     \Pi_{K_{\pp, \pp^*}} \bmat{ i_{\mC} \\ v_{\mL} } 
        %     + \Pi_{K_{\pp, \pp^*}^{\perp}} \bmat{ v_{\mC} \\ i_{\mL} } } .
    \end{align*}
    where
    \[
        \bmat{ Q_{\pp, \pp^*}^{\intercal} & 0 \\ 0 & Q_{\pp, \pp^*}^{\intercal} } 
        \bmat{ w_{\pp} \\ w_{\pp_*} \\ u_{\pp_*} \\ u_{\pp} } 
        =  \bmat{ J_{\pp, \pp^*} & I_{|\mL| + |\mC|} - J_{\pp, \pp^*}  \\ I_{|\mL| + |\mC|} - J_{\pp, \pp^*}  & J_{\pp, \pp^*} } \bmat{ v_{\mC} \\ i_{\mL} \\ i_{\mC} \\ v_{\mL} } .
    \]
    Therefore,
    for $K_{\pp, \pp^*} = \range(J_{\pp, \pp^*})$ we have
    \begin{align*}
        \varmathbb{A}_{K_{\pp,\pp^*}} = Q_{\pp, \pp^*}^{\intercal} \varmathbb{B} \, Q_{\pp, \pp^*}.
    \end{align*}
    Since $\varmathbb{B}$ is a maximal monotone operator with $\dom\varmathbb{B} = \reals^{|\mL| + |\mC|}$, 
    we have $Q_{\pp, \pp^*}^{\intercal} \varmathbb{B} \, Q_{\pp, \pp^*}$ is maximal monotone. 
    Finally from \citep[Proposition 20.44, (v)]{BauschkeCombettes2017_convex}, we conclude $\varmathbb{A}$ is maximal monotone. 
    % Note from the assumption $\unnamed\ne\emptyset$, we have $\dom\varmathbb{A} \ne\emptyset$. 

    By applying Theorem~\ref{thm-AC} and its remark, we know there is a unique Lipschitz continuous curve $({v}_{\mathcal{C}}, {i}_{\mathcal{L}}) \colon [0,\infty) \to \reals^{|{\mL}|} \times \reals^{|{\mC}|}$ that satisfies 
    \[
        \frac{d}{dt} \bmat{ v_{\mC} (t) \\ i_{\mL} (t) }
        = - m\varmathbb{A} \bmat{ v_{\mC} (t) \\ i_{\mL} (t) }
        \in - \varmathbb{A} \bmat{ v_{\mC} (t) \\ i_{\mL} (t) }
    \]
    for almost all $t\in[0,\infty)$, where $m \varmathbb{A}$ is the minimum-norm selection of $\varmathbb{A}$. 
    Let $\bmat{ i_{\mathcal{C}} (t) \\ v_{\mathcal{L}} (t) } =- m\varmathbb{A} \bmat{ {v}_{\mathcal{C}} (t) \\ {i}_{\mathcal{L}} (t) }$. 
    Moreover, the definition of $\varmathbb{A}$ implies the existence of accompanying curves $v_{\mathcal{R}}$, $i_{\mathcal{R}}$, $x$ and $y$ that satisfy KCL, KVL and V-I relations. 
    This concludes the existence of the curve. 

\emph{(ii) The whole flow $(v,x,i,y)$ is well-posed and Lipschitz continuous.} \\
    Finally, we show other curves besides $(v_{\mathcal{L}}, i_{\mathcal{C}})$ are defined uniquely and Lipschitz continuous. 
To do so, we prove  there is a Lipschitz continuous function $\mathcal{G}\colon \reals^{|\mC|} \times \reals^{|\mL|} \to \reals^\sigma \times \reals^\sigma \times \reals^m \times \reals^m$ that satisfies
\[
    \mathcal{G}(v_{\mathcal{L}}, i_{\mathcal{C}}) = (v, i, x, y).
\]
We prove the claim by finding the explicit expression of the component functions of $G$. 
We first show one key equation
\begin{equation} \label{eq-iC-vL}
    w_{\pp_*}
    = -(H_{\pp_*}^{\pp_*})^{\intercal} u_{\pp_*}.
\end{equation}
From \eqref{eq-reduced-H} we have $H_{\pp_*}^{\pp_*} w_{\pp_*} = u_{\pp_*}$. 
And as $H$ is skew-symmetric, we have $H_{\pp}^{\pp} = - (H_{\pp_*}^{\pp_*})^{\intercal}$. 
The core information we additionally use here, is the V-I relations $\frac{d}{dt} w_{\pp} = u_{\pp_*}$. 
As differentiation is a linear operation, we get \eqref{eq-iC-vL} by following
\[
    -\pr{ H_{\pp_*}^{\pp_*} }^{\intercal} u_{\pp_*}
        = H_{\pp}^{\pp} u_{\pp_*}
        = H_{\pp}^{\pp} \pr{ \frac{d}{dt} w_{\pp} }
        = \frac{d}{dt} \pr{ H_{\pp}^{\pp} w_{\pp} }
        = \frac{d}{dt} u_{\pp}
        = w_{\pp_*}.
\]
% Note $w_{\pp_*} = \bmat{ w_{i_{\mC}} \\ w_{v_{\mL}} }$ is not used while expressing the resistor values. 
Recall, from \eqref{eq-u_p-u_p*-by-w_p-w_p*} we have
\[
    u_{\pp_*}
    % &= H_{\pp}^{\pp_*} w_{\pp} + H_{\pp_*}^{\pp_*} w_{\pp_*} + H_{\Rm}^{\pp_*} w_{\Rm} 
    = \pr{ H_{\pp}^{\pp_*}+  H_{\Rm}^{\pp_*} \pr{ Q_\Rm (\nabla F)_{K_\Rm} Q_\Rm^{\intercal} - H_{\Rm}^{\Rm} }^{-1}  H_{\pp}^{\Rm} }  w_{\pp} 
        + H_{\pp_*}^{\pp_*} w_{\pp_*}.
\]
Moving the last term of right hand side to left hand side, 
multiplying both sides by $-(H_{\pp_*}^{\pp_*})^{\intercal} $ and using \eqref{eq-iC-vL}, we 
have
\[
    \pr{ \I + (H_{\pp_*}^{\pp_*})^{\intercal} H_{\pp_*}^{\pp_*}  } w_{\pp_*}
    % &= - (H_{\pp_*}^{\pp_*})^{\intercal} \pr{ u_{\pp_*} - H_{\pp_*}^{\pp_*} w_{\pp_*} }
    = -(H_{\pp_*}^{\pp_*})^{\intercal} \pr{ H_{\pp}^{\pp_*}+  H_{\Rm}^{\pp_*} \pr{ Q_\Rm (\nabla F)_{K_\Rm} Q_\Rm^{\intercal} - H_{\Rm}^{\Rm} }^{-1}  H_{\pp}^{\Rm} } w_{\pp}.
\]
Since $\I + (H_{\pp_*}^{\pp_*})^{\intercal} H_{\pp_*}^{\pp_*} \succ 0$ its inverse exists, we conclude 
\[
    w_{\pp_*}
    = 
    \underbrace{ -\pr{ \I + (H_{\pp_*}^{\pp_*})^{\intercal} H_{\pp_*}^{\pp_*} }^{-1}
    (H_{\pp_*}^{\pp_*})^{\intercal} \pr{ H_{\pp}^{\pp_*}+  H_{\Rm}^{\pp_*} \pr{ Q_\Rm (\nabla F)_{K_\Rm} Q_\Rm^{\intercal} - H_{\Rm}^{\Rm} }^{-1}  H_{\pp}^{\Rm} } }_{=:\varmathbb{C}} w_{\pp}.
\]
Organizing, we have
\[
    \bmat{ w_{\pp} \\ w_{\pp_*} \\ w_{\Rm} }
    = 
    \pr{ \bmat{ \I \\ 0 \\ 0  } + \bmat{ 0 \\ \varmathbb{C} \\ 0 }
        + \bmat{  0 \\ 0 \\  \pr{ Q_\Rm (\nabla F)_{K_\Rm} Q_\Rm^{\intercal} - H_{\Rm}^{\Rm} }^{-1}  H_{\pp}^{\Rm}   }    } 
    w_{\pp}.
\]
From Lemma~\ref{lem-partialF-fulldomain-continuous} we know $\pr{ Q_\Rm (\nabla F)_{K_\Rm} Q_\Rm^{\intercal} - H_{\Rm}^{\Rm} }^{-1}$ is Lipschitz continuous, and clearly linear operators are Lipschitz continuous, so ${\varmathbb{C}}$ is Lipschitz continuous as it is composition and sum of Lipschitz continuous functions. 
Therefore $w_{\pp} \mapsto w$ is Lipschitz continuous. 
Finally since $u=Hw$ and $H$ is indeed Lipschitz continuous as a linear operator, mapping $w_{\pp}\mapsto u$ is also Lipschitz continuous. 
As
$(w,u)$ is rearrangement of $({v},{i},x,y)$ and 
% \begin{align*}
%     \bmat{ w \\ u  } =  \bmat{ Q & 0 \\ 0 & Q } \bmat{ J & I_{\sigma + m} - J  \\ I_{\sigma + m} - J  & J } \bmat{ v \\ x \\ i \\ y }
% \end{align*}
$w_{v_{\mC}}, w_{i_{\mL}}$ are component functions of ${v}_\mC$ and ${i}_\mL$, we get the desired result. 

For $(v,i,x,y)$ that satisfies \eqref{eq-reduced-dynamics} with proper initial value, 
we know $(v_\mC, i_{\mL})$ is uniquely defined by previous observation and $(v,i,x,y) = \mathcal{G}(v_\mC, i_{\mL})$ should hold, we conclude $(v,i,x,y)$ is uniquely determined since $\mathcal{G}$ is single valued. 
Furthermore, as $v_\mC(t)$ and $i_{\mL}(t)$ are Lipschitz continuous with respect to $t$, 
we have $(v(t),i(t),x(t),y(t)) = \mathcal{G}(v_\mC(t), i_{\mL}(t))$ is also Lipschitz continuous as it is composition of Lipschitz continuous functions. 
This concludes the proof.

\end{proof}

\clearpage
\section{Equilibrium condition}
Define the set of voltages and currents in the equilibrium of the interconnect 
% composed with $\partial f$ for a given potentials and currents at the ports $(x, y)$ with $y \in \partial f(x)$
% \BEQ\label{e-dyn-interconnect-partial-f-eq}
\[
D_{x, y}=\left\{(v,i) \mid Ai = \begin{bmatrix}
-y\\
0
\end{bmatrix}, v = A^\intercal \begin{bmatrix}
x\\
e
\end{bmatrix}, v_{\mathcal{R}} = D_\mathcal{R} i_{\mathcal{R}}, 
v_{\mathcal{L}}=0, i_{\mathcal{C}}=0 \right \}.
\]

\begin{lemma}\label{lem-ad-admiss-property}
Assume the dynamic interconnect is admissible.
Then for all $(v, i) \in D_{x, y}$
we have
\[
    v_{\mathcal{R}}=i_{\mathcal{R}}=0.
\]
\end{lemma}
\begin{proof}
First, note that for all $(v, i)$ in the dynamic interconnect
\BEQ\label{e-dyn-ic-vi}
\langle v, i \rangle = \langle e, Ai \rangle 
= \left \langle \begin{bmatrix}
x\\
 e
\end{bmatrix}, \begin{bmatrix}
-y\\
0
\end{bmatrix} \right \rangle
= -\langle x, y \rangle.
\EEQ
Now suppose $(v, i) \in D_{x, y}$, then we have
\[
\langle v, i \rangle = \langle v_{\mathcal{R}}, i_{\mathcal{R}} \rangle =  \|i_{\mathcal{R}}\|_{D_\mathcal{R}}^2 = -\langle x, y \rangle.
\]
From the admissibility assumption we have $x \in \mathcal{R}(E^\intercal)$
and $y \in \mathcal{N}(E)$.
Thus 
\[
\|i_{\mathcal{R}}\|_{D_\mathcal{R}}^2 = -\langle x, y \rangle = 0,
\]
which concludes the proof.
\end{proof}

\begin{theorem}
\label{thm-equilibrium}
Assume the dynamic interconnect is admissible.
If $(x^\star,y^\star)$ is a primal-dual solution pair (with zero duality gap) for the optimization problem, then
there exist $v_{\mathcal{C}} \in \reals^{|C|}$ and $i_{\mathcal{L}} \in \reals^{|L|}$ such that
\[
\left ((0, 0, v_{\mathcal{C}}), (0, i_{\mathcal{L}}, 0)\right) \in D_{x^\star, y^\star}.
\]
% is an equilibrium of the interconnect composed with $\partial f$.
Conversely, if $y \in \partial f (x)$ and
\[
\left((v_{\mathcal{R}}, v_{\mathcal{L}}, v_{\mathcal{C}}), (i_{\mathcal{R}}, i_{\mathcal{L}}, i_{\mathcal{C}}) \right) \in D_{x, y},
\]
then
$v_{\mathcal{R}}=i_{\mathcal{R}}=0$ and 
$(x, y)$ is a primal-dual solution pair (with zero-duality) for the optimization problem.
\end{theorem}
\begin{proof}
First, observe the admissibility assumption can be rewritten as 
\[
\left\{(x,y) \mid \exists (v,i)  %\in \reals^{\sigma} \times \reals^{\sigma} 
\text{ such that } (v,i) \in D_{x,y} \right \}
= \mathcal{R}(E^\intercal) \times \mathcal{N}(E).
\]
Now suppose $(x^\star, y^\star) \in X^\star\times Y^\star$. 
Then by Karush-Kuhn-Tucker (KKT) optimality conditions, %\eqref{e-primal-dual-sol}, 
we have $(x^\star, y^\star) \in \mathcal{R}(E^\intercal) \times \mathcal{N}(E)$. 
Thus there exists $(v^\star, i^\star)$ such that
\[
(v^\star, i^\star) = \left((v_{\mathcal{R}}^\star, 0, v_{\mathcal{C}}^\star), (i_{\mathcal{R}}^\star, i_{\mathcal{L}}^\star, 0) \right) \in D_{x^\star, y^\star}.
\]
Furthermore from Lemma~\ref{lem-ad-admiss-property} we have $v_{\mathcal{R}}^\star = i_{\mathcal{R}}^\star = 0$. 
Therefore $v_{\mathcal{C}}^\star$, $i_{\mathcal{L}}^\star$ are the vectors that satisfiy the desired statement. 

Conversely, suppose 
\[
(v, i) = \left((v_{\mathcal{R}}, v_{\mathcal{L}}, v_{\mathcal{C}}), (i_{\mathcal{R}}, i_{\mathcal{L}}, i_{\mathcal{C}}) \right) \in D_{x, y}.
\]
From Lemma~\ref{lem-ad-admiss-property}, we have $v_{\mathcal{R}} = i_{\mathcal{R}} = 0$. 
Moreover, since there exists $(v,i)$ such that $(v,i) \in D_{x, y}$, 
by admissibility assumption we have $(x,y) \in \mathcal{R}(E^\intercal) \times \mathcal{N}(E)$. 
Finally, given the assumption that $y \in \partial f (x)$, 
by using KKT optimality conditions, 
we conclude $(x,y) \in X^\star\times Y^\star$. 

\end{proof}

\clearpage
\section{Energy dissipation analysis}\label{s-energy-dissipation}

In this section, we provide the proof of Theorem~\ref{thm:convergence}. 
The energy function \eqref{e-total-energy} used in the proof is related to the Lyapunov function considered in \cite{Willems1971_generation}. 
However, the dissipativity theory presented in \cite{Willems1971_generation, Willems1972_dissipative} does not directly apply to our setup. 
In our setup, we allow cases where $v_C, i_L$ oscillate, for example, a circuit with a
disconnected $L-C$ loop.
The proof is obtained by combining Barbalat's lemma \citep[Lemma 8.2]{khalil2002nonlinear} with Theorem~\ref{thm-well-posedness}.

\begin{proof} [Proof of Theorem~\ref{thm:convergence}]
    Let $(x^\star, y^\star)$ be a primal-dual solution pair.  Then by Theorem~\ref{thm-equilibrium}, there is $(v^\star, i^\star)\in D_{x^\star, y^\star}$ that satisfies
\BEQ\label{e-pd-pair-def}
(v^\star, i^\star) = ((0, 0, v_{\mathcal{C}}^\star), (0, i_{\mathcal{L}}^\star, 0)).
\EEQ
In particular, $i_{\mathcal{C}}^\star=0$ and $v_{\mathcal{L}}^\star=0$.
Define the total energy at time $t$ as
\[
\mathcal{E}(t)=\frac{1}{2}\|v_\mathcal{C}-v^\star_\mathcal{C}\|^2_{D_\mathcal{C}}+
\frac{1}{2}\|i_{\mathcal{L}}-i^\star_{\mathcal{L}}\|^2_{D_\mathcal{L}}.
\]
Then the power at time $t$ is given by
\BEA\label{e-sec-convergence-power}
\frac{d}{dt}   \mathcal{E}(t)
&=&
\langle v_{\mathcal{C}}-v_{\mathcal{C}}^\star,
D_{\mathcal{C}}\dot v_{\mathcal{C}}
\rangle 
+
\langle i_{\mathcal{L}}-i_{\mathcal{L}}^\star,
D_{\mathcal{L}}\dot i_{\mathcal{L}}
\rangle \nonumber\\
&=&
\langle v_{\mathcal{C}}-v_{\mathcal{C}}^\star,
i_{\mathcal{C}} -i_{\mathcal{C}}^\star
\rangle 
+
\langle i_{\mathcal{L}}-i_{\mathcal{L}}^\star,
v_{\mathcal{L}}-v^\star_\mathcal{L}
\rangle \nonumber\\
&=&
-
\langle v_{\mathcal{R}}-
\cancelto{0}{v_{\mathcal{R}}^\star},
i_{\mathcal{R}}-
\cancelto{0}{i_{\mathcal{R}}^\star}
\rangle
-
\langle x - x^\star,
y - y^\star
\rangle\\
&=&
-
\|i_{\mathcal{R}}\|_{D_\mathcal{R}}^2
-
\underbrace{ \langle x - x^\star,
y - y^\star}
_{ \ge 0}
\rangle \nonumber\\
&\le 0. \nonumber
\EEA
where we used \eqref{e-dyn-ic-vi} and
the monotonicity of $\partial f$. 
Therefore $\mathcal{E}(\infty) = \lim_{t\to\infty} \mathcal{E}(t)$ exists. 
Now, integrating from $0$ to $\infty$ we have
\begin{align*}
    0 
    % \le \int_0^{\infty} \pr{ \|i_{\mathcal{R}}(t)\|_{D_\mathcal{R}}^2
    %     + \langle x(t) - x^\star, y(t) - y^\star  \rangle } dt
    \le \int_0^{\infty}  \langle x(t) - x^\star, y(t) - y^\star  \rangle dt
    \le \mathcal{E}(0) - \mathcal{E}(\infty) < \infty.
\end{align*}
From Theorem~\ref{thm-well-posedness} we know the integrand is Lipschitz continuous, by Barbalat's lemma %\citep[Lemma 8.2]{khalil2002nonlinear}
we get
\begin{align*}
    % \lim_{t\to\infty} \|i_{\mathcal{R}} (t) \|_{D_\mathcal{R}}^2 =0, 
    % \quad 
    \lim_{t\to\infty} \langle x(t) - x^\star, y(t) - y^\star  \rangle = 0.
\end{align*}
Since $f$ is $\mu$-strongly convex and $M$-smooth, $\nabla f$ and $(\nabla f)^{-1}$ are strictly monotone, 
we conclude
\begin{align*}
    \lim_{t\to\infty} x(t) = x^\star, \quad 
    \lim_{t\to\infty} y(t) = y^\star,
\end{align*}
which is our desired result.

\end{proof}

% Now as 
% \begin{align*}
%     (\underline{v}(0), \underline{i}(0), \underline{x}(0), \underline{y}(0) )
%     = G( \underline{v}_\mC(0), \underline{i}_\mL(0) )
%     = G( \bar{v}_\mC, \bar{i}_\mL )
%     = \lim_{l\to\infty} G( v_\mC(t_l), i_\mL(t_l) ).
% \end{align*}
% Therefore
% \begin{align*}
%     \lim_{l\to\infty} \frac{d}{dt} \mathcal{E}(t_l) = 0.
% \end{align*}
% We have $\lim_{l\to\infty} x(t_l) = x^\star$, $\lim_{l\to\infty} y(t_l) = y^\star$. 
% % True for t_l that make v_C, i_L converge $t_l$, we have $\lim_{t\to\infty} \frac{d}{dt} \mathcal{E}(t) =0$.  
% $x$, $y$ are bounded and every converging sequence converges to $x^\star$ and $y^\star$, we're done. 

\section{Centralized classical algorithms}\label{appendix-centralized}

\subsection{Resistors and Moreau envelope}\label{sec-resistors-Moreau-envelope}
For $R>0$, define the Moreau envelope of $f\colon\reals^m\to\reals$ of parameter $R$ as
\[
^R f(x) = \inf_{z \in \reals^m} \left ( f(z) + \frac{1}{2R} \|z - x\|^2_2\right ).
\]
Then $^Rf$ is $1/R$-smooth with gradient given by
\BEQ\label{e-moreau-nabla-f}
\nabla ^R f(x) = \frac{1}{R}(x - \prox_{Rf}(x)). 
\EEQ

In this section we show that composing linear resistors with $\partial f$
is equivalent to taking a Moreau envelope of $f$. See two circuits below.

\begin{figure}[ht]
    \MoreauEnvelopeVertical
    \vspace{-0.5cm}
\end{figure}

% Now let $x = (x_1, \dots, x_m)$ and $\tilde x = (\tilde x_1, \dots, \tilde x_m)$ be the potentials at nodes $1, \dots, m$ and $m+1, \dots, 2m$ respectively. 
By KCL and Ohm's law for the first circuit, we have
\[
\frac{1}{R}(x - \tilde x) = i \in \partial f(\tilde x),
\]
which is equivalent to $\tilde x = \prox_{Rf}(x)$.
Using identity for the gradient of the Moreau envelope, we get
\[
\nabla ^R f(x) = \frac{1}{R}(x - \prox_{Rf}(x)) = \frac{1}{R}(x - \tilde x) = i.
\]
Therefore, the V-I relation on $m$ pins of $x$ in both circuits is identical.

As a consequence, consider $f$ to be $1/R$-smooth. Let $\tilde f$ be pre-Moreau envelope of $f$, \ie, $^R \tilde f = f$. 
Note that $\tilde f$ is a convex function.
Then from the series connection of the resistors ($-R$ in series with $R$ is the same as $0$-ohm resistor), we get the equivalence of the two circuits below.
\begin{figure}[ht]
    \PreMoreauEnvelope
    \vspace{-0.5cm}
\end{figure}

Note that for this circuit $x = \tilde x - R \nabla f(\tilde x)$. 
% by Ohm's law and KCL.

\subsection{Gradient flow} \label{sec-grad-flow}
Let $f\colon\reals^m\to\reals$ be a convex function. 
Consider the circuit below.
\begin{figure}[H]
    \CircuitGD
    \vspace{-0.5cm}
\end{figure}
 
Let $x$ be the potentials at $m$ pins of $\partial f$, and $y$ be the current entering those pins. 
Applying KCL and the V-I relations of the capacitor we get
\[
   D_\mC \frac{d}{dt} v_{\mC} = i_{\mC} = -y \in -\partial f(x).
\]
Since $e$ is connected to ground, we have $v_{\mC} = x - e = 0$. 
The resulting differential inclusion is
\[
    \frac{d}{dt} x \in - D_\mC^{-1} \partial f(x).
\]

For the automatic discretization of the continuous-time dynamics
of this circuit, we define the energy at time $k$ as
\[
\mathcal{E}_k = \frac{1}{2}\|x^k - x^\star\|^2_{D_\mC}.
\]

\subsection{Nesterov acceleration}  \label{s-nesterov-acceleration}

Let $f\colon\reals^m\to\reals$ be a $1/R$-smooth convex function. 
Consider the circuit below.
\begin{figure}[H]
    \CitcuitNesterov
    \vspace{-0.5cm}
\end{figure}

Observe, by Ohm's law and $y = \nabla f(x)$ we have
\[
    x^+ = x - R y = x - R \nabla f(x).
\]

From KCL, KVL, and V-I relations we have
\BEAS
    \frac{d}{dt} i_{\mL} &=& D_{\mL}^{-1} (v_{\mC}-x^{+})  \\
    \frac{d}{dt} v_{\mC} &=& - D_{\mC}^{-1} \nabla f(x). 
\EEAS
Applying KCL and Ohm's law at $x$, it follows
\[
    \nabla f(x) 
    = i_{\mL} + \frac{1}{R} (v_{\mC} - x^{+}),
\]
which implies that
\[
x = v_{\mC} + R i_{\mL}.
\]
Differentiating the above equality twice and plugging in the V-I relations, we get
\[
    \frac{d^2}{dt^2} x
    = -R (D_{\mL} D_{\mC})^{-1} \nabla f(x) - R D_{\mL}^{-1} \frac{d}{dt} x - ( D_{\mC}^{-1} - R^2 D_{\mL}^{-1} ) \frac{d}{dt} \nabla f(x).
\]
Reorganizing, we conclude
\BEQ\label{e-nesterov-accel-ode}
    \frac{d^2}{dt^2} x + R D_{\mL}^{-1} \frac{d}{dt} x 
    + ( D_{\mC}^{-1} - R^2 D_{\mL}^{-1} ) \frac{d}{dt} \nabla f(x)
    + R (D_{\mL} D_{\mC})^{-1} \nabla f(x) = 0. 
\EEQ

Under the proper selection of parameters for $\mu$-strongly convex and $L$-smooth function $f$, \eqref{e-nesterov-accel-ode} corresponds to the high-resolution ODE for NAG-SC introduced in \cite{ShiDuSuJordan2019_acceleration}
\[
    \frac{d^2}{dt^2} {x} + 2\sqrt{\mu} \frac{d}{dt}{x} + \sqrt{s} \frac{d}{dt} \nabla f(x) + (1 + \sqrt{\mu s} ) \nabla f(x) = 0. 
\]
As an immediate consequence, if we set $R = \frac{1}{4 \mu }$, $L_i = \frac{1}{8 \mu \sqrt{\mu}}$, $C_i = 2 \sqrt{\mu}$, we recover the low-resolution ODE of NAG-SC
\[
    \frac{d^2}{dt^2} x + 2\sqrt{\mu} \frac{d}{dt} x 
    +  \nabla f(x) = 0. 
\]

For the automatic discretization of the continuous-time dynamics
of this circuit, we define the energy at time $k$ as
\[
\mathcal{E}_k = \frac{1}{2}\|v_{\mC}^k - x^\star\|^2_{D_\mC} 
+ \frac{1}{2}\|i_{\mL}^k - y^\star\|^2_{D_\mL}.
\]

\subsection{Proximal point method}

Consider the circuit below.
\begin{figure}[H]
    \CircuitPPM
    \vspace{-0.5cm}
\end{figure}

\vspace{0.1cm}
Then from the discussion in \S\ref{sec-resistors-Moreau-envelope}, the above circuit is equivalent to the circuit below.

\begin{figure}[H]
    \CircuitPPMMoreau
    \vspace{-0.5cm}
\end{figure}
According to \S\ref{sec-grad-flow}, the ODE for the above circuit is
\BEQ\label{e-prox-point-method-ode}
    \frac{d}{dt} x = - D_\mC^{-1} \nabla^{R} f(x).
\EEQ
Since from \eqref{e-moreau-nabla-f} we have
\[
\prox_{Rf}(x) = x - R\nabla ^R f(x),
\]
 this circuit gives a continuous model for the proximal point method. 
 
Applying Euler discretization to \eqref{e-prox-point-method-ode} 
with a stepsize of $C_i R$ for each $i$th coordinate,
we recover proximal point method
\[
    x^{k+1} = x^{k} - R \nabla^{R} f(x^k) =  
    x^{k} - ( x^k - \prox_{R f} (x^k) ) = \prox_{R f} (x^k).
\]

For the automatic discretization of the continuous-time dynamics
of this circuit, we define the energy at time $k$ as
\[
\mathcal{E}_k = \frac{1}{2}\|x^k - x^\star\|^2_{D_\mC}.
\]

\subsection{Proximal gradient method}\label{sec-prox-grad}

Let $f\colon\reals^{m}\to\reals$ be $1/R$-smooth convex function,
and $g\colon\reals^{m}\to\reals$ be convex function.
Consider the circuit below.
\begin{center}
    \CircuitProxGrad
\end{center}
Observe, by the Ohm's law $e = x - R \nabla f(x)$.
Using KCL at $x$ and KVL at $e$, we get
\BEAS
    i_{\mC} &=& -\nabla f(x) - \nabla^R g (e) \\
    v_{\mC} &=& x - R \nabla f(x).
\EEAS
Applying V-I relation for the capacitor and eliminating $e$ gives 
% {\color{red} (and assuming $\nabla f(x)$ is differentiable. Note if $\nabla f$ is Lipschitz continuous then it is differentiable almost everywhere.) }
\[
     \frac{d}{dt} x - R \frac{d}{dt} \nabla f(x) 
    =  \frac{d}{dt} v_{\mC}
    % = \frac{1}{C} i_{\mC}  
    = -\frac{1}{C} \pr{ \nabla f(x) + \nabla^R g (x - R \nabla f(x)) }.
\]
Organizing
\[
     \frac{d}{dt} x
    = - \frac{1}{C} \pr{ \nabla^{R} g (I - R \nabla f) + \nabla f } (x) +  R \frac{d}{dt} \nabla f(x) .
\]
We can show that $R \norm{ \frac{d}{dt} \nabla f(x) }<M$ for some $M>0$, thus
 %(There would be a possibility that this assumption can be achieved from the assumption $\nabla f$ is $R$-cocoercieve. But maybe need more, since we may furthermore need differentiability of $x$ and boundedness of $\norm{\dot{x}}$. 
\begin{equation} \label{e-FBS-ODE-circuit}
     \frac{d}{dt} x
    = - \frac{1}{CR} \pr{ \pr{ R \nabla^{R} g (I - R \nabla f) + R \nabla f } (x) +  O\pr{MCR} }.
\end{equation}
Applying Euler discretization with stepsize $CR$ we have
\[
    \frac{x^{k+1} - x^{k}}{CR}
    = - \frac{1}{CR} \pr{ \pr{ R \nabla^{R} g (I - R \nabla f) + R \nabla f } (x^{k}) +  O\pr{ MCR} }.
\]
Multiplying $CR$ on both sides and reorganizing gives
\BEAS
    x^{k+1}
    &=& x^{k} - \pr{ R \nabla^{R} g (I - R \nabla f) + R \nabla f } (x^{k}) + O(MCR) \\
    &=& \pr{ \prox_{Rg} - I } (I - R \nabla f) (x^k) + \pr{ I - R \nabla f } (x^k) + O(MCR)  \\
    &=& \prox_{Rg} (I - R \nabla f) (x^k) + O(MCR) .
\EEAS
If we set $C \ll R$, we recover 
the proximal gradient method. 

For the automatic discretization of the continuous-time dynamics
of this circuit, we define the energy at time $k$ as
\[
\mathcal{E}_k = \frac{C}{2}\|e^k - e^\star\|^2_2,
\]
where $e^\star=x^\star - R \nabla f(x^\star)$.

\subsection{Primal decomposition}

Let $f_1, \ldots, f_N\colon\reals^{m}\to\reals$ be convex functions. 
Consider the circuit below. 
\begin{figure}[H]
    \CircuitPrimalDecomposition
    \vspace{-0.5cm}
\end{figure}

Let $x_1, \dots, x_N \in \reals^{m}$ be vectors of potentials at pins of $\partial f_1, \ldots, \partial f_N$ respectively.
From KVL, we have
\[
    e = x_1 = \cdots = x_N = v_{\mC}.
\]
Using KCL at $e$ and the V-I relation of nonlinear resistors we 
get $\sum_{j=1}^{N} y_j + i_{\mC} = 0$,
where  $y_j \in \partial f_j(x_j)$. 
Using the V-I relation for capacitor we have 
\[
    \frac{d}{dt} e = -\frac{1}{C} \sum_{j=1}^{N} y_j.
\]
Discretizing above V-I relations, 
we recover primal decomposition
\BEAS
    % x^{k} &=& v_{\mC}^k \\
    y_j^{k} &\in& \partial f_j(x_j^k) \\ 
    e^{k+1} &=& e^k - \frac{h}{C} \sum_{j=1}^{N} y_j^{k}.
\EEAS

For the automatic discretization of the continuous-time dynamics
of this circuit, we define the energy at time $k$ as
\[
\mathcal{E}_k = \frac{C}{2}\|e^k - x^\star\|^2_2.
\]

\subsection{Dual decomposition}

Let $f_1, \ldots, f_N\colon\reals^{m}\to\reals$ be convex functions. 
Consider the circuit below. 
\begin{figure}[H]
    \CitcuitDualDecomposition
    \vspace{-0.5cm}
\end{figure}

Using KCL at $x_j$ and V-I relation for nonlinear resistors we get
\[
    x_j \in \partial f_j^*(y_j) = \partial f_j^*(i_{\mL_j}).
\]
Using KCL at $e$ yields $\sum_{j=1}^{N} y_j = 0$.
Using KVL and V-I relation for inductors we get 
\[
    e - x_j = v_{\mL_j} = L \frac{d}{dt} i_{\mL_j}.
\]
Summing over $j=1, \ldots, N$ gives
\[
    N e - \sum_{j=1}^{N} x_j
    = L\sum_{j=1}^{N} \frac{d}{dt} i_{\mL_j}
    = L\frac{d}{dt} \sum_{j=1}^{N} y_j = 0,
\]
leading to $e = (1/N) \sum_{j=1}^{N} x_j$.  
Discretizing above V-I relations, 
we recover dual decomposition
\BEAS
    x_j^{k} &\in& \partial f_j^* (i_{\mL_j}^k) \\
    e^k &=& \frac{1}{N} \sum_{j=1}^{N} x_j^k \\
    i_{\mL_j}^{k+1}  &=& i_{\mL_j}^k - \frac{h}{L} (e^k - x_j^k).
\EEAS

For the automatic discretization of the continuous-time dynamics
of this circuit, we define the energy at time $k$ as
\[
\mathcal{E}_k = \sum_{j=1}^N \frac{L}{2}\|i_{\mL_j}^k - y_j^\star\|^2_2.
\]

\subsection{Proximal decomposition}\label{sec-prox-decomposition}
Let $f_1, \ldots, f_N\colon\reals^{m}\to\reals$ be convex functions. 
Consider the circuit below. 
\begin{figure}[H]
    \CircuitProximalDecomposition
    \vspace{-0.5cm}
\end{figure}

Let $x_1, \dots, x_N \in \reals^{m}$ be vectors of potentials at pins
of $\partial f_1, \ldots, \partial f_N$ respectively. Define $e \in \reals^{m}$ to be a vector of potentials on the bottom of the circuit.
Observe, by Ohm's law and KCL we have
\[
    y_j = i_{\mL_j} + \frac{1}{R}(e - x_j) \in \partial f_j(x_j).
\]
It implies that $x_i = \prox_{Rf_i}(e + Ri_{\mL_i})$.
The V-I relations for inductors are given by
\[
    \frac{d}{dt} i_{\mL} = v_{\mL}/L = (E^\intercal e-x)/L,
\]
where $E^\intercal=(I, \ldots, I) \in \reals^{ Nm \times m}$.

Further, note that by KCL $Ey = \sum_{j=1}^N y_j = 0$,
therefore $E\frac{d}{dt} y = \frac{d}{dt} Ey =0$. 
Using the above V-I relations for $g = N e - E x$ we get the following ODE
\BEQ\label{e-admm-kcl}
\dot g = \frac{d}{dt} RE(y - i_{\mL}) = -RE \frac{d}{dt}i_{\mL}
= -\frac{R}{L}E(E^\intercal e-x) = -\frac{R}{L} g.
\EEQ
% The solution to an ODE above has the form $g(t) = c_1 e^{-tR/L}$, 
We initialize circuit with $E i_{\mL}(0)=0$ and $EE^\intercal=NI$ gives
\[
0=Ey(0) = E(i_{\mL}(0) + (E^\intercal e(0) - x(0))/R) = -\frac{1}{R} g(0).
\]
Thus the solution to an ODE \eqref{e-admm-kcl} is $g=0$ and
we conclude $e = \frac{1}{N}Ex$.
% Thus $c_1=0$ and we conclude $e = \frac{1}{N}Ex$.

The V-I relations for the circuit are
\BEAS\label{e-admm-vi-continuous}
x_j &=& \prox_{Rf_j}(e + Ri_{\mL_j}), \quad j=1, \ldots, N  \nonumber \\
e &=& \frac{1}{N} Ex \\ % \sum_{j=1}^N x_j \\
\frac{d}{dt} i_{\mL} &=& (E^\intercal e-x)/L. \nonumber
\EEAS
Discretizing above V-I relations we recover proximal decomposition
\BEAS
x_j^{k+1} &=& \prox_{Rf_j}(e^k + Ri_{\mL_j}^k), \quad j=1, \ldots, N  \nonumber \\
e^{k+1} &=& \frac{1}{N} Ex^k \\
i_{\mL}^{k+1} &=& i_{\mL}^k + \frac{h}{L}(E^\intercal e^{k+1}-x^{k+1}). \nonumber
\EEAS

For the automatic discretization of the continuous-time dynamics
of this circuit, we define the energy at time $k$ as
\[
\mathcal{E}_k = \sum_{j=1}^N \frac{L}{2}\|i_{\mL_j}^k - y_j^\star\|^2_2
+ \gamma \|e^k - x^\star\|_2^2,
\]
where $\gamma$ is a parameter that is being optimized, see \S\ref{appendix-automatic-discr}.

\subsection{Douglas–Rachford splitting}
Let $f, g\colon\reals^{m}\to\reals$ be convex functions.
Consider the circuit below.
\begin{figure}[H]
    \begin{center} 
        \CircuitDRS
    \end{center}
    \vspace{-0.5cm}
\end{figure}
Using KCL at $x_1$ and Ohm's law we get
\[
    \frac{1}{R} ( x_2 - x_1 ) + i_{\mathcal{L}} \in \partial g (x_1),
\]
which implies $x_1 = \prox_{R g} (x_2 + R i_{\mathcal{L}})$.
Similarly, using KCL at $x_2$ and Ohm's law we get
\[
    \frac{1}{R} ( x_1 - x_2 ) - i_{\mathcal{L}} \in \partial f (x_2),
\]
which implies $x_2 = \prox_{R f} (x_1 - R i_{\mathcal{L}})$.
From KVL and V-I relation for inductors we have
\[
    \frac{d}{dt} i_{\mathcal{L}} = \frac{1}{L} ( x_2 - x_1 ).
\]

Discretizing above V-I relations with $R=L=1$ and stepsize $h=1$, 
we recover Douglas–Rachford splitting
\BEAS
    x_1^{k+1} &=& \prox_{R g} (x_2^k + R i_{\mL}^k) \\
    x_2^{k+1} &=& \prox_{R f} (x_1^{k+1} - R i_{\mL}^k)  \\
    % i_{\mL}^{k+1} &=& i_{\mL}^k + \frac{h}{L} (x_2^{k+1} - i_{\mL}^{k+1}).
    i_{\mL}^{k+1} &=& i_{\mL}^k + \frac{h}{L} (x_2^{k+1} - x_1^{k+1}).
\EEAS

For the automatic discretization of the continuous-time dynamics
of this circuit, we define the energy at time $k$ as
\[
\mathcal{E}_k =  \frac{L}{2}\|i_{\mL}^k - y_1^\star\|^2_2
+ \gamma \|x_2^k - x^\star\|_2^2,
\]
where $\gamma$ is a parameter that is being optimized, see \S\ref{appendix-automatic-discr},
and $y_1^\star \in \partial g(x^\star)$.

\subsection{Davis-Yin splitting}
Let $f, g, h\colon\reals^{m}\to\reals$ be  convex functions,
with $h$ also being $1/S$-smooth.
Consider the circuit below.
\begin{figure}[H]
    \CircuitDYS
    \vspace{-0.5cm}
\end{figure}

Using KCL at $e$ and Ohm's law we have
\[
    x_1 = x_2 - S i_{-S} + S i_{S} = x_2.
\]
Applying KCL at $x_2$, we get
\BEQ\label{dys:KCL_on_2}
    i_{\mL} = \frac{e - x_2}{-S} - \nabla h(x_2).
\EEQ
Using KCL at $x_3$ and \eqref{dys:KCL_on_2} 
it follows
\[
     \frac{x_1-x_3}{R} + i_{\mathcal{L}} = \frac{x_1-x_3}{R} + \frac{x_2 - e}{S} - \nabla h(x_2)\in \partial g(x_3).
\]
Using KCL at $x_1$, we get
\[
 \frac{x_3-x_1}{R} + \frac{e-x_1}{S} \in \partial f(x_1).
\]
Organizing, and applying V-I relation for inductor
\BEAS
    x_3 &=& \prox_{Rg} \pr{ \pr{ 1 + \frac{R}{S} } x_1 - \frac{R}{S} e - R \nabla h(x_1)  } \\
    x_1 &=& \prox_{Rf} \pr{ x_3 + \frac{R}{S} ( e- x_1 ) } \\
    i_{\mL} &=& \frac{e - x_1}{-S} - \nabla h(x_1) \\
    \frac{d}{dt} i_{\mathcal{L}} &=& \frac{1}{L} ( x_1 -x_3 ).
\EEAS
Now we eliminate the term $i_{\mathcal{L}}$. 
Differentiating \eqref{dys:KCL_on_2}, applying $L\frac{d}{dt} i_{\mathcal{L}} = x_1 -x_3$ we get
\[
    \frac{1}{L} ( x_1 -x_3 )
    = \frac{d}{dt} i_{\mathcal{L}} = \frac{d}{dt}\frac{x_1 - e}{S} - \frac{d}{dt} \nabla h(x_1).
\]
In other words, 
\[
    \frac{d}{dt} e = \frac{S}{L} ( x_3 - x_1 ) + \frac{d}{dt} x_1 - S \frac{d}{dt} \nabla h(x_1). 
\]
Using ``alternating update'' and Euler discretization of $e$ and $x_1$, we have
\BEAS
    x_3^{k+1} &=& \prox_{Rg} \pr{ \pr{ 1 + \frac{R}{S} } x_1^{k} - \frac{R}{S} e^{k} - R \nabla h(x_1^{k})  } \\
    x_1^{k+1} &=& \prox_{Rf} \pr{ x_3^{k+1} + \frac{R}{S} ( e^{k}- x_1^{k} ) } \\
    e^{k+1} &=& e^{k} +  \frac{Sh}{L} ( x_3^{k+1} - x_1^{k+1} ) + x_1^{k+1}-x_1^{k} - S h \frac{d}{dt} \nabla h(x_1^{k}).
\EEAS
Set $R=S=h=\alpha$ and $L=\alpha^2$, then the above can be rewritten as
\BEAS
    x_3^{k+1} &=& \prox_{\alpha g} \pr{ 2x_1^{k} - e^{k} - \alpha  \nabla h(x_1^{k})  } \\
    x_1^{k+1} &=& \prox_{\alpha f} \pr{ e^{k} + x_3^{k+1} - x_1^{k} } \\
    e^{k+1} &=& e^{k} +  x_3^{k+1} - x_1^{k} - \alpha^2 \frac{d}{dt} \nabla h(x_1^{k}).
\EEAS
When $\norm{ \frac{d}{dt} \nabla h(x^{k}) }$ is bounded, then $\alpha^2 \frac{d}{dt} \nabla h(x^{k}) = O(\alpha^2)$. For
small $\alpha$ we may ignore this term and recover DYS. 

For the automatic discretization of the continuous-time dynamics
of this circuit, we define the energy at time $k$ as
\[
\mathcal{E}_k =  \frac{L}{2}\|i_{\mL}^k - y_1^\star\|^2_2
+ \gamma \|e_1^k - e^\star\|_2^2,
\]
where $\gamma$ is a parameter that is being optimized, see \S\ref{appendix-automatic-discr},
and $e^\star=x^\star - R(y_1^\star + y_3^\star)$,
$y_1^\star \in \partial f(x^\star)$, $y_3^\star \in \partial g(x^\star)$.

\clearpage
\section{Decentralized classical algorithms}\label{appendix-decentralized}

In a decentralized optimization setup, we are given a graph $G=(V, A)$ 
which defines the communication pattern between agents.
This means that each agent is constrained to communicate only to its
direct neighbors for the edges of the graph.

We define $\Gamma_j$ as the neighbors of $j$ in graph $G$.
For simplicity, in each example we 
only illustrate the circuit between components indexed by $j$ and $l$,
where $j$ and $l$
are connected through an edge in the graph $G$.

\subsection{Decentralized gradient descent}
Let $f_1, \ldots, f_N\colon\reals^{m}\to\reals$ be differentiable convex functions. 
Decentralized gradient descent (DGD) is derived as a gradient descent of the appropriate
penalty formulation of a decentralized problem.
Similarly, to construct a DGD circuit we apply the gradient flow circuit of \S\ref{sec-grad-flow}
to appropriate nonlinear resistors and arrive at the following circuit.
% \begin{center} 
% \includegraphics[width=0.35\textwidth]{figures/dgd.pdf}
% \end{center}
\begin{figure}[H]
    \CircuitDGD
    \vspace{-0.5cm}
\end{figure}
The right side of the circuit contains the graph with resistors $R_{jl}$ connecting vectors of potentials
$x_j \in \reals^{m}$ and $x_l\in \reals^{m}$ for every neighbors $j$ and $l$ in the given graph $G$.

Using the KCL at $x_j$ we get
\[
0 = \nabla f_j(x_j) + \sum_{l \in \Gamma_j} \frac{x_j - x_l}{R_{jl}} + i_{\mC_j}.
\]
Applying the V-I relation for the capacitors we get
\[
\frac{d}{dt} x_j = \frac{d}{dt} v_{\mC_j} = - \frac{1}{C} \left(\nabla f_j(x_j) + \sum_{l \in \Gamma_j} (x_j - x_l)/R_{jl}\right).
\]
Euler discretization recovers the DGD
\[
x_j^{k+1}
= \left(1 - \sum_{l \in \Gamma_j}\frac{h}{CR_{jl}}\right) x_j^k
+ \sum_{l \in \Gamma_j} \frac{h}{CR_{jl}} x_l^k - \frac{h}{C} \nabla f_j(x_j^k),
\]
with gradient stepsize $h/C$ and the mixing matrix
\[
W_{jl} = \left\{ \begin{array}{ll} 
    1 - \sum_{l \in \Gamma_j}\frac{h}{CR_{jl}} & \text{if }j=l \\
    \frac{h}{CR_{jl}} & \text{if } j \neq l, \quad l \in \Gamma_j \\
    0 & \text{otherwise}.
    \end{array}\right. 
\]

For the automatic discretization of the continuous-time dynamics
of this circuit, we define the energy at time $k$ as
\[
\mathcal{E}_k = \sum_{j=1}^N \frac{C}{2}\|x_j^k - x_j^\star\|^2_2.
\]

\subsection{Diffusion}
Let $f_1, \ldots, f_N\colon\reals^{m}\to\reals$ be $1/R$-smooth convex functions.
Decentralized gradient descent is derived as a forward-backward splitting fixed
point iteration of the appropriate
penalty formulation of a decentralized problem.
Similarly, to construct a diffusion circuit we apply the proximal
gradient circuit of \S\ref{sec-prox-grad}
to appropriate nonlinear resistors and arrive at the following circuit.
% \begin{center} 
% \includegraphics[width=0.35\textwidth]{figures/diffusion.pdf}
% \end{center}
\begin{figure}[H]
    \CircuitDiffusion
    \vspace{-0.5cm}
\end{figure}
The right side of the circuit is the graph with linear resistors $R_{jl}$ connecting vectors of potentials
$e_j \in \reals^{m}$ and $e_l\in \reals^{m}$ for every neighbors $j$ and $l$ in the given graph $G$.

By Ohm's law we have $v_{\mC_j} = e_j = x_j - R \nabla f_j(x_j)$.
Using the KCL at $e_j$ we get
\[
\nabla f_j(x_j) + \sum_{l \in \Gamma_j} \frac{e_j - e_l}{R_{jl}} + i_{\mC_j}=0 .
\]
Applying the V-I relation for the capacitors we get
\BEAS
\frac{d}{dt} e_j &=& \frac{d}{dt}x_j - R \frac{d}{dt} \nabla f_j(x_j) \\
&=& - \frac{1}{C} \left(\nabla f_j(x_j) + \sum_{l \in \Gamma_j} \frac{(x_j - R \nabla f_j(x_j)) - (x_l - R \nabla f_l(x_l))}{R_{jl}}\right).
\EEAS
We can show that $R \norm{ \frac{d}{dt} \nabla f_j(x_j) }<M$ for some $M>0$, thus
\[
\frac{d}{dt}x_j = - \frac{1}{C} \left(\nabla f_j(x_j) + \sum_{l \in \Gamma_j} \frac{(x_j - R \nabla f_j(x_j)) - (x_l - R \nabla f_l(x_l))}{R_{jl}}+O(MC)\right).
\]
Applying Euler discretization with stepsize $CR$ gives
\BEAS
x_j^{k+1} 
&= & \left(1 - \sum_{l \in \Gamma_j}\frac{R}{R_{jl}}\right) (x_j^k -  R \nabla f_j(x_j^k)) \\
&& + \sum_{l \in \Gamma_j} \frac{R}{R_{jl}} (x_l^k - R \nabla f_l(x_l^k)) + O(MCR),
\EEAS
with gradient stepsize $R$ and the mixing matrix
\[
W_{jl} = \left\{ \begin{array}{ll} 
    1 - \sum_{l \in \Gamma_j}\frac{R}{R_{jl}} & \text{if }j=l \\
    \frac{R}{R_{jl}} & \text{if } j \neq l, \quad l \in \Gamma_j \\
    0 & \text{otherwise}.
    \end{array}\right. 
\]
If we set $C \ll R$, we recover the diffusion method.

For the automatic discretization of the continuous-time dynamics
of this circuit, we define the energy at time $k$ as
\[
\mathcal{E}_k = \sum_{j=1}^N \frac{C}{2}\|e_j^k - e_j^\star\|^2_2,
\]
where $e_j^\star = x_j^\star - R \nabla f_j(x_j^\star)$ for all $j=1, \ldots, N$.

\subsection{Decentralized ADMM} \label{s-dadmm}
Let $f_1, \ldots, f_N\colon\reals^{m}\to\reals$ be convex functions.
Then the decentralized ADMM circuit is given below. 
% For simplicity we 
% only illustrate the circuit between $\partial f_j$ and $\partial f_l$
% where $j$ and $l$
% are connected through an edge in the graph $G$.

\begin{figure}[H]
    \CircuitDADMM
    \vspace{-0.5cm}
\end{figure} 
Note that this circuit is similar to the one in proximal decomposition in \S\ref{sec-prox-decomposition}
with the difference that instead of a single net $e$ we now have a net $e_{jl}$ for each
edge $(j,l)$ in graph $G$.
Denote currents on inductors to be ${i_{\mL}}_{jl} \in \reals^{m}$ and ${i_{\mL}}_{lj} \in \reals^{m}$.

Using KCL at $x_j$ we get
\BEQ\label{e-dec-admm-kcl}
\sum_{l \in \Gamma_j} \left ( {i_{\mL}}_{jl} + \frac{e_{jl} - x_j}{R}\right ) \in \partial f_j(x_j).
\EEQ
We initialize the circuit such that ${i_{\mL}}_{jl}(0) +{i_{\mL}}_{lj}(0)=0$
for each edge $(j,l)$ in graph $G$.
Now consider KCL at $e_{jl}$
\BEQ\label{e-dec-admm-kcl2}
{i_{\mL}}_{jl} + {i_{\mL}}_{lj} =  - \frac{(e_{jl} - x_j)}{R} - \frac{(e_{jl} - x_l)}{R}. 
\EEQ
Using V-I relation for inductor we also have
\[
\frac{d}{dt}{i_{\mL}}_{jl} + \frac{d}{dt}{i_{\mL}}_{lj} =
\frac{1}{L}(2e_{jl} - x_j - x_l).
\]
Combining the two equalities above we get an ODE
\[
\frac{d}{dt}\left ({i_{\mL}}_{jl} + {i_{\mL}}_{lj} \right) = -\frac{R}{L}
\left ({i_{\mL}}_{jl} + {i_{\mL}}_{lj} \right).
\]
Using initial conditions the solution of an ODE is ${i_{\mL}}_{jl} + {i_{\mL}}_{lj}=0$.
From \eqref{e-dec-admm-kcl2} we conclude that $e_{jl}=\frac{1}{2}(x_j+x_l)$.

Using \eqref{e-dec-admm-kcl}, we get the V-I relations for the circuit
\BEAS\label{e-dec-admm-vi-continuous}
x_j &=& \prox_{(R/|\Gamma_j|)f_j}\left (\frac{1}{|\Gamma_j|} \sum_{l \in \Gamma_j} (R{i_{\mL}}_{jl} + e_{jl})\right ) \\
e_{jl} &=& \frac{1}{2}(x_j+x_l) \\ 
\frac{d}{dt} {i_{\mL}}_{jl} &=& \frac{1}{L}(e_{jl} - x_j),
\EEAS
for every $j=1, \ldots, N$ and every edge $(j,l)$ in graph $G$.
Discretizing the V-I relations with stepsize $L/R$, 
we recover decentralized ADMM,
\BEAS 
x_j^{k+1} &=& \prox_{(R/|\Gamma_j|)f_j}\left ( \frac{1}{|\Gamma_j|} \sum_{l \in \Gamma_j} (R{i^k_{\mL}}_{jl} + e^k_{jl})\right ) \\
e_{jl}^{k+1} &=& \frac{1}{2}(x_j^{k+1}+x_l^{k+1}) \\ 
{i_{\mL}}_{jl}^{k+1} &=& {i_{\mL}}_{jl}^k + \frac{1}{R}(e_{jl}^{k+1} - x_j^{k+1}).
\EEAS 

For the automatic discretization of the continuous-time dynamics
of this circuit, we define the energy at time $k$ as
\[
\mathcal{E}_k = \sum_{\text{edge}~\{j,l\}} 
\left( \frac{L}{2}\|{i^k_{\mL}}_{jl} -  {i^\star_{\mL}}_{jl}\|^2_2
+ \frac{L}{2}\|{i^k_{\mL}}_{lj} - {i^\star_{\mL}}_{lj}\|^2_2
+ \gamma \|e_{jl}^k - x^\star\|_2^2 \right),
\]
where $\gamma$ is a parameter that is being optimized, see \S\ref{appendix-automatic-discr},
and ${i^\star_{\mL}}_{jl}$ is the current through inductor at equilibrium.

\subsection{PG-EXTRA} \label{s-PG-EXTRA}
Let $f_1, \ldots, f_N\colon\reals^{m}\to\reals$ be convex functions, and
$h_1, \ldots, h_N\colon\reals^{m}\to\reals$ be convex $M$-smooth functions.
Then the PG-EXTRA circuit is given below. 
\begin{figure}[ht]
    \begin{center} 
    \CircuitPgExtra
    \end{center}
    \vspace{-0.5cm}
\end{figure}
Denote current on inductor going from $x_j$ to $x_l$ to be ${i_{\mL}}_{jl} \in \reals^{m}$.

Recall \S\ref{sec-resistors-Moreau-envelope} and apply Ohm's law to get 
\[
\frac{e_j - \tilde x_j}{R} = \frac{e_j - \prox_{Rf_j}(e_j)}{R} = \nabla^R f_j(e_j) = \frac{x_j - e_j}{-R}.
\]
This yields $x_j = e_j - R\nabla^R f_j(e_j) = \prox_{Rf_j}(e_j)$.
Using KCL at $x_j$ we get
\BEQ\label{e-pgextra-kcl}
\frac{e_j - x_j}{-R} = \nabla h_j(x_j) + 
\sum_{l \in \Gamma_j} \left ( {i_{\mL}}_{jl} + \frac{x_j - x_l}{R_{jl}}\right ).
\EEQ

Define the mixing matrix
\[
W_{jl} = \left\{ \begin{array}{ll} 
    1 - \sum_{l \in \Gamma_j}\frac{R}{R_{jl}} & \text{if }j=l \\
    \frac{R}{R_{jl}} & \text{if } j \neq l, \quad l \in \Gamma_j \\
    0 & \text{otherwise}.
    \end{array}\right. 
\]
Rearranging the terms in~\eqref{e-pgextra-kcl} we get
\BEAS 
e_j &=& x_j - R\nabla h_j(x_j) - \sum_{l \in \Gamma_j} \left ( R{i_{\mL}}_{jl} + \frac{x_j - x_l}{R_{jl}/R}\right ) \\
&=& x_j - \sum_{l=1}^N W_{jl}(x_j - x_l) - R\nabla h_j(x_j) - \sum_{l \in \Gamma_j} R{i_{\mL}}_{jl} \\
&=& \sum_{l=1}^N W_{jl} x_l - R\nabla h_j(x_j) - \sum_{l \in \Gamma_j} R{i_{\mL}}_{jl}.
\EEAS

Using the V-I relation for inductor we also have
\[
\frac{d}{dt}{i_{\mL}}_{jl} =
\frac{1}{L_{jl}}(x_j - x_l).
\]
Set $L_{jl}=R_{jl}$ for every edge $(j,l)$ in graph $G$.
Define $w_j = \sum_{l \in \Gamma_j} R{i_{\mL}}_{jl}$, then
\BEAS
\frac{d}{dt} w_j %&=&\frac{d}{dt}\sum_{l \in \Gamma_j} R{i_{\mL}}_{jl} \\
&=& \sum_{l \in \Gamma_j} \frac{R}{L_{jl}}(x_j - x_l) \\
&=& x_j - \sum_{l=1}^N W_{jl}x_l.
\EEAS

Combining the above, we get the V-I relations for the circuit
\BEA\label{e-pgextra-vi-continuous}
x_j &=& \prox_{Rf_j}\left (\sum_{l=1}^N W_{jl} x_l - R\nabla h_j(x_j) - w_j\right ) \\
\frac{d}{dt} w_j &=& x_j - \sum_{l=1}^N W_{jl}x_l,\nonumber
\EEA
for every $j=1, \ldots, N$ and every edge $(j,l)$ in graph $G$.
Discretizing the above V-I relations with stepsize $1/2$, 
and following the decentralized notation of \cite[\S11.3]{RyuYin2022_largescale},
we recover PG-EXTRA,
\BEA\label{e-pgextra-vi-discretized}
x^{k+1} &=& \prox_{Rf}\left (W x^k - R\nabla h(x^k) - w^k\right ) \\
w^{k+1} &=& w^k + \frac{1}{2}(I - W)x^k.\nonumber
\EEA

We can simplify the circuit by eliminating potentials $e_j$ as shown below.
\begin{figure}[ht]
    \begin{center} 
    \CircuitPgExtraSimple
    \vspace{-0.5cm}
\end{figure}

For the automatic discretization of the continuous-time dynamics
of this circuit, we define the energy at time $k$ as
\[
\mathcal{E}_k = \sum_{\text{edge}~\{j,l\}} 
\frac{L_{jl}}{2}\|{i^{k+1}_{\mL_{jl}}} - {i^\star_{\mL}}_{jl}\|^2_2
+ \sum_{j=1}^N \gamma \|x_j^k - x^\star\|_2^2 ,
\]
where $\gamma$ is a parameter that is being optimized, see \S\ref{appendix-automatic-discr},
and ${i^\star_{\mL}}_{jl}$ is the current through inductor
at equilibrium.

\section{Automatic discretization}\label{appendix-automatic-discr}

In this paper, we discretize admissible dynamic interconnects corresponding
to the following decentralizes setup with graph consensus
\BEQ\label{e-dist-opt-consensus}
\begin{array}{ll}
\underset{{x_1, \ldots, x_N\in \reals^{m/N}}}{\mbox{minimize}}&
    f_1(x_1) + \cdots + f_N(x_N)\\
    \mbox{subject to}& x_j = x_l,  \quad j =1, \ldots, N,
    \quad l \in \Gamma_j,
\end{array}
\EEQ
where $\Gamma_j$ contains the neighbors of agent $j$ in the communication
graph, see \S\ref{appendix-decentralized}.
We assume that the communication graph is connected, 
ensuring that all agents can communicate with each 
other~\cite[\S 11.2]{RyuYin2022_largescale}.
The static interconnect for this problem corresponds to the consensus problem
\[
\begin{array}{ll}
\underset{{x_1, \ldots, x_N\in \reals^{m/N}}}{\mbox{minimize}}&
    f_1(x_1) + \cdots + f_N(x_N)\\
    \mbox{subject to}& x_1 = \cdots = x_N,
\end{array}
\]
which is a special case of \eqref{e-dist-opt-primal} where $E^\intercal=(I, \ldots, I)\in \reals^{m \times m/N}$.
Therefore, we have
$n=m/N$ nets each of size $N$ with $x_j \in \reals^n$ for all $j=1, \ldots, N$.
This setup generalizes the setup
of the classical methods discussed in \S\ref{appendix-centralized} and \S\ref{appendix-decentralized}.

For automatic discretization, we focus on dynamic interconnects that have the same RLC circuit across each net, \ie, the dynamic interconnects represented with the multi-wire notation.

% \ie, the dynamic interconnect consists of $m$ identical copies of the same RLC circuit for the $m$ coordinates of $x\in \reals^m$.

\paragraph{Runge--Kutta method.}
The capacitor and inductor ODEs are of the form
\[
\frac{d}{dt}x(t) = F(x(t)).
\]
We discretize ODEs using the two-stage Runge--Kutta method,
with coefficients $\alpha$, $\beta$, 
and stepsize $h$:
\BEAS
x^{k+1/2} &=& x^k + \alpha h F(x^k) \\
x^{k+1} &=& x^k + \beta h F(x^k) + (1-\beta) h F(x^{k+1/2}).
\EEAS 

We clarify that simpler one-stage discretization schemes can also be used. We chose two-stage Runge--Kutta to demonstrate that multi-stage discretization schemes are compatible with our automatic discretization methodology.

\paragraph{Energy descent.}
Let a discrete-time optimization algorithm generate a sequence $\{(v^k, i^k, x^k, y^k)\}_{k=1}^\infty$ 
with $v^k, i^k \in \reals^\sigma$ (voltages across and currents through the branches of interconnect) 
and $x^k, y^k \in \reals^m$ (potentials at terminals and currents leaving terminals).
Let the subscripts $\mathcal{R}$, $\mathcal{L}$, and $\mathcal{C}$ denote the components 
related to resistors, inductors, and capacitors, respectively.
Then the energy stored in the circuit is given by
\BEAS
\mathcal{E}_k &=& \frac{1}{2}\|v_\mathcal{C}^{k}-v^\star_\mathcal{C}\|^2_{D_\mathcal{C}}+
\frac{1}{2}\|i_\mathcal{L}^k-i^\star_\mathcal{L}\|^2_{D_\mathcal{L}}.
\EEAS

\begin{lemma}\label{lem-descent-discretization}
Assume $f\colon \reals^m \to \reals \cup \{ \infty \}$ is a strictly convex function 
and the dynamic interconnect is admissible.
Let a discrete-time optimization algorithm generate a sequence
$\{(v^k, i^k, x^k, y^k)\}_{k=1}^\infty$.
If there exists 
$\eta > 0$ such that for all $k=1, 2, \ldots$ the energy descent 
\BEQ\label{e-energy-descent-cond}
D_k =    
 \left (  \mathcal{E}_{k+1} + \eta \langle x^k-x^\star, y^k-y^\star\rangle \right) - \mathcal{E}_k \leq 0
\EEQ 
holds, then $x^k$ converges to a primal solution.
\end{lemma}
\begin{proof}
Suppose there exists $\eta>0$ for which~\eqref{e-energy-descent-cond} holds.
Then we have
\BEAS 
 0 &\leq & \mathcal{E}_{K+1} \\
 &\leq& \mathcal{E}_K - \eta \langle x^K-x^\star, y^K-y^\star\rangle  \\
&\leq & \mathcal{E}_0 - \sum_{k=0}^K \eta \langle x^k-x^\star, y^k-y^\star\rangle.
\EEAS 
By monotonicity of subdifferential operator $\partial f$, we have
\[
\langle x^k-x^\star, y^k-y^\star\rangle \geq 0, \quad y^k \in \partial f(x^k).
\]
Thus rearranging the terms we get
\[
0 \leq \sum_{k=0}^K\eta\langle x^k-x^\star, y^k-y^\star\rangle
\leq \mathcal{E}_0.
\]
Sending $K$ to infinity, by the summability argument it follows that
\BEQ\label{e-descent-convergence}
\langle x^k-x^\star, y^k-y^\star\rangle \to 0.
\EEQ

Define the Lagrangian function 
\[
L(x, z, y) = f(x) - y^T(x - E^\intercal z).
\]
Since $x^\star \in \mathcal{R}(E^\intercal)$, there exists
some $z^\star$ such that $x^\star = E^\intercal z^\star$.
Then for fixed $z^\star$ and $y^\star$, function $L(x, z^\star, y^\star)$ is strictly convex.
Its subgradient is given by $(y - y^\star) \in  \partial_x L(x, z^\star, y^\star)$ for
$y \in \partial f(x)$, therefore, $0 \in  L(x^\star, z^\star, y^\star)$. 
Together with strict convexity this implies that $L(x, z^\star, y^\star)$ 
achieves a unique global minimum at $x^\star$ with $L(x^\star, z^\star, y^\star) = f(x^\star)$.
By the subgradient inequality, we have
\[
\langle x^k-x^\star, y^k-y^\star\rangle \geq L(x^k, z^\star, y^\star) - f(x^\star) \geq 0.
\]
Then the condition~\eqref{e-descent-convergence} implies $L(x^k, z^\star, y^\star) \to f(x^\star)$.
Therefore, $x^k \to x^\star$ which concludes the proof.
\end{proof}

By the descent lemma~\ref{lem-descent-discretization}, the discretization is dissipative 
if there exist value $\eta > 0$ such that 
\[
D_k =    
 \left (  \mathcal{E}_{k+1} + \eta \langle x^k-x^\star, y^k-y^\star\rangle \right) - \mathcal{E}_k \leq 0
\]
for all $k=1,2, \ldots$
Since the descent $D_k$ is defined using a one-step transition,
without loss of generality, it suffices to consider $k=1$.

\paragraph{Solver dissipative term.}
To provide more flexibility with the Ipopt~\cite{wachter2006implementation,andersson2019casadi} solver, we also incorporate dissipation from the linear resistors as in the continuous-time energy dissipation \eqref{eq-cont-power}, \ie, we try to establish
\BEQ\label{e-solver-dissipative-term}
D_k =    
 \left (  \mathcal{E}_{k+1}  +
\eta \langle x^k-x^\star, y^k-y^\star\rangle 
+ \rho R\|i_\mathcal{R}^k\|_2^2 \right)
- \mathcal{E}_k
\EEQ
with $\eta>0$ and $\rho\ge 0$.
Also see \S\ref{s-energy-dissipation}.
If there exist values $\eta > 0$
and $\rho\geq 0$ such that $D_k \leq 0$ holds, then the discretization is sufficiently dissipative and Lemma~\ref{lem-descent-discretization} applies.

\subsection{Dissipative discretization}\label{sec-app-dissipative-discretization}
In this section, we fix $\alpha$, $\beta$, $h$, $\eta$, and $\rho$ and describe a
convex optimization 
problem that checks whether the discretization is dissipative.
We focus on problem \eqref{e-dist-opt-consensus}.

\paragraph{Worst-case optimization problem.}
To verify if the dissipativity condition $D_k \leq 0$~\eqref{e-solver-dissipative-term} 
is satisfied for a given discretization, we can alternatively solve a worst-case problem.
Specifically, this entails determining if the optimal value of the following optimization problem is non-positive:
\BEQ\label{e-app-worst-case-descent}
\begin{array}{ll}
\mbox{maximize} &
    \mathcal{E}_2 - \mathcal{E}_1 
    + \eta \langle x^1-x^\star, y^1-y^\star\rangle + \rho R\|i_\mathcal{R}^1\|_2^2\\
    \mbox{subject to}
& \mathcal{E}_s = \frac{1}{2}\|v_\mathcal{C}^{s}-v^\star_\mathcal{C}\|^2_{D_\mathcal{C}} +
\frac{1}{2}\|i_\mathcal{L}^s-i^\star_\mathcal{L}\|^2_{D_\mathcal{L}}, \quad s \in \{1, 2\} \\
& (v^1, i^1, x^1, y^1) ~\text{is feasible initial point} \\
& (v^2, i^2, x^2, y^2) ~\text{is generated by discrete optimization method from initial point} \\
& f \in \mathcal{F},
\end{array}
\EEQ
where $f, v^k, i^k, x^k, y^k, v^\star, i^\star, x^\star, y^\star$
are the decision variables
and $\mathcal{F}$ is a family of functions 
(\eg, \ $L$-smooth convex) that the algorithm is to be applied to.

\paragraph{Reformulated worst-case optimization problem.}
Recall that we assume that the RLC circuit across each net is the same. 
Thus we can define an operator $\operatorname{mat}(z)$ that reshapes 
vector $z\in\reals^{\sigma}$ 
into a matrix of size $n \times \sigma/n$, 
where each row contains information (voltage or
current) of the electric components that belong to the same net.
Define index sets
\BEAS
I_K &=& \{ 1, 1.5, 2, \star\}, \\
I_N &=& \{ 1, \ldots, N\}, \\
I_K \times I_N &=& \{ (k, l) \mid  l \in I_N, k \in I_K \},
\EEAS
and matrices
\BEAS
H &=& \left [\operatorname{mat}(v^1) \quad  \operatorname{mat}(i^1) \quad \left [y^k_l \right]_{(k,l) \in I_K \times I_N}\right ] 
\in \reals^{n \times (2\sigma /n + |I_K| N)},\\
G &=& H^TH \in \symm_+^{2\sigma/n + |I_K| N}, \\
F &=& \left [f^k_l \right]_{(k,l) \in I_K \times I_N} \in \reals^{|I_K| N},
\EEAS
where $y^k_l \in \partial f_l(x^k)$ and $f_l^k = f_l(x^k)$ for all $l \in I_N$.
Note that we have
\[
|F| = |I_K| N, \qquad |G| = (2\sigma/n + |I_K| N)^2,
\]
and for $|I_K|=4$ this simplifies to
\[
|F| = 4 N, \qquad |G| = (2\sigma/n + 4 N)^2.
\]
Recall the circuit ODEs are discretized with the two-stage Runge--Kutta method,
leading to the variables
$\operatorname{mat}(v^k)$, $\operatorname{mat}(i^k)$, $x^k$, $y^k$ that are linear combinations of 
columns in $H$. The coefficients of these linear combinations are polynomials in $\alpha$, $\beta$, and $h$.
In other words, there exist matrices $\mathbf{v}^k, \mathbf{i}^k, \mathbf{x}^k, \mathbf{y}_l^k$ such that
\[
\operatorname{mat}(v^k) = H \mathbf{v}^k, \quad
\operatorname{mat}(i^k) = H \mathbf{i}^k, \quad
x^k = H \mathbf{x}^k, \quad
y^k_l = H \mathbf{y}^k_l
\]
for all $k \in I_K$, $l \in I_N$. Similarly, we can find $ \mathbf{f}^k_l$ such that $f^k_l = F \mathbf{f}^k_l$.

For fixed parameters $\alpha$, $\beta$, $h$, $\eta$, and $\rho$, the problem~\eqref{e-app-worst-case-descent} can be reformulated
as 
\[
\begin{array}{ll}
\underset{f_1, \ldots, f_N, H}{\mbox{maximize}} &
    \mathcal{E}_2 - \mathcal{E}_1  +
\eta \langle x^1-x^\star, y^1-y^\star\rangle +  \rho R\|i_\mathcal{R}^1\|^2_2\\
    \mbox{subject to}
& \mathcal{E}_s = \frac{1}{2}\|v_\mathcal{C}^{s}-v^\star_\mathcal{C}\|^2_{D_\mathcal{C}} +
\frac{1}{2}\|i_\mathcal{L}^s-i^\star_\mathcal{L}\|^2_{D_\mathcal{L}}, \quad s \in \{1, 2\} \\
& \operatorname{mat}(v^k) = H \mathbf{v}^k, \quad k \in I_K \\
& \operatorname{mat}(i^k) = H \mathbf{i}^k, \quad k \in I_K \\
& x^k = H \mathbf{x}^k, \quad k \in I_K \\
& y^k_l = H \mathbf{y}^k_l, \quad k \in I_K, \quad l \in I_N \\
& f_l \in \mathcal{F}_{\mu_l, M_l}(\reals^n), \quad l \in I_N .
\end{array}
\]

By the interpolation lemma (\cite{taylor2023optimal}, Theorem 2), $f_l \in \mathcal{F}_{\mu_l, M_l}(\reals^n)$ if and only if
\BEAS
0 &\geq& f_l^j - f_l^i  + \langle g_l^j, x^i - x^j\rangle 
+ \frac{1}{2M_l} \|g_l^i - g_l^j\|^2_2 \\
&& + 
\frac{\mu_l}{2 (1 - \mu_l / M_l)} \| x^i - x^j - 1/{M_l} (g_l^i - g_l^j)\|_2^2
, \quad i, j \in I_K.
\EEAS
Therefore, we can replace infinite dimensional decision variable $f_l \in \mathcal{F}_{\mu_l, M_l}(\reals^n)$ with $|I_K|(|I_K|-1)$ inequalities.

\paragraph{Grammian formulation.}
Now using Grammian formulation, 
the problem of finding the worst-case energy difference over a given 
family of functions reduces to solving an SDP, similar to~\cite{TaylorHendrickxGlineur2017_smooth}.
This SDP can be presented compactly as
\BEQ\label{e-sdp-primal-compact}
\begin{array}{ll}
\underset{G, F}{\mbox{maximize}}&
     [F^T \; \operatorname{vec}(G)^T]D p  \\
    \mbox{subject to}
& [F^T \; \operatorname{vec}(G)^T]S_{lij} p \leq 0, \quad l=1, \ldots, N, \; i, j \in I_K \\
& G \succeq 0,
\end{array}
\EEQ
where $p$ is a vector with dummy variables that encode the monomials of $\alpha, \beta, h, \eta, \rho$, and
$D\in \reals^{(|F|+|G|)\times |p|}$ and $S \in  \reals^{(|I_K|N)\times(|F|+|G|)\times |p|}$ 
are some matrices with constant coefficients.

\paragraph{Dualization.}
Define variables for the energy descent as
\[
V_D = [F^T \; \operatorname{vec}(G)^T]D p,
\]
and for interpolating inequality indexed by $lij$ as
\[
(V_S)_{lij} = [F^T \; \operatorname{vec}(G)^T]S_{lij} p.
\]
% \BEQ\label{e-sdp-primal-compact}
% \begin{array}{ll}
% \underset{G, F}{\mbox{maximize}}&
%      D  \\
%     \mbox{subject to}
% & S_{lij} \leq 0, \quad l=1, \ldots, N, \quad i, j \in I_K \\
% & G \succeq 0,
% \end{array}
% \EEQ
Let $Z$ and $\lambda_{lij}$ for all $i, j \in I_K$, $l=1, \ldots, N$ be the dual variables for problem~\eqref{e-sdp-primal-compact}. 
Vertically stack $\lambda_{lij}$ and $(V_S)_{lij}$ to form vectors $\lambda$ and $V_S$ respectively.
The Lagrangian that generates primal problem \eqref{e-sdp-primal-compact} is
\[
L(G, f, Z, \lambda) = V_D - \lambda^T V_S + \Tr(GZ),
\]
and the dual problem is given by
\BEQ\label{e-sdp-dual-compact}
\begin{array}{ll}
\underset{Z, \lambda}{\mbox{minimize}}&
     0 \\
    \mbox{subject to}
& D_F p - \lambda^T S_F p = 0 \\
& D_G p - \lambda^T S_G p + Z = 0 \\
% & \nabla_F D - \lambda^T \nabla_F S = 0 \\
% &\nabla_G D - \lambda^T \nabla_G S + Z = 0 \\
& Z \succeq 0 \\
& \lambda \geq 0.
\end{array}
\EEQ

\paragraph{Algebraic proof.}
Using a weak duality we have $p^\star \leq d^\star$. 
Let $Z^\star$ and $\lambda^\star$ be optimal dual variables with $d^\star = 0$, 
then for all $G \in \symm_+^{(|\mathcal{C}| +|\mathcal{L}|)/n + |I_K| N}$ 
and $F \in \reals^{|I_K| N}$ it follows that
\BEAS
L(G, F, Z^\star, \lambda^\star) &=& 
% F^T \underbrace{\left (\nabla_F D - (\lambda^\star)^T\nabla_F S  \right )}_{=0} \\
F^T \underbrace{\left ( D_F p - (\lambda^\star)^T S_F p  \right )}_{=0} + \Tr \left( G
\underbrace{\left (D_G p - (\lambda^\star)^T S_G p + Z^\star\right)}_{=0}\right)\\
&=&0.
\EEAS
Therefore, having $Z^\star$ and $\lambda^\star$ gives us an algebraic proof for the worst-case one step energy difference 
\[
V_D = \sum_{l,i,j} \underbrace{\lambda^\star_{lij}}_{\geq 0} \underbrace{(V_S)_{lij}}_{\leq 0} - \underbrace{\Tr(GZ^\star)}_{\geq 0}  \leq 0,
\]
where $G\succeq 0$ because $G$ is a Gram matrix and $(V_S)_{lij} \leq 0$ for all $f_l \in \mathcal{F}_{\mu_l, L_l}$.

\subsection{Optimizing over discretizations}\label{sec-app-opt-discretization}
In this section we also optimize over the parameters 
$\alpha$, $\beta$, $h$, $\eta$, and $\rho$.

\subsubsection{QCQP formulation}
% Functions
% $D$ and $S_{lij}$ for all $i, j \in I_K$, $l=1, \ldots, N$ are polynomial in $\alpha, \beta, h, b, d$.
% In particular, for vectorized variables $G$ and $F$ we have 
% \[
% D = [F^T \; G^T]D p, \qquad S_{lij} = [F^T \; G^T]S_{lij} p,
% \]
% where $p$ is a vector with dummy variables that encode the monomials,
% $D\in \reals^{(|F|+|G|)\times |p|}$ and $S \in  \reals^{|\lambda|\times(|F|+|G|)\times |p|}$ 
% are matrices with constant coefficients.
We can formulate the dual problem~\eqref{e-sdp-dual-compact} as QCQP 
following \cite{DasGuptaVanParysRyu2023_branchandbound},
\[
\begin{array}{ll}
\underset{p, \lambda, P}{\mbox{minimize}}&
     0 \\
    \mbox{subject to}
& D_F p - \lambda^T S_F p = 0 \\
& D_G p - \lambda^T S_G p + PP^T = 0 \\
& p^T Q_e p + a_e^T p = 0, \quad e = 1, \ldots, |p|\\
& P \textrm{ is lower triangular} \\
& \diag(P) \geq 0 \\
& \lambda \geq 0,
\end{array}
\]
where relations for dummy variables are specified using quadratic or bilinear constraints 
$p^T Q_e p + a_e^T p = 0$.

\subsubsection{Lifted nonconvex SDP}
Alternatively, we can formulate the dual problem~\eqref{e-sdp-dual-compact} as lifted
nonconvex semidefinite problem with respect to a variable
$w = \left ( p, \lambda \right ) \in \reals^{|p|+|\lambda|} $.
Specifically, we have
\BEQ\label{e-lifted-noncvx-sdp}
\begin{array}{ll}
\underset{w, W, Z}{\mbox{minimize}}&
     0 \\
    \mbox{subject to}
& \overline D_F(i,:) w - 
\Tr \left (\overline S_F(:,i,:) W \right )  = 0, \quad i = 1, \ldots, |F| \\
& \overline D_G(i,:) w - 
\Tr \left (\overline S_G(:,i,:) W \right ) + Z(i,i)  = 0, \quad i = 1, \ldots, |G| \\
& \Tr \left ( \overline Q_e W \right ) + \overline a_e^T w = 0 \\
& W = ww^T\\
& Z \succeq 0\\
& \lambda \geq 0,
\end{array}
\EEQ
where
$\overline S_F = \left [ \begin{array}{cc} 0 & \frac{1}{2} S_F^T \\ 
\frac{1}{2} S_F & 0 \end{array}  \right ]$, $\overline D_F = \left [ \begin{array}{cc} D_F & 0 \end{array}  \right ]$,
$\overline S_G = \left [ \begin{array}{cc} 0 & \frac{1}{2} S_G^T \\ 
\frac{1}{2} S_G & 0 \end{array}  \right ]$,
$\overline D_G = \left [ \begin{array}{cc} D_G & 0 \end{array}  \right ]$,
$\overline Q_e = \left [ \begin{array}{cc} Q_e &0  \\ 
0 & 0 \end{array}  \right ]$ and $\overline a_e =\left ( a_e, 0 \right )$.
In the above the transpose for the third order tensors $S_F$ and $S_G$ is obtained 
by transposing the first and third dimensions.
Note that with the exception of the rank-1 constraint $W = ww^T$, the constraints 
define convex sets.

\subsubsection{SDP relaxation}
To find globally optimal solutions to the nonconvex optimization problem,
methods like spacial branch-and-bound require good initial bounds on the variables.
Following \cite{DasGuptaVanParysRyu2023_branchandbound},
an SDP relaxation of~\eqref{e-lifted-noncvx-sdp} is given by
\BEQ\label{e-lifted-noncvx-sdp-relaxed}
\begin{array}{ll}
\underset{w, W, Z}{\mbox{minimize}}&
     0 \\
    \mbox{subject to}
& \overline D_F(i,:) w - 
\Tr \left (\overline S_F(:,i,:) W \right )  = 0, \quad i = 1, \ldots, |F| \\
& \overline D_G(i,:) w - 
\Tr \left (\overline S_G(:,i,:) W \right ) + Z(i,i)  = 0, \quad i = 1, \ldots, |G| \\
& \Tr \left ( \overline Q_e W \right ) + \overline a_e^T w = 0 \\
& \left [ \begin{array}{cc} W & w \\ 
w^T & 1 \end{array}  \right ] \succeq 0\\
& Z \succeq 0\\
& \lambda \geq 0.
\end{array}
\EEQ
Problem~\eqref{e-lifted-noncvx-sdp-relaxed} is now a convex optimization problem,
since the rank-1 constraint $W=ww^T$ has been relaxed to $W \succeq ww^T$.
This constraint in turn can be represented equivalently using the Schur complement.

\section{Package \texttt{ciropt}}\label{appendix-ciropt}

In this section, we present a simple problem instance to demonstrate 
the step-by-step process of obtaining a discretized algorithm with our methodology.

\paragraph{Optimization problem.}
Consider a problem
\[
\begin{array}{ll}
\mbox{minimize}& f(x),
\end{array}
\]
where $f$ is a convex function.

\paragraph{Determine the static interconnect.}
Static interconnect is determined from the optimality conditions.
The optimality condition for this problem is to find an $x$ such that
$0 \in \partial f(x)$. The corresponding static interconnect for this condition provided below.
\begin{center} 
\begin{circuitikz}[scale=.6, every node/.style={scale=0.7}][american voltages]
\ctikzset{label/align = straight}
    \pgfmathsetmacro{\w}{4}
    \pgfmathsetmacro{\h}{3}
    \pgfmathsetmacro{\boxheight}{2}
    \pgfmathsetmacro{\boxwidth}{2}
    \pgfmathsetmacro{\compwidth}{3}
    \pgfmathsetmacro{\midspace}{\w - \boxwidth}
    \pgfmathsetmacro{\wirespace}{\h/4}
    
    \circuitbox{\w}{0}{\boxwidth/2}{\boxheight/2}{\Large $\partial f$}
    \draw
        ( 0 , 0) to [short, -*] ( \boxwidth/2 + \midspace , 0 ) node[above, xshift=-6pt] {$x$};

    \wiremul{ \boxwidth * 2/3}{0}{0.2}{$m$}
\end{circuitikz}
\end{center}

\paragraph{Admissible dynamic interconnect.}
An admissible dynamic interconnect with RLC components relaxes to the static interconnect in equilibrium. The following provides an example of such a dynamic interconnect.
\begin{center} 
\CircuitExample
\end{center}

The V-I relations for the circuit (left column) and convergent discretized method
found by our method (right) are displayed below.
\begin{center}{
% \small
\begin{tabular}{l|l}
\parbox{0.3\linewidth}{%
\begin{align*}
x &= \prox_{(R/2) f}\left (z\right )\\
y &=\frac{2}{R}(z-x)\\
\frac{d}{dt} e_2 &= - \frac{1}{2CR}(R y + 3e_2) \\
\frac{d}{dt} z &= - \frac{1}{4CR}(5R y + 3e_2)
\end{align*}}
&
\parbox{0.3\linewidth}{%
\begin{align*}
x^k &= \prox_{(R/2) f}\left (z^k\right )\\
y^k &=\frac{2}{R}(z^k-x^k)\\
e_2^{k+1} &= e_2^k - \frac{h}{2CR}(R y^k + 3e_2^k) \\
z^{k+1} &= z^k - \frac{h}{4CR}(5R y^k + 3e_2^k).
\end{align*}}
\end{tabular} }
\end{center}

\paragraph{Automatic discretization.}
Now we find a discretization parameters for this dynamic interconnect that guarantee algorithm convergence
using \texttt{ciropt} package.

\paragraph{Step 1.} Define a problem.
\begin{lstlisting}[language=mypython]
import ciropt as co
problem = co.CircuitOpt()
\end{lstlisting}

\paragraph{Step 2.} 
Define function class, in this example $f$ is convex and nondifferentiable, \ie,
$\mu=0$ and $M=\infty$.
\begin{lstlisting}[language=mypython]
f = co.def_function(problem, mu=0, M=np.inf)
\end{lstlisting}

\paragraph{Step 3.} Define the optimal points.
\begin{lstlisting}[language=mypython]
x_star, y_star, f_star = f.stationary_point(return_gradient_and_function_value=True)
\end{lstlisting}

\paragraph{Step 4.} Define values for the RLC components and
discretization parameters, here for simplicity 
we take $\alpha=0$ and $\beta=1$.
\begin{lstlisting}[language=mypython]
R, C = 1, 10
h, eta = problem.h, problem.eta
\end{lstlisting}

\paragraph{Step 5.} Define the one step transition in the discretized V-I relations.
\begin{lstlisting}[language=mypython]
z_1 = problem.set_initial_point()
e2_1 = problem.set_initial_point()
x_1 = co.proximal_step(z_1, f, R/2)[0]
y_1 = (2 / R) * (z_1 - x_1)
e1_1 = (e2_1 - R * y_1) / 2
v_C1_1 = e2_1 / 2 - z_1
v_C2_1 = e2_1

e2_2 = e2_1  -  h / (2 * R * C) * (R * y_1 + 3 * e2_1)  
z_2 = z_1  -  h / (4 * R * C) * (5 * R * y_1 + 3 * e2_1)
x_2 = co.proximal_step(z_2, f, R/2)[0]
y_2 = (2 / R) * (z_2 - x_2)
v_C1_2 = e2_2 / 2 - z_2
v_C2_2 = e2_2 
\end{lstlisting}

\paragraph{Step 6.} Define the dissipative term
\[
\mathcal{E}_2- \mathcal{E}_1 
 +  \eta\langle x^1-x^\star, y^1-y^\star\rangle.
\]
Solve the final problem.

\begin{lstlisting}[language=mypython]
E_1 = (C/2) * (v_C1_1 + x_star)**2 + (C/2) * (v_C2_1)**2
E_2 = (C/2) * (v_C1_2 + x_star)**2 + (C/2) * (v_C2_2)**2
Delta_1 = eta * (x_1 - x_star) * (y_1 - y_star) 

problem.set_performance_metric(E_2 - (E_1 - Delta_1))
params = problem.solve()[:1]
\end{lstlisting} 

This gives the disretization parameters
\[
b=6.66, \qquad h=6.66. 
\]
The resulting provably convergent algorithm is 
\BEAS
x^k &=& \prox_{(1/2) f}(z^k )\\
y^k &=& 2(z^k-x^k)\\
w^{k+1} &=& w^k - 0.33(y^k + 3w^k) \\
z^{k+1} &=& z^k - 0.16(5 y^k + 3w^k).
\EEAS

\paragraph{New algorithm.} 
Solve your problem using new algorithm. 
Consider Huber penalty function $\phi:\reals \to \reals$
\[
\phi(x) = 
\begin{cases}
x^2 & |x| \leq 1 \\
2x-1 & |x| > 1.
\end{cases} 
\]
We consider the primal problem
\[
\begin{array}{ll}
\text{minimize} & f(x) = \sum_i \phi(x_i - c_i) \\
\text{subject to} & Ax = b,
\end{array}
\]
where $A\in \reals^{m \times n}$, $b\in \reals^m$ and $c\in \reals^n$, and solve the dual problem
\[
\begin{array}{ll}
\text{maximize} & g(y) = -f^*(-A^\intercal  y ) - b^\intercal y.
\end{array}
\]

We apply our algorithm to solve the dual problem.
Note that the proximal operator $\mathbf{prox}_{\alpha g}(\tilde y)$ is equivalent to
\[
x = \argmin_x \left( f(x) + (\alpha/2)\|Ax-b\|_2^2 + \tilde y^\intercal (Ax-b) \right),
\quad y = \tilde y + \alpha(Ax-b).
\]

Since $f$ is CCP and $2$-smooth (as a Huber loss), $f^*$ is $1/2$-strongly convex. 
We take $m=30$, $n=100$ and sample entries of $A$, 
$c$ and $b$ from i.i.d. Gaussian distribution.
Finally we rescale the entries of $A$ by $\lambda_{\min}(AA^\intercal )$ to have $g$ that is $1/2$-strongly convex.
The following Figure~\ref{fig-app-huber} presents the results of the algorithm applied to a random problem instance. 

\begin{figure}[H]
    \centering
    \includegraphics[width=0.7\textwidth]{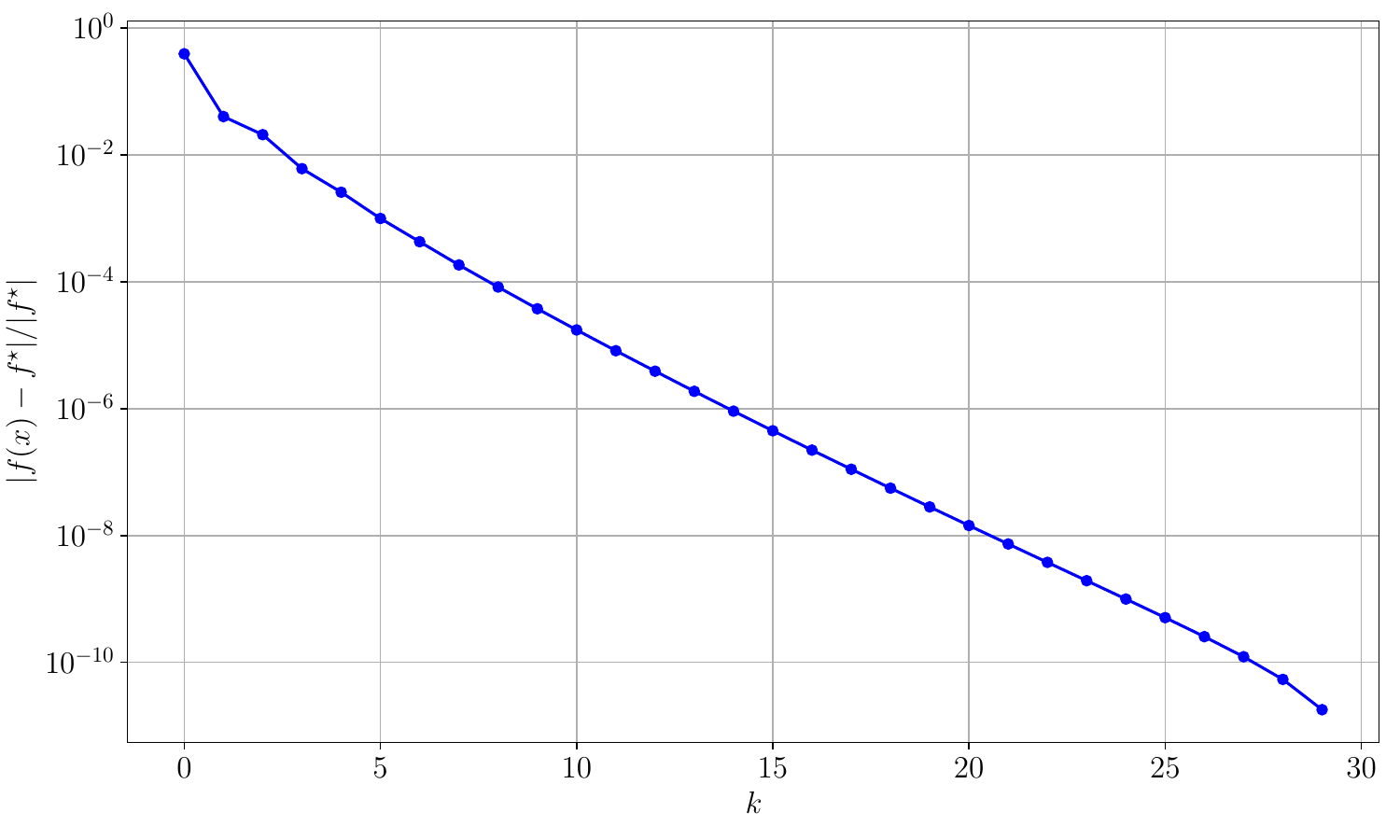}
    \caption{Relative error across iterations when applying the new algorithm.}
    \label{fig-app-huber}
\end{figure}

\section{Numerical experiments}\label{app-numerical}

\subsection{Decentralized ADMM+C}\label{sec-dadmm-c}

Consider a decentralized optimization problem
\BEQ 
\begin{array}{ll} \nonumber
 \underset{x \in \reals^m}{\mbox{minimize}}&
    \sum_{i=1}^{N} f_i(x),
\end{array}
\EEQ
where $f_1, \dots, f_N$ are CCP. Suppose furthermore we know some of the functions are strongly convex, that is, suppose there is a subset $S \subset \set{ 1, 2, \dots, N }$ such that $f_j$ are strongly convex for $j \in S$. 
We wish to find an efficient algorithm that fully exploits the additional information for $f_j$'s. 

To solve the problem in a decentralized manner, define a primal variable $x_i \in \reals^m$ for each agent function $f_i$.
To leverage the strong convexity of $f_j$ for each $j \in S$, 
we could consider implementing a specialized update rule for $x_j$ that is more effective for strongly convex 
functions.
% However, since the updates of $x_j$'s for $j\in S$ affect to other $x_i$'s that $i \notin S$ but communicate with $x_j$'s, it may be hard to impose such update rules while keeping the convergence of all agents. 

We consider a modification of the DADMM circuit in \S\ref{s-dadmm}. 
Recall from \S\ref{s-nesterov}, a circuit with a capacitor and inductor corresponds to a method with momentum. 
It is known~\cite{Polyak1964_methods} that momentum accelerates the convergence of methods for strongly convex functions. 
Therefore, we propose to attach capacitors to the circuit in \S\ref{s-dadmm}, on the nets that are directly related to $x_j$'s in $j\in S$. 
We anticipate that the method derived by discretization of a new circuit (using our automatic discretization methodology) 
will outperform the DADMM.

Consider a modified decentralized geometric median problem from~\cite{ShiLingWuYin2015_proximal}. 
Suppose each agent $i \in \set{ 1, \dots, N}$ holds vector $b_i \in \reals^m$, 
and consider the minimization problem
\BEQ\label{e-dadmmc-ex-problem}
\begin{array}{ll}
 \underset{x \in \reals^m}{\mbox{minimize}}&
    \sum_{i \in S} \left (\| x - b_i\|_2 + \|x - b_i\|_2^2\right) + \sum_{i \notin S} \| x - b_i\|_2.
\end{array}
\EEQ
The minimization subproblem has an explicit solution, \ie, 
\begin{align*}
    \prox_{\rho f_i}(z)
    &= b_i - \frac{b_i-\tilde{z}}{ \|b_i-\tilde{z}\|_2} (\|b_i-\tilde{z}\|_2 - \tilde \rho )_+,
\end{align*}
where
\begin{align*}
    \tilde{z} 
    &= \begin{cases}
        z & i \notin S \\
        \frac{1}{1 + 2 \rho } (z + 2 \rho b_i) & i \in S ,
    \end{cases}
    \qquad\qquad
    \tilde{\rho} 
    = \begin{cases}
        \rho & i \notin S \\
        \frac{\rho}{1 + 2 \rho} & i \in S .
    \end{cases}
\end{align*}
% \begin{align*}
%     \prox_{\alpha f_i}(z)
%     = \begin{cases}
%         b_i - \frac{b_i-z}{\norm{b_i-z}} \max\set{ \norm{b_i-z} - \alpha, 0 } & i \in S \\
%         b_i - \frac{b_i-\frac{1}{1 + 2 \alpha \epsilon } (z + 2 \alpha \epsilon b_i)}{\norm{b_i-\frac{1}{1 + 2 \alpha \epsilon } (z + 2 \alpha \epsilon b_i)}} \max\set{ \norm{b_i - \frac{1}{1 + 2 \alpha \epsilon } (z + 2 \alpha \epsilon b_i) } -  \frac{\alpha}{1 + 2 \alpha \epsilon}, 0 } & i \notin S
%     \end{cases}
% \end{align*}
We set $m = 100$, $N=6$, $S=\set{4,5}$, and sample vectors $b_i \in \reals^{100}$ from the 
uniform distribution over $[-100,100]^{100}$. 
We use graph $G$ provided in Figure~\ref{fig:dadmm_c_n6}.
We initialize iterates to $x^0_i = b_i$ for all $i$.

We use a modified DADMM circuit~\S\ref{s-dadmm} for the graph in Figure~\ref{fig:dadmm_c_n6}. 
This modified version includes an extra capacitor connected at $e_{45}$,
to which we refer as DADMM+C.
\CircuitDADMMC
Note when $N=1$, the DADMM+C circuit corresponds to the 
Nesterov acceleration circuit~\S\ref{s-nesterov}. 

Using KCL at $e_{45}$, we have
\[
    {i_{\mL}}_{45} + {i_{\mL}}_{54} =  - \frac{(e_{45} - x_4)}{R} - \frac{(e_{45} - x_5)}{R} - C \frac{d}{dt} e_{45}. 
\]
Thus for $\{j,l\}\neq \{4,5\}$  the update rule of $e_{jl}$ 
is given by \eqref{e-dec-admm-kcl2},
while for $e_{45}$ we get
\[
    \frac{d}{dt} e_{45}
        = - \frac{1}{C} \pr{ {i_{\mL}}_{45} + {i_{\mL}}_{54} + \frac{1}{R} \pr{ 2 e_{45} - x_4 - x_5 } }.
\]
Other V-I relations remain unchanged as in~\S\ref{s-dadmm}.
The resulting algorithm becomes
\BEAS 
x_j^{k+1} &=& \prox_{(R/|\Gamma_j|)f_j}\left ( \frac{1}{|\Gamma_j|} \sum_{l \in \Gamma_j} (R{i^k_{\mL}}_{jl} + e^k_{jl})\right ) \\
e_{jl}^{k+1} &=& \begin{cases}
    e_{45}^{k} - \frac{h}{CR} \pr{ R({i_{\mL}^k}_{45} + {i_{\mL}^k}_{54}) + 2 e_{45}^{k} - x_4^{k+1} - x_5^{k+1} } 
    \quad & \{j, l\} = \{4,5\}  \\
    \frac{1}{2}(x_j^{k+1}+x_l^{k+1}) & \text{otherwise}
\end{cases} \\ 
{i_{\mL}}_{jl}^{k+1} &=& {i_{\mL}}_{jl}^k + \frac{h}{L}(e_{jl}^{k+1} - x_j^{k+1}).
\EEAS 

We consider the circuit with $R=0.8$, $L=2$ and $C=15$.
To discretize the circuit, we take advantage of the fact that 
the strong convexity of $f_i$ is $2$ for 
$i \in S$~\eqref{e-dadmmc-ex-problem}.
Specifically, we apply our automatic discretization methodology 
to convex functions, setting 
$\mu=0$ for $f_i$ with $i\notin S$ and $\mu=2$ for $f_i$ with $i \in S$,
and using smoothness $M=100$.
The sufficiently dissipative parameters we find are
\[
\eta=3.70, \quad h = 3.52, \quad \rho=0, \quad \alpha=0, \quad \beta=1,
\quad \gamma=4.48.
\] 
We compare DADMM+C with DADMM and P-EXTRA. 
Based on grid search, 
we set $R=0.6$ for DADMM in \S \ref{s-dadmm}, 
and $R=1$ and $h_1=\cdots=h_N=0$ for PG-EXTRA in \S \ref{s-PG-EXTRA} to get P-EXTRA. 
Note that the parameters of the proximal operators for DADMM are scaled by ${1}/{|\Gamma_j|}$, 
in contrast to P-EXTRA, where $|\Gamma_j|$ is generally not equal to $1$. 
% Also the two $R$ values for each method do not need to match.
We use Metropolis mixing matrix for P-EXTRA,
\[
    W_{ij} = \begin{cases}
        \frac{1}{\max\set{ |\Gamma_i|, |\Gamma_j| } + 1} & \text{ if $i \in \Gamma_j$}  \\ 
        1 - \sum_{j\in\Gamma_j} W_{ij} & \text{ if } i=j \\
        0 & \text{otherwise}.
    \end{cases}
\] 
The numerical results are illustrated in Figure~\ref{fig:dadmm_c_n6_appndx}. 
The relative error for DADMM+C decreases to $10^{-10}$ in 
$66$ iterations, for DADMM in $87$ iterations
and for P-EXTRA in $294$ iterations.
\begin{figure} [H]
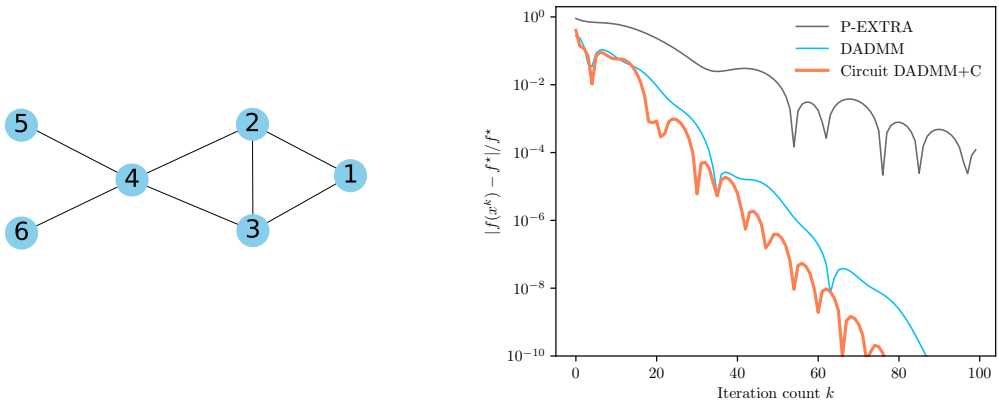

    \centering
    \begin{subfigure}[c]{0.35\textwidth}
        \centering
        \includegraphics[width=\textwidth]{figures_main/graph_n6.pdf}
        % \caption{Subfigure A}
    \end{subfigure}
    \hfill
    \begin{subfigure}[c]{0.54\textwidth}
        \centering
        \includegraphics[width=\textwidth]{figures_main/freldif_n6_circuit_dadmm_c_theta1.pdf}
        % \caption{Subfigure B}
    \end{subfigure}
    \caption{(Left) Underlying graph $G$. (Right) Relative error $\left|f(x^k) - f^\star\right|/{f^\star}$ vs. $k$.
    }
    \label{fig:dadmm_c_n6_appndx}
\end{figure}

\subsubsection{Convergence proof of decentralized ADMM+C}
\label{ss:dadmm-c-convergence}
% maybe adopting $E^\intercal=(I, \ldots, I) \in \reals^{ m \times m/N}$ would be better

We first review the meaning of the numerical values in the previous section. 
Suppose $(x^\star, y^\star)$ be a primal-dual solution pair.  Then by Theorem~\ref{thm-equilibrium}, there is $(v^\star, i^\star)\in D_{x^\star, y^\star}$ that satisfies $(v^\star, i^\star) = ((0, 0, v_{\mathcal{C}}^\star), (0, i_{\mathcal{L}}^\star, 0))$. 
The numerical values imply, for the energy function
\[
\mathcal{E}_k = \sum_{ ( j,l) \in A } \norm{ i_{\mL_{jl}}^k - i_{\mL_{jl}}^\star }^2
    + \sum_{ j<l, \{ j,l \} \subset S } \frac{15}{2} \norm{ e_{jl}^{k} - x_j^\star }^2 
    + 3.52 \sum_{ j<l, \{ j,l \} \not\subset S } \norm{ e_{jl}^k - x_j^\star }^2,
\]
following inequality is true up to certain numerical precision
\[
    \left (  \mathcal{E}_{k+1} +
        3.7\langle x^{k+1}-x^\star, y^{k+1}-y^\star\rangle \right) - \mathcal{E}_k \leq 0.
\] 
This inequality guarantees the convergence of the method we've used in the experiment. 
Inspired form the numerical results, we could obtain an analytic proof for generalized cases as well.  
To clarify, $A$ is the set of edges introduced in \S\ref{appendix-decentralized}, and each edge is counted twice in the sum $\sum_{ ( j,l) \in A }$. 
% Furthermore, we express the terms for each component separately since we're not dealing with a symmetric circuit, in the sense that capacitors are attached only to the $e_{jl}$'s with $\{j,l\} \subset S$. 
% We state and provide the result as a lemma. 
\begin{lemma}
\label{lem:dadmm-c-convergence}
Let $f_j\colon \reals^m \to \reals \cup \{ \infty \}$ are CCP functions for $j \in \{1,\dots, N\}$ and $S \subset \set{ 1, 2, \dots, N }$. 
Consider the generalized DADMM+C
\BEAS 
x_j^{k+1} &=& \prox_{(R/|\Gamma_j|)f_j}\left ( \frac{1}{|\Gamma_j|} \sum_{l \in \Gamma_j} (R{i^k_{\mL}}_{jl} + e^k_{jl})\right ) \\
e_{jl}^{k+1} &=& \begin{cases}
    e_{jl}^{k} - \frac{h}{CR} \pr{ R({i_{\mL}^k}_{jl} + {i_{\mL}^k}_{lj}) + 2 e_{jl}^{k} - x_j^{k+1} - x_l^{k+1} } 
    \quad & \{j, l\} \subset S  \\
    \frac{1}{2}(x_j^{k+1}+x_l^{k+1}) & \text{otherwise}
\end{cases}  \\ 
{i_{\mL}}_{jl}^{k+1} &=& {i_{\mL}}_{jl}^k + \frac{h}{L}(e_{jl}^{k+1} - x_j^{k+1}),
\EEAS 
with initilzation $i_{\mL_{jl}}^0 = i_{\mL_{lj}}^0$ for all edge $(j,l)$ in G.
% Suppose the sequence
% $\{(e^k, i_{\mL}^k, x^k, y^k)\}_{k=1}^\infty$
% is generated by DADMM+C, with initilzation $i_{\mL_{jl}}^0 = i_{\mL_{lj}}^0$ for all edge $(j,l)$ in G. 
Let $(x^\star, y^\star)$ be a primal-dual solution pair and $(v^\star, i^\star)\in D_{x^\star, y^\star}$. 
Define the energy function as
\[
    \mathcal{E}_k 
    = \sum_{ ( j,l) \in A } \frac{L}{2}  \norm{ i_{\mL_{jl}}^k - i_{\mL_{jl}}^\star }^2
    + \sum_{ j<l, \{ j,l \} \subset S } \frac{C}{2} \norm{ e_{jl}^{k} - x_j^\star }^2 
    + \sum_{ j<l, \{ j,l \} \not\subset S } \frac{h}{2R} \norm{ e_{jl}^k - x_j^\star }^2 .
\]    
Then for all $R, L, C, h, \tau > 0$ that satisfy
\[
    \max\set{ 1, \frac{2h}{CR} }
    \le \tau^2    
    \le 2 - \frac{hR}{L},
\] 
following inequality is true
\[   
 \left (  \mathcal{E}_{k+1} 
            +  \frac{h}{2R} \pr{ 2 - \frac{hR}{L} - \tau^2 } \sum_{ ( j,l) \in A } \left\|   e_{jl}^{k+1} -  x_j^{k+1}   \right\|^2 
            + h \langle x^{k+1}-x^\star, y^{k+1}-y^\star\rangle \right) - \mathcal{E}_k \leq 0.
\]
\end{lemma}

\begin{proof}
    For notation simplicity, define
    \[
        y_{jl}^{k+1} = i_{\mL_{jl}}^{k} + \frac{1}{R} \pr{  e_{jl}^{k} - x_j^{k+1} }.
    \]
    Note, from the first line of the algorithm we have
    $ x_j^{k+1} + \frac{R}{|\Gamma_j|} y_j^{k+1}
    = \frac{1}{|\Gamma_j|} \sum_{l\in\Gamma_j} \pr{ R i_{\mL_{jl}}^k + e_{jl}^k }$,
and therefore
\[
    y_j^{k+1} 
    = \sum_{l\in\Gamma_j} \pr{ i_{\mL_{jl}}^k + \frac{1}{R} \pr{ e_{jl}^k - x_j^{k+1} } }
    = \sum_{l\in\Gamma_j} y_{jl}^{k+1}.
\]

\emph{(i) Difference of $\sum_{ ( j,l) \in A } \frac{L}{2}  \norm{ i_{\mL_{jl}}^k - i_{\mL{jl}}^\star }^2$.} 

Name
\[
    \Delta_L = \sum_{ ( j,l) \in A } \frac{L}{2}  \norm{ i_{\mL_{jl}}^{k+1} - i_{\mL_{jl}}^\star }^2
    - \sum_{ ( j,l) \in A } \frac{L}{2}  \norm{ i_{\mL_{jl}}^{k} - i_{\mL_{jl}}^\star }^2.
\]
Observe
\BEAS
    \frac{\Delta_L}{h}
    &=& \frac{1}{h} \pr{ \sum_{ ( j,l) \in A } \frac{L}{2}  \norm{ {i_{\mL}^k}_{jl} - i_{\mL_{jl}}^\star + \frac{h}{L}(e_{jl}^{k+1} - x_j^{k+1})  }^2
    - \sum_{ ( j,l) \in A } \frac{L}{2}  \norm{ i_{\mL_{jl}}^{k} - i_{\mL_{jl}}^\star }^2 } \\
    &=& \sum_{ ( j,l) \in A } \inner{ e_{jl}^{k+1} - x_j^{k+1} }{ i_{\mL_{jl}}^{k} - i_{\mL_{jl}}^\star } 
    + \sum_{ ( j,l) \in A } \frac{h}{2L} \norm{ e_{jl}^{k+1} - x_j^{k+1} }^2 \\
    &=& \sum_{ ( j,l) \in A } \inner{ e_{jl}^{k+1} - x_j^{k+1} }{ y_{jl}^{k+1} - \frac{1}{R} \pr{  e_{jl}^{k} - x_j^{k+1} } - i_{\mL_{jl}}^\star } 
    + \sum_{ ( j,l) \in A } \frac{h}{2L} \norm{ e_{jl}^{k+1} - x_j^{k+1} }^2 \\
    &=& - \frac{1}{R} \sum_{ ( j,l) \in A }  \inner{ e_{jl}^{k+1} - x_j^{k+1} }{ e_{jl}^{k} - x_j^{k+1} }
     +  \sum_{ ( j,l) \in A } \inner{ e_{jl}^{k+1} - x_j^{k+1} }{ y_{jl}^{k+1} - i_{\mL_{jl}}^\star } \\
    && + \sum_{ ( j,l) \in A } \frac{h}{2L} \norm{ e_{jl}^{k+1} - x_j^{k+1} }^2 \\
    &=& -\pr{ \frac{1}{R} - \frac{h}{2L} } \sum_{ ( j,l) \in A } \norm{ e_{jl}^{k+1} - x_j^{k+1} }^2 
       - \frac{1}{R} \sum_{ ( j,l) \in A } \inner{ e_{jl}^{k+1} - x_j^{k+1} }{ e_{jl}^{k} - e_{jl}^{k+1} } \\ 
    &&
       - \sum_{ ( j,l) \in A } \inner{ x_j^{k+1} - x_j^\star }{ y_{jl}^{k+1} - i_{\mL_{jl}}^\star }
       + \sum_{ ( j,l) \in A } \inner{ e_{jl}^{k+1} - x_j^\star }{ y_{jl}^{k+1} - i_{\mL_{jl}}^\star }.
\EEAS
On the other hand, from Theorem~\ref{thm-equilibrium} we know $i_{\mR}^\star=0$, by KCL at $x_j$ we have $y_j^\star = \sum_{l \in \Gamma_j} i_{\mL_{jl}}^\star$. Therefore
\[
    \sum_{ ( j,l) \in A } \inner{ x_j^{k+1} - x_j^\star }{ y_{jl}^{k+1} - i_{\mL_{jl}}^\star }
    = \sum_{ j=1 }^{N} \sum_{l \in \Gamma_j} \inner{ x_j^{k+1} - x_j^\star }{ y_{jl}^{k+1} - i_{\mL_{jl}}^\star }
    = \sum_{ j=1 }^{N} \inner{ x_j^{k+1} - x_j^\star }{ y_{j}^{k+1} - y_{j}^\star }.
\]
Moreover, $i_{\mL_{jl}}^\star = -i_{\mL_{lj}}^\star$, $e_{jl}^{k+1} =  e_{lj}^{k+1}$ holds by their definition, 
and $x_j^\star = x_l^\star$ as $x^\star$ is the solution. Therefore we see
\[
    \sum_{ ( j,l) \in A } \inner{ e_{jl}^{k+1} - x_j^\star }{ i_{\mL_{jl}}^\star }
    = \sum_{ j<l, (j,l) \in A } \inner{ e_{jl}^{k+1} - x_j^\star }{ i_{\mL_{jl}}^\star + i_{\mL_{lj}}^\star }
    = 0.
\]
Lastly, following equality is true for $\tau \in (0,\infty)$
\BEAS
    \inner{ e_{jl}^{k+1} - x_j^{k+1} }{ e_{jl}^{k} - e_{jl}^{k+1} } 
    &=& \frac{1}{2} \norm{ \tau(e_{jl}^{k+1} - x_j^{k+1}) + \frac{1}{\tau}(e_{jl}^{k} - e_{jl}^{k+1}) }^2 \\&&
        -  \frac{\tau^2}{2} \norm{ e_{jl}^{k+1} - x_j^{k+1} }^2 - \frac{1}{2\tau^2}  \norm{ e_{jl}^{k} - e_{jl}^{k+1} }^2.
\EEAS
Finally, applying above observations we have
\BEAS
    \frac{\Delta_L}{h}
    &=& -\pr{ \frac{1}{R} - \frac{h}{2L} -  \frac{\tau^2}{2R} } \sum_{ ( j,l) \in A } \norm{ e_{jl}^{k+1} - x_j^{k+1} }^2 \\&&
        + \frac{1}{2R\tau^2} \sum_{ ( j,l) \in A }  \norm{ e_{jl}^{k} - e_{jl}^{k+1} }^2    
       - \frac{1}{2R} \sum_{ ( j,l) \in A } \norm{ \tau(e_{jl}^{k+1} - x_j^{k+1}) + \frac{1}{\tau}(e_{jl}^{k} - e_{jl}^{k+1}) }^2 \\ &&
      - \sum_{ j=1 }^{N} \inner{ x_j^{k+1} - x_j^\star }{ y_{j}^{k+1} - y_{j}^\star }
       + \sum_{ ( j,l) \in A } \inner{ e_{jl}^{k+1} - x_j^\star }{ y_{jl}^{k+1}  }.
\EEAS

\emph{(ii) Difference of $\sum_{ j<l, \{ j,l \} \subset S } \frac{C}{2} \norm{ e_{jl}^{k} - x_j^\star }^2$.}

Name
\[
    \Delta_C = \sum_{ j<l, \{ j,l \} \subset S } \frac{C}{2} \norm{ e_{jl}^{k+1} - x_j^\star }^2 
        - \sum_{ j<l, \{ j,l \} \subset S } \frac{C}{2} \norm{ e_{jl}^{k} - x_j^\star }^2  .
\]
Plugging the definition of the method, we have
\BEAS
    \frac{\Delta_C}{h}
    &=& \frac{1}{h} \pr{ \frac{1}{2} \sum_{ \{ j,l \} \subset S } \frac{C}{2} \norm{ e_{jl}^{k+1} - x_j^\star }^2 
    - \frac{1}{2} \sum_{ \{ j,l \} \subset S } \frac{C}{2} \norm{ e_{jl}^{k+1} - x_j^\star -  \pr{ e_{jl}^{k+1} - e_{jl}^{k} }  }^2 } \\
    &=& \frac{C}{2h} \sum_{ \{ j,l \} \subset S } \inner{ e_{jl}^{k+1} - e_{jl}^{k} }{ e_{jl}^{k+1} - x_j^\star }  
         - \frac{C}{4h}  \sum_{ \{ j,l \} \subset S }  \norm{ e_{jl}^{k+1} - e_{jl}^{k} }^2   \\
    &=& - \frac{1}{2} \sum_{ \{ j,l \} \subset S } \inner{ y_{jl}^{k+1} + y_{lj}^{k+1} }{ e_{jl}^{k+1} - x_j^\star }  
         - \frac{C}{4h}  \sum_{ \{ j,l \} \subset S }  \norm{ e_{jl}^{k+1} - e_{jl}^{k} }^2   \\
    &=& - \sum_{ \{ j,l \} \subset S } \inner{ y_{jl}^{k+1} }{ e_{jl}^{k+1} - x_j^\star } 
        - \frac{C}{4h}  \sum_{ \{ j,l \} \subset S }  \norm{ e_{jl}^{k+1} - e_{jl}^{k} }^2.
\EEAS
% We've used $e_{lj}^{k+1} = e_{jl}^{k+1}$ in the last equation. 
% Therefore
% \[
%     \frac{\Delta_C}{h}
%     = - \sum_{ \{ j,l \} \subset S } \inner{ y_{jl}^{k+1} }{ e_{jl}^{k+1} - x_j^\star }
%     - \frac{C}{4h}  \sum_{ \{ j,l \} \subset S }  \norm{ e_{jl}^{k+1} - e_{jl}^{k} }^2.
% \]

\emph{(iii) Difference of $\sum_{ j<l, \{ j,l \} \not\subset S } \frac{h}{2R}  \norm{ e_{jl}^k - x_j^\star }^2$.}

Name
\[
    \Delta_\gamma = 
        \sum_{ j<l, \{ j,l \} \not\subset S } \frac{h}{2R}  \norm{ e_{jl}^{k+1} - x_j^\star }^2
        - \sum_{ j<l, \{ j,l \} \not\subset S } \frac{h}{2R}  \norm{ e_{jl}^k - x_j^\star }^2 .
\]
For $\{j,l\} \not\subset S$, from the initialization $i_{\mL_{jl}}^{0} = -i_{\mL_{lj}}^{0}$ and from $e_{jl}^{k+1} = \frac{1}{2} \pr{ x_j^{k+1} + x_l^{k+1} }$, inductively we can check
\[
    i_{\mL_{jl}}^{k+1} 
    = i_{\mL_{jl}}^{k} + \frac{h}{L} \pr{ e_{jl}^{k+1} - x_j^{k+1} }
    = -i_{\mL_{lj}}^{k} - \frac{h}{L} \pr{ e_{jl}^{k+1} - x_l^{k+1} }
    = -i_{\mL_{lj}}^{k+1} .
\]
Therefore $i_{\mL_{jl}}^{k} = -i_{\mL_{lj}}^{k}$ for all $k$. 
And from the definition of $y_{lj}^{k+1}$, we have
\BEAS
    y_{lj}^{k+1}
    % = i_{\mL_{lj}}^{k} + \frac{1}{R} \pr{  e_{lj}^{k} - x_l^{k+1} }
    % &=& -i_{\mL_{jl}}^{k} + \frac{1}{R} \pr{  e_{jl}^{k} - ( 2 e_{jl}^{k+1} - x_j^{k+1} ) }  \\ 
    = -i_{\mL_{jl}}^{k} - \frac{1}{R} \pr{  e_{jl}^{k} -  x_j^{k+1} } + \frac{2}{R} \pr{ e_{jl}^{k} - e_{jl}^{k+1} } 
    = -y_{jl}^{k+1} + \frac{2}{R} \pr{ e_{jl}^{k} - e_{jl}^{k+1} }.
\EEAS
And thus
\[
    e_{jl}^{k+1} - e_{jl}^{k}
    =   - \frac{R}{2} \pr{ y_{lj}^{k+1} +  y_{lj}^{k+1} } .
\]
Now proceeding the similar calculation and argument for $\Delta_C$, we have
\BEAS
    \frac{\Delta_\gamma}{h}
    % &=& \sum_{\{ j,l \} \not\subset S } \frac{1}{R} \inner{ e_{jl}^{k+1} - e_{jl}^{k} }{ e_{jl}^{k+1} - x_j^\star }  
    %      - \sum_{ \{ j,l \} \not\subset S } \frac{1}{2R}  \norm{ e_{jl}^{k+1} - e_{jl}^{k} }^2   \\
    %&=& - \sum_{ j<l, \{ j,l \} \not\subset S } \frac{h}{R} \inner{  y_{jl}^{k+1} + y_{lj}^{k+1} }{ e_{jl}^{k+1} - x_j^\star }  
    %    - \sum_{ j<l, \{ j,l \} \not\subset S } \frac{h}{R}  \norm{ e_{jl}^{k+1} - e_{jl}^{k} }^2   \\
    &=& - \sum_{ \{ j,l \} \not\subset S } \inner{  y_{jl}^{k+1} }{ e_{jl}^{k+1} - x_j^\star }  
         - \sum_{ j<l, \{ j,l \} \not\subset S } \frac{1}{2R} \norm{ e_{jl}^{k+1} - e_{jl}^{k} }^2 .
\EEAS

Finally, summing the calculations in (i), (ii), (iii), we have
\begin{align*}
    &\frac{1}{h} \pr{ \mathcal{E}_{k+1} - \mathcal{E}_k } 
        + \sum_{ j=1 }^{N} \inner{ x_j^{k+1} - x_j^\star }{ y_{j}^{k+1} - y_{j}^\star } 
        + \pr{ \frac{1}{R} - \frac{h}{2L} -  \frac{\tau^2}{2R} } \sum_{ ( j,l) \in A } \norm{ e_{jl}^{k+1} - x_j^{k+1} }^2 \\
    &= \frac{1}{h} \pr{ \Delta_L + \Delta_C + \Delta_\gamma} 
        + \sum_{ j=1 }^{N} \inner{ x_j^{k+1} - x_j^\star }{ y_{j}^{k+1} - y_{j}^\star } 
        + \frac{1}{2R} \pr{ 2 - \frac{hR}{L} - \tau^2 } \sum_{ ( j,l) \in A } \norm{ e_{jl}^{k+1} - x_j^{k+1} }^2 \\
    &=  - \frac{1}{R} \sum_{ ( j,l) \in A } \norm{ \tau(e_{jl}^{k+1} - x_j^{k+1}) + \frac{1}{\tau}(e_{jl}^{k} - e_{jl}^{k+1}) }^2 \\&\quad
       - \frac{1}{2R} \pr{ \frac{CR}{2h} - \frac{1}{\tau^2} }  \sum_{ \{ j,l \} \subset S }  \norm{ e_{jl}^{k+1} - e_{jl}^{k} }^2
       - \frac{1}{2R} \pr{ 1 - \frac{1}{\tau^2} } \sum_{ \{ j,l \} \not\subset S }  \norm{ e_{jl}^{k+1} - e_{jl}^{k} }^2.
\end{align*}
Therefore, for all $R, L, C, h, \tau > 0$ that satisfy 
\[
    2 - \frac{hR}{L} - \tau^2 \ge 0, \quad
    \frac{CR}{2h} - \frac{1}{\tau^2} \ge 0, \quad
    1 - \frac{1}{\tau^2} \ge 0
\]
or equivalently,
\[
    \max\set{ 1, \frac{2h}{CR} }
    \le \tau^2    
    \le 2 - \frac{hR}{L},
\]
we conclude the desired inequality. 

\end{proof}

\subsection{PG-EXTRA + Parallel C}
\label{s:pg-extra-c}

In this section, we introduce an additional pipeline of designing new optimization algorithm via circuit. 
The previous automatized discretization pipeline has the advantage of guaranteeing convergence; however, it may provide a conservative step size since it considers all worst-case scenarios. 
As a result, it may eliminate the possibility of finding efficient step size that works for certain optimization problem in practice. 

Our circuit-based approach has the advantage of designing a variant of the prior method quickly, that is likely to converge and possibly works better based on physical intuition. 
Furthermore, the variant method provides greater freedom in selecting parameters to tune. 
We provide an example of new optimization method obtained with exploiting these advantages, that outperforms PG-EXTRA for the problem considered in the paper introduced PG-EXTRA \cite{ShiLingWuYin2015_proximal}.

We use a modified PG-EXTRA circuit~\S\ref{s-PG-EXTRA}, % for the graph in Figure~\ref{fig:pg_extra_par_c_n20_appndx}. 
that includes extra capacitors connected parallel to inductors. 
\begin{center} 
\PgExtraParallelC
\end{center} 
Recalling \eqref{e-total-energy}, we know the energy for this circuit is defined as below
\[
    \mathcal{E}(t) 
    = \sum_{j<l, (j,l) \in A} \frac{C_{jl}}{2} \norm{ v_{\mC_{jl}}(t) - v_{\mC_{jl}}^\star }^2 
    + \sum_{j<l, (j,l) \in A} \frac{L_{jl}}{2} \norm{ i_{\mL_{jl}}(t) - i_{\mL_{jl}}^\star }^2.
\]
Note, compared to the energy of PG-EXTRA, we have additional energy terms for capacitors. 
Observe
\[
    v_{\mC_{jl}}(t) - v_{\mC_{jl}}^\star
    = x_j(t) - x_l(t) - ( x_j^\star - x_l^\star )
    = x_j(t) - x_l(t). 
\]
Previously, the energy dissipation by the resistors only reduced the values of $\| i_{\mL_{jl}}(t) - i_{\mL_{jl}}^\star \|^2$, but now it also reduces the values of $\|x_j(t) - x_l(t)\|^2$ for $(j,l) \in A$. 
Intuitively, we may hope that this dissipation accelerates the convergence $\lim_{t\to\infty} ( x_j(t) - x_l(t) ) = 0$, thus eventually speed up the convergence to the optimal. 
This observation provides informal motivation for the method. 

Following the arguments of \S\ref{s-PG-EXTRA}, 
% the dyanimcs corresponding to above circuit becomes
% \BEAS
%     e_j &=& \sum_{l=1}^{N} W_{jl}x_l - R \nabla h_j(x_j) - \sum_{l \in \Gamma_j} R i_{\mL_{jl}} - \sum_{l \in \Gamma_j} R i_{\mC_{jl}} \\
%     x_j &=& \prox_{R f_j} (e_j) \\
%     \frac{d}{dt} i_{\mL_{jl}} &=& \frac{1}{L_{jl}} v_{\mC_{jl}} \\
%     \frac{d}{dt} v_{\mC_{jl}} &=& \frac{1}{C_{jl}} i_{\mC_{jl}} .
% \EEAS
% Considering the discrete counterpart, we have
% \BEAS
%     e_j^k &=& \sum_{l=1}^{N} W_{jl} x_l^k - R \nabla h_j(x_j^k) 
%         - \sum_{l \in \Gamma_j} R i_{\mL_{jl}}^k - \sum_{l \in \Gamma_j} R i_{\mC_{jl}}^k \\
%     x_j^{k+1} &=& \prox_{R f_j} (e_j^{k}) \\
%     i_{\mL_{jl}}^{k+1} &=& i_{\mL_{jl}}^{k} + \frac{s}{L_{jl}} v_{\mC_{jl}}^k \\
%     i_{\mC_{jl}}^{k+1} &=& \frac{C}{s} \pr{ v_{\mC_{jl}}^{k+1} - v_{\mC_{jl}}^{k} }.
% \EEAS
and additionally defining $u_j = \sum_{l \in \Gamma_j} R i_{\mC_{jl}}$ and setting $C_{jl} = \frac{C}{R_{jl}}$ for $(j,l) \in A$ with some constant $C>0$, we derive the following method  
\BEA \label{e:pg-extra-par-c} \nonumber
    x^{k+1} &=& \prox_{R f} \pr{ Wx^k - R \nabla h(x^k) - w^k - u^k } \\ 
    w^{k+1} &=& w^{k} + s (I-W) x^k \\
    u^{k+1} &=& \frac{C}{s} (I-W) (x^{k+1} - x^k). \nonumber
\EEA
We now consider the decentralized quadratic programming from \cite{ShiLingWuYin2015_proximal}. 
Suppose each agent $j\in\{1,\dots, N\}$ holds a symmetric positive semidefinite matrix $Q_j\in\reals^{m\times m}$, vectors $a_j, p_j \in \reals^m$ and scalars $b_j \in \reals$. Consider the minimization problem 
% \BEQ \label{e-pg-extra-par-c-ex-problem}
\[
\begin{array}{ll}
 \underset{x \in \reals^m}{\mbox{minimize}}&
    \frac{1}{N} \sum_{j =1}^{N} \left( x^\intercal Q_j x + p_j^\intercal x  \right), \\
    {\mbox{subject to}}& a_j^\intercal x \le b_j, \quad j=1, \dots, N.
\end{array}
\]
% \EEQ
Set $f_j (x) = \delta_{\{ z \mid a_j^\intercal z \le b_j\}} (x)$ and $h_j(x) = x^\intercal Q_j x + p_j^\intercal x $, where 
\[
    \delta_{\{ z \mid a_j^\intercal z \le b_j\}}(x) = \begin{cases}
        0 & \text{if } a_j^\intercal x \le b_j \\
        \infty & \text{otherwise } 
    \end{cases}
\]
is the indicator function. 
Then the given optimization problem %\eqref{e-pg-extra-par-c-ex-problem} 
recasts to
\[
\begin{array}{ll}
 \underset{x \in \reals^m}{\mbox{minimize}}&
    \frac{1}{N} \sum_{j =1}^{N} \left( f_j(x) + h_j(x)  \right).
\end{array}
\]
Since the minimization subproblem has an explicit solution
\[
    \prox_{R f_j} (z) = \begin{cases}
        z & \text{if } a_j^\intercal z \le b_j \\
        z + \frac{ b_j-a_j^\intercal z  }{\| a_j \|_2^2} a_j & \text{otherwise},
    \end{cases}
\]
this problem can be solved by PG-EXTRA. 

We follow the same setting of \cite{ShiLingWuYin2015_proximal}. 
We set $m=50$. 
Each $Q_j$ is generated by taking the product of $\tilde{Q}_j$ and its transpose, where $ \tilde{Q}_j \in \reals^{m \times m}$ is a matrix with elements that follow an i.i.d. Gaussian distribution. 
Each $p_j$ is generated to follow an i.i.d. Gaussian distribution. 
Vectors $a_j$ and $b_j$ are also randomly generated, however, we conducted the experiment for the case that the solution of the constrained problem differs from that of the unconstrained problem.
We use Metropolis mixing matrix as in \S\ref{sec-dadmm-c}. 

The numerical results are illustrated in Figure~\ref{fig:pg_extra_par_c_n20_appndx}. 
We compare PG-EXTRA and the variant method \eqref{e:pg-extra-par-c} obtained from the modified circuit with additional parallel capacitors.  
We use $R=0.05$, $R=0.07$ for PG-EXTRA and $R=0.07$, $C=0.3$ and $s=0.8$ for \eqref{e:pg-extra-par-c}. 
The parameters for PG-EXTRA were obtained through a grid search. 
The parameters for \eqref{e:pg-extra-par-c} are hand-optimized starting from $C=0$ and $s=0.5$, the parameter selection that makes \eqref{e:pg-extra-par-c} to coincide with PG-EXTRA when $u^0=0$. 
The relative error for \eqref{e:pg-extra-par-c} decreases to $10^{-8}$ in 
$147$ iterations, while for PG-EXTRA with $R=0.05$ in $214$ iterations.
\begin{figure} [H]
    \centering \vspace{-4mm}
    \begin{subfigure}[c]{0.44\textwidth}
        \centering
        \includegraphics[width=\textwidth]{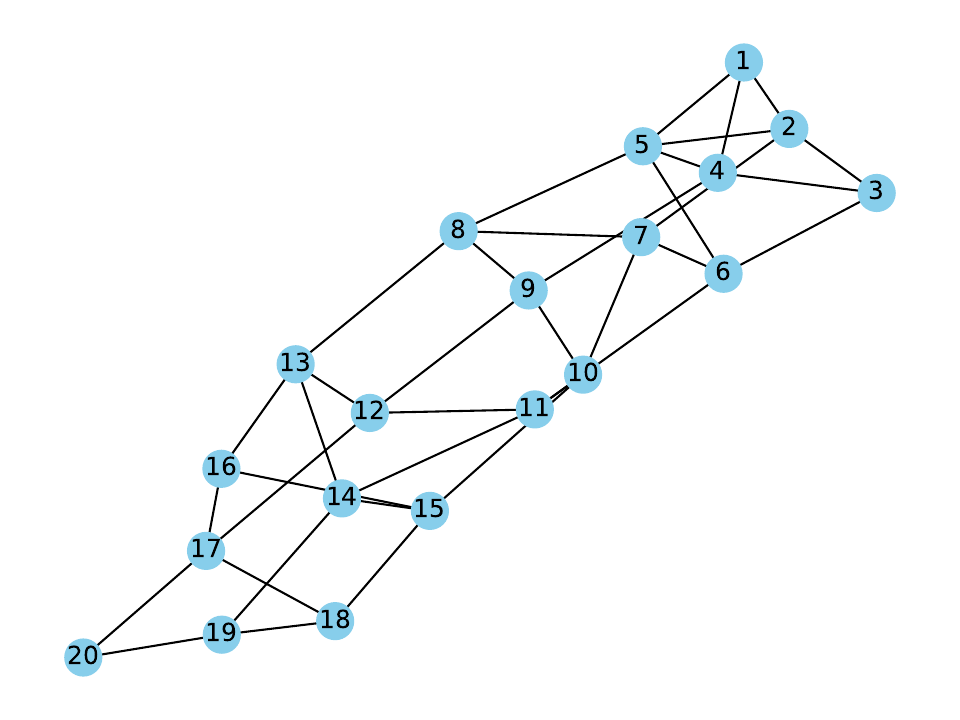}
    \end{subfigure}
    \hfill
    \begin{subfigure}[c]{0.55\textwidth}
        \centering
        \includegraphics[width=\textwidth]{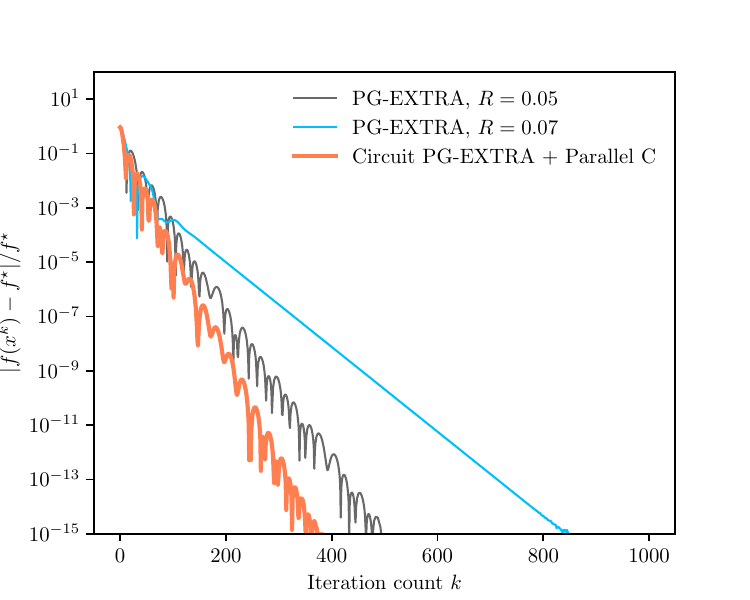}
    \end{subfigure}
    \caption{(Left) Underlying graph $G$. (Right) Relative error $\left|f(x^k) - f^\star\right|/{f^\star}$ vs. $k$.
    }
    \label{fig:pg_extra_par_c_n20_appndx}
\end{figure}

\end{document}